\documentclass[a4,11pt]{paper}
\usepackage{enumitem,color,amssymb,bm,amsmath,stmaryrd,tikz-cd,mathtools}
\usepackage[english]{babel}  
\usepackage{graphicx}
\usepackage{lipsum}
\usepackage{xcolor}
\usepackage{booktabs}
\usepackage{scalerel}
\usepackage{ntheorem}
{\theoremstyle{empty}
}
\usepackage[top=1in, left=1in, right=1in, bottom=1in]{geometry}
\newcommand{\R}{\mathbb{R}}
\newcommand{\Z}{\mathbb{Z}}

\subsectionfont{\large\sf\bfseries\color{black!50!blue}} 
\sectionfont{\Large\sf\bfseries\color{black!50!blue}} 
\titlefont{\LARGE\sf\bfseries\color{black!50!blue}} 
\numberwithin{equation}{section}
\newcommand{\inclj}{\mathsf{j}}
\newcommand{\llb}{\llbracket}
\newcommand{\rrb}{\rrbracket}
\newcommand{\Ringm}{{\Fbb[\wvar,\zvar]}}
\newcommand{\Ringh}{{\Fbb[\zvar]}}
\newcommand{\Ringhh}{\Fbb}
\newcommand{\Rings}{\Abb^\star}
\newcommand{\CFL}{\mathsf{CFL}}
\newcommand{\CFLm}{\CFL^-}
\newcommand{\CFLh}{\CFL^\circ}
\newcommand{\CFLhh}{\widehat{\CFL}}
\newcommand{\CFLs}{\CFL^\star}
\newcommand{\HFL}{\mathsf{HFL}}
\newcommand{\HFLm}{\HFL^-}
\newcommand{\HFLh}{\HFL^\circ}
\newcommand{\HFLhh}{\widehat{\HFL}}
\newcommand{\HFLs}{\HFL^\star}

\newcommand{\HFsTorsion}{\HFKs_{tor}}
\newcommand{\HFsTF}{\HFKs_{t.f.}}

\newcommand{\HFhTF}{\HFKh_{t.f.}}
\newcommand{\HFmTorsion}{\HFKm_{tor}}
\newcommand{\HFmTF}{\HFKm_{t.f.}}

\newcommand{\wedgewhh}{\widehat{\mywedge}_\wpoint}

\newcommand{\wedgezhh}{\widehat{\mywedge}_\zpoint}
\newcommand{\wedgews}{\mywedge^\star_\wpoint}
\newcommand{\wedgezs}{\mywedge^\star_\zpoint}
\newcommand{\HFKm}{\mathsf{HF}^-}
\newcommand{\HFKmw}{\mathsf{HF}^-_w}
\newcommand{\HFKh}{\mathsf{HF}^\circ}
\newcommand{\HFKhw}{\mathsf{HF}^\circ_w}
\newcommand{\HFKhh}{\widehat{\mathsf{HF}}}
\newcommand{\HFKs}{\mathsf{HF}^\star}
\newcommand{\HFKsw}{\mathsf{HF}^\star_w}

\newcommand{\tb}{\mathbf{t}}
\newcommand{\zpt}{z}
\newcommand{\wpt}{w}
\newcommand{\grw}{\mathrm{gr}_\wpoint}
\newcommand{\grz}{\mathrm{gr}_\zpoint}

\newcommand{\HFHAhat}{\Ringh\text{-}\mathsf{Modules}}

\newcommand{\HFHsn}{\mywedge^\star_n\text{-}\mathsf{Modules}}
\newcommand{\HFHhn}{\mywedge^\circ_n\text{-}\mathsf{Modules}}
\newcommand{\HFHstar}{\Abb^\star\text{-}\mathsf{Modules}}
\newcommand{\forget}{\mathsf{f}}
\newcommand{\incl}{\mathsf{i}}

\newcommand{\Path}{\mathsf{A}}
\newcommand{\Links}{\mathsf{Links}}
\newcommand{\LinksC}{{\mathsf{Links}}^{c}}

\newcommand{\nPLinks}{n\text{-}\mathsf{PointedLinks}}
\newcommand{\nPLinksC}{n\text{-}{\mathsf{PointedLinks}}^{c}}

\newcommand{\oZ}{{\overline{\zpoint}}}
\newcommand{\oW}{{\overline{\wpoint}}}
\newcommand{\zpoint}{{\mathsf{z}}}
\newcommand{\wpoint}{{\mathsf{w}}}
\newcommand{\Cell}{{\overline{\mathcal{C}}}}
\newcommand{\Deco}{\mathcal{C}}

\newcommand{\Tria}{\overline{\mathcal{T}}}
\newcommand{\Divide}{\mathsf{A}}
\newcommand{\qed}{\hfill \ensuremath{\Box}}
\newcommand{\SCob}{\mathcal{S}}
\newcommand{\HAction}{\mathsf{A}}

\newcommand{\Triangle}{\triangle}

\newcommand{\HSurf}{\Sigma}
\newcommand{\Surface}{{S}}
\newcommand{\bSurface}{{\overline{\Surface}}}
\newcommand{\asf}{{\mathsf{a}_{\zpoint,\wpoint}}}

\newcommand\mywedge{\scalerel*{\bigwedge}{j}}

\newcommand{\Triangles}{\mathsf{Tri}}

\newcommand{\ppoint}{\mathsf{p}}

\newcommand{\wvar}{\mathsf{u}}
\newcommand{\zvar}{\mathsf{v}}

\newcommand{\afrak}{\mathfrak{a}}
\newcommand{\SpinC}{{\mathrm{Spin}^c}}

\newcommand{\oalphas}{{\overline{\alphas}}}

\newcommand{\oS}{\overline{S}}

\newcommand{\oH}{\overline{H}}

\newcommand{\Link}{\mathcal{L}}
\newcommand{\TSurface}{\mathcal{S}}

\newcommand{\ra}{\rightarrow}

\newcommand{\fmap}{\mathfrak{f}}
\newcommand{\gmap}{\mathfrak{g}}

\newcommand{\tmap}{\mathfrak{t}}

\newcommand{\colvec}[2][.8]{%
  \scalebox{#1}{%
    \renewcommand{\arraystretch}{.8}%
    $\begin{pmatrix}#2\end{pmatrix}$%
  }
}

\newcommand{\Sbf}{\mathbf{S}}

\newcommand{\cbf}{\mathbf{c}}

\newcommand{\Abb}{\mathbb{A}}

\newcommand{\Ahat}{{\widehat{\Abb}}}

\newcommand{\Fbb}{\mathbb{F}}

\newcommand{\Dcal}{\mathcal{D}}

\newcommand{\Mcal}{\mathcal{M}}

\newcommand\PD{\mathrm{PD}}

\newcommand{\Ker}{\mathrm{Ker}}

\newcommand{\Image}{\mathrm{Im}}

\newcommand\Sym{\mathrm{Sym}}

\newtheorem{thm}{\bfseries\color{black!50!blue} Theorem}[section]
\newtheorem{prop}[thm]{\bfseries\color{black!50!blue} Proposition}

\newtheorem{lem}[thm]{\bfseries\color{black!50!blue} Lemma}

\newtheorem{rmk}[thm]{\bfseries\color{black!50!blue} Remark}

\newtheorem{defn}[thm]{\bfseries\color{black!50!blue} Definition}
\newtheorem{ex}[thm]{\bfseries\color{black!50!blue} Example}

\def\endproof{\relax\ifmmode\expandafter\endproofmath\else
  \unskip\nobreak\hfil\penalty50\hskip.75em\hbox{}\nobreak\hfil\bull
  {\parfillskip=0pt \finalhyphendemerits=0 \bigbreak}\fi}
\def\endproofmath$${\eqno\bull$$\bigbreak}
\def\bull{\vbox{\hrule\hbox{\vrule\kern3pt\vbox{\kern6pt}\kern3pt
\vrule}\hrule}}

\newcommand\x{\mathbf x}
\newcommand\y{\mathbf y}

\newcommand\z{\mathbf z}

\newcommand{\spinc}{\mathfrak s}
\newcommand{\spinct}{\mathfrak t}

\newcommand\alphas{\mbox{\boldmath$\alpha$}}

\tikzset{commutative diagrams/.cd,
	mysymbol/.style={start anchor=center,end anchor=center,draw=none}
}

\begin{document}
\title{A Heegaard-Floer TQFT for link cobordisms}%
\author{Eaman Eftekhary}%
\institution{\sf{School of Mathematics, Institute for Research in Fundamental Sciences (IPM), Tehran, Iran}}
\maketitle
\begin{abstract}
We introduce a Heegaard-Floer homology functor from the category of oriented links in closed $3$-manifolds and oriented surface cobordisms in $4$-manifolds connecting them 
to the category of  $\Fbb[\zvar]$-modules and   $\Fbb[\zvar]$-homomorphisms  between them, where $\Fbb$ is the field with two elements. In comparison with previously defined TQFTs for decorated links and link cobordisms, the construction of this paper has the advantage of being independent from the decoration. Some of the basic properties of this functor are also explored.  
\end{abstract}
\tableofcontents
\section{Introduction}
Ozsv\'ath and Szab\'o constructed powerful invariants of closed $3$-manifolds  in \cite{OS-3m1}. Given a pointed closed $3$-manifold $Y$, these invariants are constructed as the isomorphism classes of four  graded $\Z[\wvar]$-modules which are decomposed according to $\SpinC$ structures  on $Y$. We focus on the theory with coefficients  in the field $\Fbb=\Z/2$ with two elements, where the {\emph{naturality}} of the construction may be addressed.  The invariants of Ozsv\'ath and Szab\'o, which are called the Heegaard-Floer groups of $Y$ are then  given as four $\Fbb[\wvar]$-modules, denoted by 
\begin{align*} 
\HFKs(Y)=\bigoplus_{\spinc\in\SpinC(Y)}\HFKs(Y,\spinc),\quad\text{for}\ \star\in\{-,+,\infty,\wedge\}.
\end{align*}
 Associated with a $4$-dimensional cobordism $X$ from a pointed closed $3$-manifold $Y$ to another pointed closed $3$-manifold $Y'$, which is decorated by an arc connecting the marked point of $Y$ to the marked point of $Y'$,
the $\Fbb[\wvar]$-homomorphisms 
\begin{align*} 
\HFKs(X,\spinct)=\fmap_{X,\spinct}^{\star}: \HFKs(Y,\spinct|_{Y})\ra  \HFKs(Y',\spinct|_{Y'}),\quad\spinct\in\SpinC(X)
\end{align*}
are also constructed. Juh\'asz in \cite{Juhasz-naturality} (and Juh\'asz, Thurston and Zemke in \cite{JTZ}) proved that  $ \HFKs(Y)$ may be constructed as a concrete $\Fbb[\wvar]$-module associated to the pointed closed $3$-manifold $Y$ (rather than the isomorphism class of such a module). For the hat theory, the sum over all $\SpinC$ structures is well-defined and gives the homomorphism 
\begin{align*}
\HFKhh(X)=\widehat{\fmap}_{X}: \HFKhh(Y)\ra  \HFKhh(Y').
\end{align*}
Therefore, $ \HFKhh$ gives a functor from the category of pointed closed $3$-manifolds  and decorated cobordisms between them to the category of  
$\Fbb$-modules and homomorphisms between them.\\

The theory was subsequently developed to construct invariants for pointed knots \cite{OS-knot, Ras, Ef-LFH},  pointed links \cite{OS-link} and  balanced sutured manifolds \cite{Juh,AE-1}.   The approach of   \cite{Juhasz-naturality} (also \cite{JTZ}) implies  that the outcomes for pointed links and sutured manifolds are  functors between appropriate categories, as explored in \cite{Juhasz-naturality,JTZ,AE-2,Zemke-1}. We would like to reconsider the case of oriented links  and the cobordisms between them:

\begin{defn}\label{defn:links}
A {\emph{link}} $(Y,L)$ consists of a closed $3$-manifold $Y$ and  an oriented closed $1$-submanifold $L$ of $Y$ intersecting every connected component of $Y$.
A {\emph{link cobordism}} 
\begin{align*}
(X,\Surface):(Y,L)\ra (Y',L')
\end{align*}
 from  the link $(Y,L)$ to the link $(Y',L')$ consists of a $4$-manifold $X$ with boundary $-Y\amalg Y'$ and a properly embedded compact and oriented surface $\Surface$   in $X$ with boundary $-L\amalg L'$, so that the intersection of every connected component of $X$ (respectively, $\Surface$) with either of $Y$ and $Y'$ (respectively, $L$ and $L'$) is non-empty. A link cobordism $(X,\Surface)$ is called a {\emph{link concordance}} if every connected component of $\Surface$ is homeomorphic to a cylinder. The category of links and link cobordisms  (respectively, link concordances) between them is denoted by $\Links$ (respectively, by $\LinksC$). 
\end{defn}	 

The hat theory version of our main result in this paper may be stated as the following theorem:

\begin{thm}\label{thm:main-intro}
There are  Heegaard-Floer functors, denoted by 
\begin{align*}
\HFKh:\LinksC\ra \HFHAhat\quad\text{and}\quad\HFKhw:\Links\ra \HFHAhat,
\end{align*}		
which assigns the bi-graded $\Ringh$-modules $\HFKh(Y,L)$ and  $\HFKhw(Y,L)\subset \HFKh(Y,L)$ to every link $(Y,L)$ (respectively). These modules decompose according to $\SpinC$ structures on $Y$. Associated with every  link cobordism $(X,\Surface):(Y,L)\ra (Y',L')$, $\HFKhw$ assigns a  $\Ringh$-homomorphism 
\begin{align*}
\HFKhw(X,\Surface)=\gmap_{X,\Surface}^\circ:\HFKhw(Y,L)\ra \HFKhw(Y',L').	
\end{align*}	 
If $(X,\Surface)$ is a concordance,
$\HFKh$ assigns a  $\Ringh$-homomorphism 
\begin{align*}
	\HFKh(X,\Surface):\HFKh(Y,L)\ra \HFKh(Y',L'),	
\end{align*}	 
to it which restricts to $\HFKhw(X,\Surface)$ on the submodule $\HFKhw(Y,K)$  of $\HFKs(Y,K)$, and may thus be denoted by $\gmap_{X,\Surface}^\circ$ without confusion. 
\end{thm}	

 Pointed links and decorated cobordisms between them are already considered by Alishahi and the author \cite{AE-1,AE-2} and independently by Zemke \cite{Zemke-1}.  Theorem~\ref{thm:main-intro} gives a weaker version of Heegaard Floer  groups and cobordism maps, which is instead {\emph{not sensitive}} to the decoration. To describe the difference more precisely, we need to review/set up some notation.\\ 
 
  We remind the reader that a  {\emph{pointed link}} $(Y,L,\ppoint)$ is an oriented link $(Y,L)$ together with a  finite collection $\ppoint$ of points on $L$ intersecting every connected component of $L$. A {\emph{decorated link cobordism}} from  a pointed link $(Y,L,\ppoint)$ to the pointed link $(Y',L',\ppoint')$ is a triple $(X,\Surface,\Path)$, where $(X,\Surface)$ is a link cobordism and $\Path$ is a properly embedded $1$-manifold without closed components on $\Surface$ connecting $\ppoint$ to $\ppoint'$. In particular, the number of markings in either of $\ppoint$ and $\ppoint'$ is the same positive integer $n$. We may thus assume that $\ppoint$ and $\ppoint'$ are maps from a fixed set $[n]=\{1,\ldots,n\}$ to the links $L$ and $L'$, respectively.  We then write 
  \begin{align*}
  (X,\Surface,\Path):(Y,L,\ppoint)\ra (Y',L',\ppoint').
  \end{align*}
Having fixed $n$, the category of $n$-pointed links and decorated link cobordisms between them is denoted by  $\nPLinks$.\\

The notions of decorated links and cobordisms between them which is used here is  less general in comparison with the notion used in \cite{Zemke-1}. Each point in $\ppoint$ determines a pair of basepoints using the orientation of $L$, one in $\wpoint$ and the other one in $\zpoint$. The decoration gives a decomposition of the surface to a union of strips  $\Surface_\zpoint$ and a subsurface $\Surface_\wpoint$ which is isotopic to $\Surface-\Path$. These two subsurfaces are separated by a dividing set, which may be identified as the union of two parallel copies of $\Path$. This gives a decorated link cobordism  in the sense of \cite{Zemke-1}. The aforementioned loss of generality (in comparison with the setup in \cite{Zemke-1}) is chosen in favor of a more smooth exposition, as the goal is to get rid of the decoration anyway.   \\

Let  $\Abb^-=\Ringm $ denote the polynomial ring with two variables  over $\Fbb$, define 
\begin{align*}
\Abb^\circ=\Ringm/\zvar=\Fbb[\wvar]\quad\text{and}\quad\Ahat=\Ringm/\langle\wvar,\zvar\rangle=\Fbb
\end{align*}
 and equip them with the structure of  $\Ringm$-module as quotients of $\Ringm$.   For $\star=-,\circ,\wedge$, let $\HFHstar$ denote the category of 
$\Abb^\star$-modules and 
$\Abb^\star$-homomorphisms between them.  
Let $\mywedge^\star_n$ denote the exterior algebra over $\Abb^\star$ generated by $[n]$. Denote the category of  $\mywedge^\star_n$-modules and 
$\mywedge^\star_n$-homomorphisms between them  by $\HFHsn$.  \\

Associated with an $n$-pointed link $(Y,L,\ppoint)$ and a $\SpinC$ structure $\spinc\in\SpinC(Y)$,
the constructions of \cite{AE-2} and \cite{Zemke-1} give the concrete 
$\mywedge^\star_n$-modules $\HFLs(Y,L,\ppoint,\spinc)$. We then set 
\begin{align*}
\HFLs(Y,L,\ppoint)=\bigoplus_{\spinc\in\SpinC(Y)}\HFLs(Y,L,\ppoint,\spinc),\quad\text{for}\ \ \star\in\{-,\circ,\wedge\}.
\end{align*}
Moreover, associated with a decorated link cobordism $	(X,\Surface,\Path)$ from $(Y,L,\ppoint)$ to $(Y',L',\ppoint')$
and $\spinct\in\SpinC(X)$, the aforementioned construction gives the concrete $\mywedge_n^\star$-homomorphism
\begin{align*}
	\HFLs(X,\Surface,\Path,\spinct)=\fmap_{X,\Surface,\Path,\spinct}^\star:
	\HFLs(Y,L,\ppoint,\spinct|_{Y})\ra 	\HFLs(Y',L',\ppoint',\spinct|_{Y'}).
\end{align*}	
For $\star=\circ,\wedge$, one may add all these $\mywedge_n^\star$-homomorphisms to obtain a well-defined  $\mywedge_n^\star$-homomorphism
\begin{align*}
	\HFLs(X,\Surface,\Path)=\fmap_{X,\Surface,\Path}^\star:
	\HFLs(Y,L,\ppoint)\ra 	\HFLs(Y',L',\ppoint').
\end{align*}	
This gives the Heegaard-Floer functors
 \begin{align*}
 	\HFLs_n:\nPLinks\ra \HFHsn,\quad\text{for}\ \ \star\in\{\circ,\wedge\},\ n\in\Z^+.
 \end{align*}		

 The action of $\mywedge^\star_n$ on  $\HFLs(Y,L,\ppoint)$ comes from the markings  in $\ppoint$ (more precisely, from the basepoints  in $\wpoint$). For  $w\in\wpoint$ denote the corresponding basepoint action map by $\Phi^\star_w$. Set
\begin{align*}
	\HFKs(Y,L;\ppoint)=\left\{\x\in\HFLs(Y,\ppoint)\ \big|\ \Phi^\star_w(\x)=0,\ \ \forall\ w\in\wpoint\right\}.
\end{align*}
If $L$ is a knot and $|\ppoint|=1$, then $\HFKm(Y,L;\ppoint)=\HFLm(Y,L,\ppoint)$ (see Proposition~\ref{prop:HF-of-knots}) while the groups $\HFKs(Y,L;\ppoint)$ and $\HFLs(Y,L,\ppoint)$ are typically different for  $\star=\circ,\wedge$.  For a decorated link cobordism $(X,\Surface,\Path)$ and $\spinct\in\SpinC(X)$ as above,
the $\mywedge_n^\star$-homomorphism $\fmap_{X,\Surface,\Path,\spinct}^\star$  induces a concrete $\Abb^\star$-homomorphism
\begin{align*}
\HFKs(X,\Surface,\spinct;\Path)=\gmap_{X,\Surface,\spinct;\Path}^\star:
\HFKs(Y,L,\spinct|_{Y};\ppoint)\ra 	\HFKs(Y',L',\spinct|_{Y'};\ppoint').	
\end{align*}	

We may formulate a more detailed and general form of Theorem~\ref{thm:main-intro} as the following theorem. 

\begin{thm}\label{thm:main-intro-2}	
For $\star=-,\circ,\wedge$, the $\Abb^\star$-module $\HFKs(Y,L;\ppoint)$ associated with the pointed link $(Y,L,\ppoint)$ is independent of $\ppoint$ and may be constructed as a concrete $\Abb^\star$-module 
\begin{align*}
\HFKs(Y,L)=\bigoplus_{\spinc\in\SpinC(Y)}\HFKs(Y,L,\spinc)
\end{align*}
 associated with the link $(Y,L)$. Moreover, given a link cobordism $(X,\Surface):(Y,L)\ra (Y',L')$ and  two decorations $\Path$ and $\Path'$ of $\Surface$ and $\spinct\in\SpinC(X)$,
the maps
\begin{align*}
\gmap_{X,\Surface,\spinct;\Path}^\star,\gmap_{X,\Surface,\spinct;\Path'}^\star
:\HFKs(Y,L,\spinct|_{Y})\ra 	\HFKs(Y',L',\spinct|_{Y'})
\end{align*}	
are identical on $\HFKsw(Y,L,\spinct|_{Y})=\HFKs(Y,L,\spinct|_{Y})\cap (\zvar\HFLs(Y,L,\ppoint,\spinct|_{Y}))$. In particular, for $\star=-,\circ$ we obtain a well-defined $\Abb^\star$-homomorphism
\begin{align*}
	\gmap_{X,\Surface,\spinct}^\star
	:\HFKsw(Y,L,\spinct|_{Y})
	\ra \HFKsw(Y',L',\spinct|_{Y'}).
\end{align*}	
Moreover, if $\Surface$ is a union of cylinders, $\gmap_{X,\Surface,\spinct;\Path}^\star=\gmap_{X,\Surface,\spinct;\Path'}^\star$ and we obtain a well-defined map
\begin{align*}
	\gmap_{X,\Surface,\spinct}^\star:\HFKs(Y,L,\spinct|_{Y})\ra \HFKs(Y',L',\spinct|_{Y'})\quad\text{for}\ \ \star\in\{-,\circ,\wedge\}.
\end{align*}	
\end{thm}	

The equality of $\gmap_{X,\Surface,\spinct;\Path}^\star$ and $\gmap_{X,\Surface,\spinct;\Path'}^\star$ on $\HFKsw(Y,L,\spinct|_{Y})$ is very interesting, since they do not seem to be the same map on $\HFKs(Y,L,\spinct|_{Y})$. Nevertheless, the author does not have examples of decorations $\Path$ and $\Path'$ of the same link cobordism $(X,\Surface):(Y,L)\ra (Y',L')$ such that $\gmap_{X,\Surface,\spinct;\Path}^\star\neq \gmap_{X,\Surface,\spinct;\Path'}^\star$ as maps from $\HFKs(Y,L,\spinct|_{Y})$ to	$\HFKs(Y',L',\spinct|_{Y'})$.\\

In the case of hat theory, we may  compare the decorated link cobordism TQFT with the link cobordism TQFT of this paper.  For this purpose, note that taking the kernel of the action of $\mywedge^\star_n$ gives a functor
\begin{align*}
\incl^\star_n:\HFHsn\ra \HFHstar.	
\end{align*}	
Denote by $\inclj^\star_n:\HFHsn\ra \HFHstar$ the functor which assigns the $\Rings$-module $\incl^\star_n(H)\cap (\zvar H)$ to a $\mywedge^\star_n$-module $H$.
There are also forgetful functors (which forget the markings and decorations on links and  cobordisms/concordances), and are denoted by
\begin{align*}
\forget_n:\nPLinks\ra \Links	\quad\text{and}\quad    \forget^c_n:\nPLinksC\ra \LinksC.
\end{align*}	
For $\star=\circ,\wedge$, let us denote the restriction of $\HFLs_n$ to the subcategory $\nPLinksC$ by $\HFLs_n$ as well. Theorem~\ref{thm:main-intro-2} then gives the following conclusion.
\begin{thm}\label{thm:main-intro-3}
With the above notation in place, the following diagrams are commutative
	\begin{center}
			\begin{tikzcd}[ row sep=large, column sep=large, execute at end picture={\draw[->] (-0.3,0) arc[start angle=-180,delta angle=300,radius=3mm];}]
			\nPLinks\arrow[r,"\HFLh_n"]
			\arrow[d,"\forget_n"']
			&\HFHhn\arrow[d,"\inclj^\circ_n"]	\\
			\Links\arrow[r,"\HFKhw"']&
			\HFHAhat
		\end{tikzcd}
\quad\text{and}\quad		
			\begin{tikzcd}[ row sep=large, column sep=large, execute at end picture={\draw[->] (-0.3,0) arc[start angle=-180,delta angle=300,radius=3mm];}]
					\nPLinksC\arrow[r,"\HFLs_n"]
					\arrow[d,"\forget^c_n"']
					&\HFHsn\arrow[d,"\incl^\star_n"]	\\
					\LinksC\arrow[r,"\HFKs"']&
					\HFHstar
				\end{tikzcd},
		\end{center}	
where $\star=\circ,\wedge$. 
In other words, we have $\inclj^\circ_n\circ\HFLh_n=\HFKhw\circ\forget_n$ and $\incl^\star_n\circ\HFLs_n=\HFKs\circ\forget^c_n$.	
\end{thm}

Decorated link cobordisms and the corresponding Heegaard-Floer maps are studied extensively in the literature in the past couple of years. Some of the results may be translated to conclusions about the (undecorated) link cobordism maps.

\begin{defn}\label{defn:compression-disk} 
	Let $(X,\Surface):(Y,L)\ra (Y',L')$ denote a link  cobordism. A  properly embedded disk $D$ in $X-\Surface$ is called a {\emph{compression disk}} for $(X,\Surface)$ if  $\partial D\subset \Surface$ is homotopically non-trivial, while there is a closed subset $W$ of $X$ which is diffeomorphic to $D\times [-1,1]^2$, so that $D$ is identified as $D\times\{(0,0)\}$, while $W\cap \Surface=\partial D\times [-1,1]\times \{0\}$  is a neighborhood of $\partial D$ in $\Surface$.
	If $D$ is a compression disk for $(X,\Surface)$ as above, the link cobordism $(X,\Surface_D):(Y,L)\ra (Y',L')$, with
	\begin{align*}
		\Surface_D=\big(\Surface-(\partial D\times [-1,1]\times\{0\})\big)\cup\big(
		D\times\{-1,1\}\times\{0\}\big),	
	\end{align*}	 
	is called the {\emph{compression}} of $(X,\Surface)$ along $D$. We then write $(X,\Surface)\leadsto_D(X,\Surface_D)$.
\end{defn} 
Compression along a disk increases the Euler characteristic of a surface, without changing the corresponding (relative) homology group. In particular, genus minimizing surfaces may not be compressed. The following theorem provides an obstruction for the existence of compression disks.
\begin{thm}\label{them:compression-along-disk-intro}
	Suppose that the link cobordism $(X,\Surface_D):(Y,L)\ra (Y',L')$ is obtained from the link cobordism $(X,\Surface):(Y,L)\ra (Y',L')$ by compressing along a compression disk $D$, as above. Then $	\gmap_{X,\Surface}^\star=\zvar\cdot\gmap_{X,\Surface_D}^\star$ for $\star=-,\circ$.	
\end{thm}

The construction of link cobordism invariants which are not sensitive to the decoration automatically gives invariants of smooth slice disks and slice surfaces.  
\begin{defn}
Given a slice knot $(S^3,K)$,  a slice disk $(B^4,D)$ for $K$ and $\star\in\{-,\circ,\wedge\}$, the Heegaard-Floer class $\tmap_{D,K}^\star\in\HFKs(S^3,K)$ is defined equal to $\gmap^\star_{X,\Surface}(1)$, where $(X,\Surface)$ is a link cobordism from the unknot $(S^3,U)$ to $(S^3,K)$ which is obtained from $(B^4,D)$ by removing a small ball  neighborhood of the center of the disk $D$. Similarly, for a slice surface $(B^4,D)$ for an arbitrary knot $(S^3,K)$ and $\star\in\{-,\circ\}$, the Heegaard-Floer class $\tmap_{D,K}^\star\in\HFKsw(S^3,K)$ is defined equal to $\gmap^\star_{X,\Surface}(\zvar)$, where as before, $(X,\Surface):(S^3,U)\ra (S^3,K)$ is obtained from $(B^4,D)$ by removing a small ball  neighborhood of a point in the interior of $D$ and $\zvar$ denotes the generator of $\HFKsw(S^3,U)$. 	
\end{defn}	

Given a slice disk $D$ for a slice knot $(S^3,K)$, a closely related invariant $\tmap_{D,K,\ppoint}\in\HFLhh(K,\ppoint)$ is defined in \cite{JM-Concordance}, where $\ppoint$ is a marked point on $K$. The construction is used in \cite{JZ-slice-disks} to distinguish non-isotopic slice disks for the same knot from one another. The dependence of $\HFLhh(K,\ppoint)$ on the decoration forces a detour for actual applications of the invariants of \cite{JM-Concordance}. Our construction implies that $\tmap_{D,K}^\star$ may be constructed as a concrete class in a concrete group (i.e. in $\HFKs(S^3,K)$) and may thus be used directly to distinguish between non-isotopic slice disks for the same knot. In particular, we give examples of slice disks for $K\#-K$, where $K$ is the figure-eight knot, with different invariants in Example~\ref{ex:slice-disks-for-F8}. Several computations and some of the basic properties of the invariants are explored in Section~\ref{sec:properties}. For other interesting examples and computations, the reader is referred  to \cite{Mahkam}. \\ 

{\color{black!50!blue}\textbf{Acknowledgements.}}
The author would like to thank Ian Zemke, David Gay and Roohallah Mahkam for very helpful discussions.
\section{Background on pointed trisection diagrams}\label{sec:trisection-diagrams}
\subsection{Trisection diagrams for decorated link cobordisms}\label{subsec:trisection-diagrams}
A {\emph{cut system}} $\alphas$ on a smooth closed connected surface $\HSurf$ of genus $g$ is a collection of $m\geq g$ disjoint simple closed curves on $\HSurf$ such that $\HSurf-\alphas$ is a (possibly disconnected) genus-$0$ surface. Given a set $\wpoint$ of basepoints on $\HSurf$, two cut systems are {\emph{slide-equivalent}} away from $\wpoint$ if they are related by isotopies and handle-slides away from $\wpoint$. The tuples $(\HSurf,\alphas_1,\ldots,\alphas_k,\wpoint)$ and $(\HSurf',\alphas'_1,\ldots,\alphas'_k,\wpoint')$ (where each $\alphas_i$ is a cut system on $\Sigma$ and each $\alphas'_i$ is a cut system on $\Sigma'$) are {\emph{slide-diffeomorphic}} if there is a diffeomorphism $f:(\HSurf,\wpoint)\ra (\HSurf',\wpoint')$ (mapping $\wpoint$ to $\wpoint'$) such that $f(\alphas_i)$ is slide-equivalent to $\alphas'_i$ away from $\wpoint'$ for every $i\in\Z/k$. \\

When $k=2$, the tuple $H=(\HSurf,\alphas_1,\alphas_2,\wpoint)$ is a pointed Heegaard diagram which determines a sutured manifold. In particular, if the set of basepoints is a union $\wpoint\amalg\zpoint$, where each one of $\wpoint$ and $\zpoint$ includes precisely one basepoint in each connected component of $\HSurf-\alphas_j$ for $j=1,2$, the pointed Heegaard diagram $H=(\HSurf,\alphas_1,\alphas_2,\wpoint,\zpoint)$ determines an oriented pointed link $\Link_H$ in a closed $3$-manifold $Y_H$, which consist of an oriented link $L_H$ together with a finite collections of  marked points on the link. Associated with every marked point on the oriented link $L_H$, we may consider a small arc on $L_H$ containing it with an end-point $w\in\wpoint$ and an end-point  $z\in\zpoint$, so that the arc is oriented from $w$ to $z$. Therefore, we may assume that $\wpoint,\zpoint\subset L_H$, and that the points in $\wpoint$ and $\zpoint$   alternate on each component  of $L_H$. We further assume that  each connected component of $L_H$ has non-empty intersection with either of $\wpoint$ and $\zpoint$.

\begin{defn}\label{defn:trisection-diagram}
	A {\emph{trisection diagram}} of type $(g;n;\cbf)$, with $\cbf=(c_0,c_1,c_2)\in (\Z^+)^3$, is a diagram $H=(\HSurf,\alphas_0,\alphas_1,\alphas_2,\wpoint,\zpoint)$ consisting of  a smooth closed surface $\HSurf$ of genus $g$, a  cut system $\alphas_i$  for every $i\in\Z/3$ and two sets $\wpoint,\zpoint\subset \HSurf-\alphas_0-\alphas_1-\alphas_2$ of basepoints, so that each one of the $n$  components of  $\HSurf-\alphas_i$ includes precisely one basepoint from $\wpoint$ and one from $\zpoint$. Further, the sub-diagram $H_i$ which is obtained from $H$ by removing $\alphas_i$,   is 	a Heegaard diagram for a $c_i$-component pointed link $\Link_i$ in the $3$-manifold $Y_i$. $H$ is called {\emph{nice}} if each $\Link_i$ is null-homologous in $Y_i$, $Y_0=\#^{k_0}(S^2\times S^1)$ for some non-negative integer $k_0$ and $\Link_0$ is a pointed unlink in $Y_0$.	The slide equivalence class of $H$ is denoted by $[H]$. 
\end{defn}

In illustrations of a trisection diagram $H=(\HSurf,\alphas_0,\alphas_1,\alphas_2,\wpoint,\zpoint)$,  we usually denote the curves in $\alphas_0,\alphas_1$ and $\alphas_2$ by red, black and blue curves, respectively. The basepoints in $\wpoint$ and $\zpoint$ are denoted by small filled circles and small empty circles, respectively.\\

Let  $H=(\HSurf,\alphas_0,\alphas_1,\alphas_2,\wpoint,\zpoint)$ be a nice trisection diagram  and $H_i$ denote the link diagram obtained from $H$ by removing $\alphas_i$ for $i\in\Z/3$. The diagram $H$ 
determines
a smooth closed oriented $4$-manifold cobordism $X=X_H$ with boundary $-Y_2\amalg Y_1$, where $Y_i=Y_{H_i}$. 
Moreover, $H$ determines   an oriented  surface $\Surface=\Surface_H\subset X$ with boundary $-L_2\amalg L_1$, where $L_i=L_{H_i}$ is the oriented  link in $Y_i$ underlying the pointed link $\Link_i=\Link_{H_i}$.  Let us briefly describe the construction, which also appears in \cite[Section 2.4]{AE-2}. Let $U_i$ denote the compression body determined by the surface $\HSurf$ and the cut system $\alphas_i$. The basepoints $\wpoint$ and $\zpoint$ determine a union of oriented arcs in $U_i$, which is denoted by $T_i$. We denote the end-point of $T_i$ on $\partial U_i=\HSurf$ by $\wpoint_i,\zpoint_i$ to distinguish them for different $i\in\Z/3$.  Let $\Triangle$ denotes a triangle with vertices $v_0,v_1,v_2$ and edges $e_0,e_1,e_2$ (in the clockwise order) so that $e_i$ connects $v_{i+1}$ to $v_{i-1}$ for every $i\in\Z/3$. Define 
\begin{align}\label{eq:triangulation-of-surface}
	X^\circ_H=\frac{\left(\HSurf\times \Triangle\right)\coprod\left(\coprod_{i\in\Z/3}U_i\times e_i\right)}{\Sigma\times e_i\sim \partial U_i\times e_i}\quad\text{and}\quad
\Surface^\circ_H=\frac{\left(
	(\wpoint\cup\zpoint)
	\times\Triangle\right)\coprod\left(\coprod_{i\in\Z/3}T_i\times e_i\right)}{
	(\wpoint\cup\zpoint)
	\times e_i\sim 
	(\wpoint_i\cup\zpoint_i)
	\times e_i}.
\end{align}

The boundary of $X^\circ_H$ may be identified as $\coprod_{i\in\Z/3}Y_i$.  Since $H$ is nice, the boundary component $Y_0=\#^{k_0}S^2\times S^1$ may be  filled with $\natural^{k_0}(B^3\times S^1)$ to give an oriented $4$-manifold $X_H$, which is a cobordism from $Y_2$ to $Y_1$. Note that filling $Y_0$ with $\natural^{k_0}(B^3\times S^1)$ is uniquely determined up to diffeomorphism. Moreover, we may attach $c_0$ disks to $\Surface^\circ_H$ along $L_0\subset Y_0$, or equivalently, retract each connected component of $L_0$ to a point in $Y_0$, to obtain the {\emph{link cobordism}} $(X_H,\Surface_H)$ from $(Y_2,L_2)$ to $(Y_1,L_1)$. We then write $(X_H,\Surface_H):(Y_2,L_2)\ra (Y_1,L_1)$. 
Note that the Euler characteristic of $\Surface_H$ is given by $\chi(\Surface_H)=c_0-n$.
\\

We denote by $\bSurface=\bSurface_H$ the quotient of $\Surface$ obtained by contracting each connected component of $L_1\cup L_2$ to a point (or equivalently, the surface obtained by attaching an abstract disk to each connected component of $L_1\cup L_2$). The construction of $\Surface^\circ_H$  from (\ref{eq:triangulation-of-surface}) provides (the isotopy class of) a {\emph{triangulation}} $\Tria_H$ of $\bSurface$. Every component $L^v$ of $L_i$ collapses to a vertex $v$ in $\bSurface$ with label $i\in\Z/3$ and such points form the vertices of $\Tria_H$. 
We may also retract each connected component of $T_i\times e_i$ (which is a rectangle) to the edge in $\partial^-T_i\times e_i$ (with label $e_i$) to obtain the aforementioned triangulation of $\bSurface$. Each triangle in $\Tria_H$ corresponds either to a basepoint in $\wpoint$ or to a basepoint in $\zpoint$, and may be colored gray or pink accordingly. If a basepoint $w\in\wpoint$ is in the same connected component of $\HSurf-\alphas_i$ as $z\in\zpoint$, the corresponding triangles have an edge with label $e_i$ in common. We denote the vertices labeled $0,1$ and $2$ (and the edges labeled $e_0,e_1$ and $e_2$) by red, black and  blue  colors, respectively. Therefore, the edges and their two endpoints have different colors.
  \\

If we remove the red vertices, and the black and blue edges from $\Tria_H$, a cell decomposition $\Cell_H$ of $\bSurface_H$ is obtained. The vertices of $\Cell_H$ are either black or blue, and the edges are all red, and every domain in the complement of the vertices and the edges is a polygon. $\Tria_H$ may be reconstructed from $\Cell_H$, by inserting a red vertex in the center of each polygon and connecting it to the black and blue vertices on the boundary of the polygon by blue and black edges, respectively. The cell decomposition $\Cell_H$ comes from a properly embedded $1$-manifold $\Deco_H$ on $\Surface_H$, whose components connect $L_2$ to $L_1$, and is determined by $H$. Let $\Surface_\zpoint$ denote a tubular neighborhood of $\Deco_H$ in $\Surface_H$ and $\Surface_\wpoint$ denote the closure of $\Surface-\Surface_\zpoint$. Then $\Surface_\wpoint$ and $\Surface_\zpoint$ intersect in a dividing set $\Divide_H$ and $(\Surface_H,\Divide_H)$ give a {\emph{decoration}} $\SCob_H$ of $\Surface_H$ and 
$(X_H,\SCob_H):(Y_2,\Link_2)\ra (Y_1,\Link_1)$ is a {\emph{decorated cobordism}}. These decorated cobordisms are special cases of the construction of \cite{Zemke-1}. The construction is also a special case of the general construction of Alishahi and the author in \cite{AE-2}. Note that $\Cell_H$ determines the isotopy class of the decoration $\SCob_H$  from $\Link_2$ to $\Link_1$ up to twists along $\Link_1$ and $\Link_2$.

\subsection{Stabilizations and perturbations of pointed trisection diagrams}
 The {\emph{stabilizations}} of  a  trisection diagram $H=(\HSurf,\alphas_0,\alphas_1,\alphas_2,\wpoint,\zpoint)$ are obtained as connected sums with one of the diagrams $H^j_s=(T,\alpha^j_0,\alpha^j_1,\alpha^j_2)$ of Figure~\ref{fig:i-stabilization}. 
To form  $H\#_{p}H^j_s$, choose a point \[p\in \HSurf-\alphas_0-\alphas_1-\alphas_2-\wpoint-\zpoint.\] 
We then remove small disk neighborhoods $D_{p}$ and $D_q$ of $p\in \HSurf$ and $q\in T$ from $\HSurf$ and $T$ respectively,  and connect the resulting circle boundaries  by a $1$-handle to obtain a new closed surface $\HSurf'$ of genus $g+1$. We set $\alphas'_i=\alphas_i\cup\{\alpha^j_i\}$ and define the $j$-stabilization of  $H$ at $p$ by 
\[H'=H\#_{p}H^j_s=(\HSurf',\alphas'_0,\alphas'_1,\alphas'_2,\wpoint,\zpoint).\] 
A $j$-stabilization $H'$ of $H$ preserves the boundary manifold $Y_i$ for $i\neq j$, while it changes the boundary manifold 
$Y_j$ to $Y_j\#(S^2\times S^1)$. Moreover,  the pointed links 
\[L'_i=L_{H'_i}\subset Y'_i=Y_{H'_i}=Y_i\quad\text{and}\quad  L'_j\subset Y'_j=Y_j\#(S^2\times S^1)\] 
are naturally identified with the links $L_i\subset Y_i$ and $L_j\subset Y_j\#(S^2\times S^1)$, respectively.  Further, the link cobordisms $\Surface_H$ and $\Surface_{H'}$ are identified and the decorated cobordism $\SCob_{H'}$ is diffeomorphic to $\SCob_H$ (as decorations of $\Surface_H=\Surface_{H'}$). 

\begin{figure}
	\def\svgwidth{0.95\textwidth}
	{\footnotesize{
			\begin{center}
				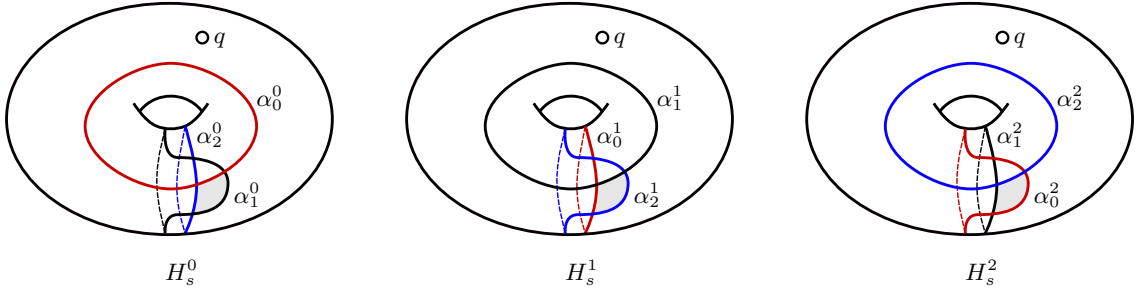
	\end{center}}}
	\caption{The $j$-stabilization is given as the connected sum with the diagram $H^j_s$ for $i\in\Z/3$.}
	\label{fig:i-stabilization}
\end{figure} 

\begin{defn}\label{defn:equivalence-of-HD}
	The nice trisection diagrams $H$ and $H'$ are called {\emph{equivalent}} if after suitable stabilizations of $H$ and $H'$, they are slide equivalent. The decorated cobordisms $(X,\SCob)$ and $(X',\SCob')$ are called {\emph{equivalent}} if the trisection diagrams representing them are equivalent. The equivalence class of $(X,\SCob)$ is denoted by $[X,\SCob]$. We say that a nice trisection diagram $H$ is {\emph{compatible}} with the decorated cobordism $(X,\SCob)$ if $[X_H,\SCob_H]=[X,\SCob]$. We say that the nice trisection diagram $H$ is {\emph{compatible}} with the link cobordism $(X,\Surface)$ if there is a decoration $(X,\SCob)$ of $(X,\Surface)$ such that $H$ is compatible with $(X,\SCob)$. 
\end{defn}
\begin{rmk}\label{rmk:equivalence}
It follows from  Theorem 4.19, Lemma 6.11 and proof of Theorem 7.2 from \cite{AE-2}, that the trisection diagrams $H$ and $H'$  are equivalent if and only if after suitable stabilizations, there is a diffeomorphism of decorated cobordisms between 
\[(X_H,\SCob_H):\left(Y_{H_2},\Link_{H_2}\right)\ra \left(Y_{H_1},\Link_{H_1}\right)
\quad\text{and}\quad
(X_{H'},\SCob_{H'}):\big(Y_{H'_2},\Link_{H'_2}\big)\ra \big(Y_{H'_1},\Link_{H'_1}\big).
\]
In particular, the decorated cobordisms $(X,\SCob)$ and $(X',\SCob')$ are equivalent, if and only if after suitable stabilizations of (either of) them, there is a diffeomorphism of decorated cobordisms between them.
\end{rmk}

Stabilizations/destabilizations and slide diffeomorphisms suffice for going from one  trisection diagram compatible with a decorated cobordism to another such diagram. If we would also like to change the decoration, new types of modifications are necessary, which only change the decoration of the surface, and not the $4$-manifold $X$ or the embedding of the link cobordism in $X$. \\

Consider a decorated cobordism $(X,\TSurface)$, and let $\Surface$ be the underlying surface and $\bSurface$ be the corresponding closed surface obtained by collapsing each boundary component of $\Surface$ to a point. Let $\Tria$ denote the triangulation of $\bSurface$ encoded in $\TSurface$. Choose a red vertex $v$ of $\Tria$  and a pair of edges adjacent to it which are labeled by $e_{1}$ and $e_{2}$, respectively. The choice is determined by  the dashed green curve in Figure~\ref{fig:triangulation}(a). We then cut along the aforementioned two edges, to obtain a new diagram where $v$ is replaced by two red vertices $v'$ and $v''$, while the new diagram includes a $4$-gon domain, as in Figure~\ref{fig:triangulation}(b). The latter domain may be decomposed to a pink triangle and a gray triangle (the {\emph{new}} triangles), as in Figure~\ref{fig:triangulation}(c), to produce the new triangulation $\Tria'$ of $\bSurface$. The triangulation $\Tria'$ on $\bSurface$ determines a new decorated surface $(X,\TSurface')$,   which is represented by another equivalence class of trisection diagrams. If $H'$ belongs to the latter equivalence class, we say that $[H']$ is obtained from $[H]$ by a {\emph{perturbation}}.  The {\emph{deperturbation}} of $\Tria'$ to $\Tria$ may then be described as the process of collapsing the two new triangles to their common edge  (which is labeled $e_0$) so that $v'$ and $v''$ are collapsed to the same vertex $v$, which decomposes the aforementioned edge to the disjoint union of two edges with labels $e_1$ and $e_2$. \\

\begin{figure}
	\def\svgwidth{0.75\textwidth}
	{\footnotesize{
			\begin{center}
\begingroup%
  \makeatletter%
  \providecommand\color[2][]{%
    \errmessage{(Inkscape) Color is used for the text in Inkscape, but the package 'color.sty' is not loaded}%
    \renewcommand\color[2][]{}%
  }%
  \providecommand\transparent[1]{%
    \errmessage{(Inkscape) Transparency is used (non-zero) for the text in Inkscape, but the package 'transparent.sty' is not loaded}%
    \renewcommand\transparent[1]{}%
  }%
  \providecommand\rotatebox[2]{#2}%
  \newcommand*\fsize{\dimexpr\f@size pt\relax}%
  \newcommand*\lineheight[1]{\fontsize{\fsize}{#1\fsize}\selectfont}%
  \ifx\svgwidth\undefined%
    \setlength{\unitlength}{841.49986389bp}%
    \ifx\svgscale\undefined%
      \relax%
    \else%
      \setlength{\unitlength}{\unitlength * \real{\svgscale}}%
    \fi%
  \else%
    \setlength{\unitlength}{\svgwidth}%
  \fi%
  \global\let\svgwidth\undefined%
  \global\let\svgscale\undefined%
  \makeatother%
  \begin{picture}(1,0.32263752)%
    \lineheight{1}%
    \setlength\tabcolsep{0pt}%
    \put(0,0){\includegraphics[width=\unitlength,page=1]{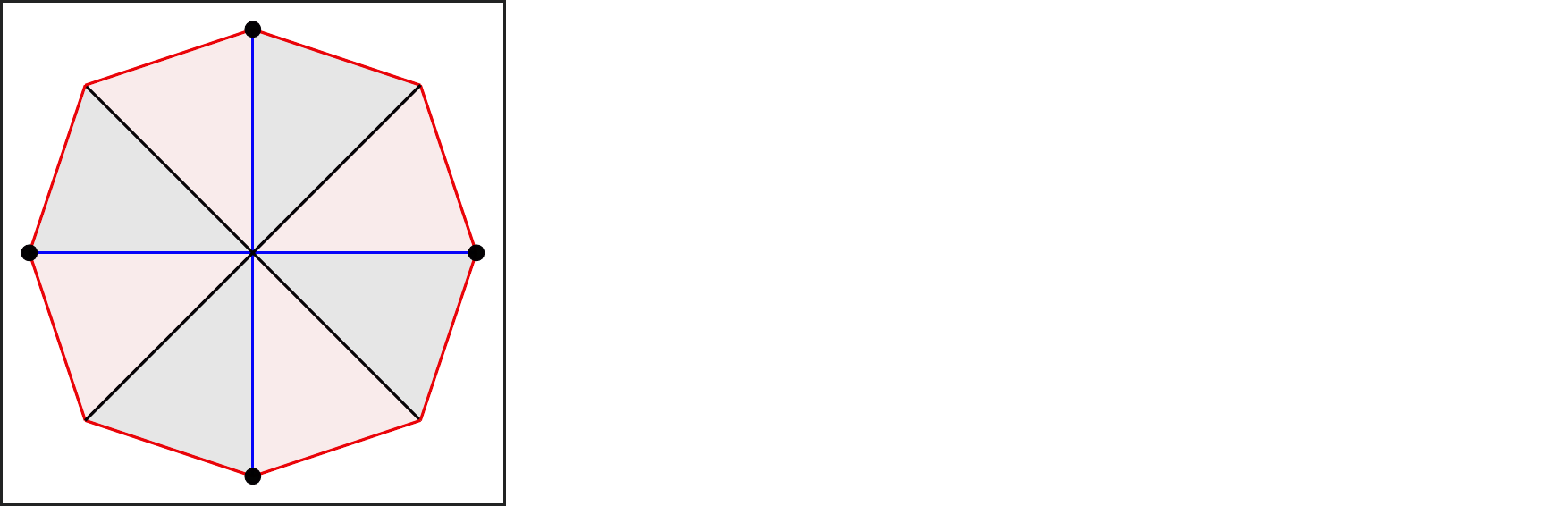}}%
    \put(0.00980364,0.01871584){\color[rgb]{0.01176471,0.01960784,0.02352941}\makebox(0,0)[lt]{\lineheight{1.25}\smash{\begin{tabular}[t]{l}(a)\end{tabular}}}}%
    \put(0,0){\includegraphics[width=\unitlength,page=2]{Triangulation-Perturbation.pdf}}%
    \put(0.34848458,0.01871584){\color[rgb]{0.01176471,0.01960784,0.02352941}\makebox(0,0)[lt]{\lineheight{1.25}\smash{\begin{tabular}[t]{l}(b)\end{tabular}}}}%
    \put(0,0){\includegraphics[width=\unitlength,page=3]{Triangulation-Perturbation.pdf}}%
    \put(0.68716573,0.01871584){\color[rgb]{0.01176471,0.01960784,0.02352941}\makebox(0,0)[lt]{\lineheight{1.25}\smash{\begin{tabular}[t]{l}(c)\end{tabular}}}}%
    \put(0,0){\includegraphics[width=\unitlength,page=4]{Triangulation-Perturbation.pdf}}%
  \end{picture}%
\endgroup%

	\end{center}}}
	\caption{A perturbation of diagram (a) is formed by first, cutting along the dashed green line and replacing the central vertex with two vertices to form diagram (b), and then decomposing the middle $4$-gon to a pair of pink and gray triangles, as in diagram (c).}
	\label{fig:triangulation}
\end{figure}

The corresponding cell decomposition $\Cell$ of $\oS$  is modified in a simple ways: we decompose one cell (which is a $2m$-gon) to two cells (a $2k$ gon and a $2(m-k+1)$-gon), as illustrated in Figure~\ref{fig:perturbation-d} by the modification from diagram $(a)$ to diagram $(b)$. The particular case, where $k=1$ (or $k=m$) is called a {\emph{simple}} perturbation. A simple  perturbation corresponds to repeating a red edge of $\Cell$, which adds one bigon to the set of cells, as the modification from diagram $(a)$ to diagram $(c)$ in  Figure~\ref{fig:perturbation-d} illustrates. Alternatively, a simple perturbation  replaces one component of the $1$-manifold determining the decoration $\SCob$ of the link cobordism $\Surface\subset X$  with two parallel copies of the same component, as illustrated in the modification from  Figure~\ref{fig:Simple-Perturbation-C}(a) to  Figure~\ref{fig:Simple-Perturbation-C}(b).

\begin{figure}
	\def\svgwidth{0.85\textwidth}
	{\footnotesize{
			\begin{center}
\begingroup%
  \makeatletter%
  \providecommand\color[2][]{%
    \errmessage{(Inkscape) Color is used for the text in Inkscape, but the package 'color.sty' is not loaded}%
    \renewcommand\color[2][]{}%
  }%
  \providecommand\transparent[1]{%
    \errmessage{(Inkscape) Transparency is used (non-zero) for the text in Inkscape, but the package 'transparent.sty' is not loaded}%
    \renewcommand\transparent[1]{}%
  }%
  \providecommand\rotatebox[2]{#2}%
  \newcommand*\fsize{\dimexpr\f@size pt\relax}%
  \newcommand*\lineheight[1]{\fontsize{\fsize}{#1\fsize}\selectfont}%
  \ifx\svgwidth\undefined%
    \setlength{\unitlength}{1051.49993381bp}%
    \ifx\svgscale\undefined%
      \relax%
    \else%
      \setlength{\unitlength}{\unitlength * \real{\svgscale}}%
    \fi%
  \else%
    \setlength{\unitlength}{\svgwidth}%
  \fi%
  \global\let\svgwidth\undefined%
  \global\let\svgscale\undefined%
  \makeatother%
  \begin{picture}(1,0.31526337)%
    \lineheight{1}%
    \setlength\tabcolsep{0pt}%
    \put(0,0){\includegraphics[width=\unitlength,page=1]{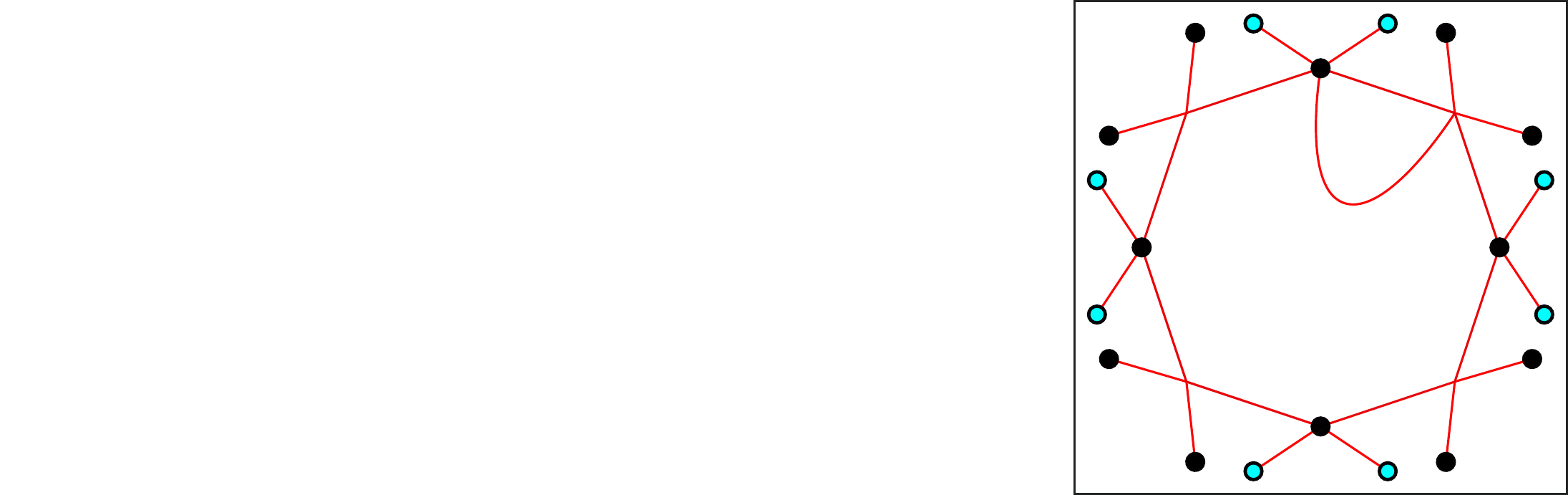}}%
    \put(0.7068474,0.01497789){\color[rgb]{0.01176471,0.01960784,0.02352941}\makebox(0,0)[lt]{\lineheight{1.25}\smash{\begin{tabular}[t]{l}$(c)$\end{tabular}}}}%
    \put(0,0){\includegraphics[width=\unitlength,page=2]{Red-Perturbation.pdf}}%
    \put(0.02211104,0.01497797){\color[rgb]{0.01176471,0.01960784,0.02352941}\makebox(0,0)[lt]{\lineheight{1.25}\smash{\begin{tabular}[t]{l}$(a)$\end{tabular}}}}%
    \put(0,0){\includegraphics[width=\unitlength,page=3]{Red-Perturbation.pdf}}%
    \put(0.36447906,0.01497789){\color[rgb]{0.01176471,0.01960784,0.02352941}\makebox(0,0)[lt]{\lineheight{1.25}\smash{\begin{tabular}[t]{l}$(b)$\end{tabular}}}}%
    \put(0,0){\includegraphics[width=\unitlength,page=4]{Red-Perturbation.pdf}}%
  \end{picture}%
\endgroup%

	\end{center}}}
	\caption{ A perturbation adds one red edge to the decorated cobordism, and is illustrated by the modification of $(a)$ to $(b)$. If one of the new cells is a bigon (as in $(c)$), the perturbation is called simple. }
	\label{fig:perturbation-d}
\end{figure}  

\begin{figure}
	\def\svgwidth{0.95\textwidth}
	{\footnotesize{
			\begin{center}
\begingroup%
  \makeatletter%
  \providecommand\color[2][]{%
    \errmessage{(Inkscape) Color is used for the text in Inkscape, but the package 'color.sty' is not loaded}%
    \renewcommand\color[2][]{}%
  }%
  \providecommand\transparent[1]{%
    \errmessage{(Inkscape) Transparency is used (non-zero) for the text in Inkscape, but the package 'transparent.sty' is not loaded}%
    \renewcommand\transparent[1]{}%
  }%
  \providecommand\rotatebox[2]{#2}%
  \newcommand*\fsize{\dimexpr\f@size pt\relax}%
  \newcommand*\lineheight[1]{\fontsize{\fsize}{#1\fsize}\selectfont}%
  \ifx\svgwidth\undefined%
    \setlength{\unitlength}{1255.58859446bp}%
    \ifx\svgscale\undefined%
      \relax%
    \else%
      \setlength{\unitlength}{\unitlength * \real{\svgscale}}%
    \fi%
  \else%
    \setlength{\unitlength}{\svgwidth}%
  \fi%
  \global\let\svgwidth\undefined%
  \global\let\svgscale\undefined%
  \makeatother%
  \begin{picture}(1,0.26484099)%
    \lineheight{1}%
    \setlength\tabcolsep{0pt}%
    \put(0,0){\includegraphics[width=\unitlength,page=1]{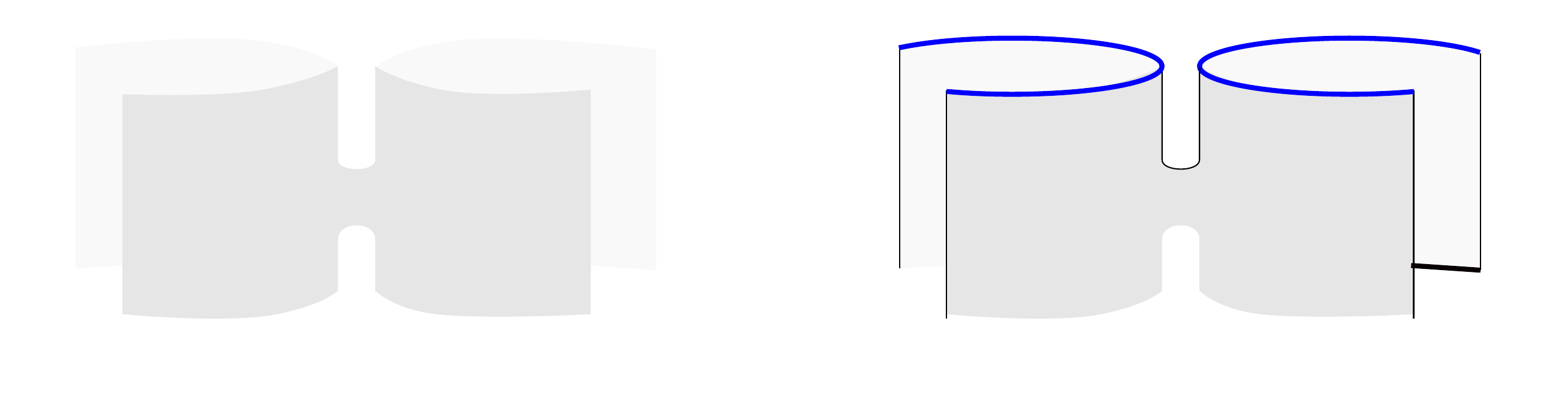}}%
    \put(0.21550787,0.00276265){\color[rgb]{0.10196078,0.10196078,0.10196078}\makebox(0,0)[lt]{\lineheight{1.25}\smash{\begin{tabular}[t]{l}\textbf{(a)}\end{tabular}}}}%
    \put(0,0){\includegraphics[width=\unitlength,page=2]{Simple-Perturbation-C.pdf}}%
    \put(0.74115762,0.00276265){\color[rgb]{0.10196078,0.10196078,0.10196078}\makebox(0,0)[lt]{\lineheight{1.25}\smash{\begin{tabular}[t]{l}\textbf{(b)}\end{tabular}}}}%
    \put(0,0){\includegraphics[width=\unitlength,page=3]{Simple-Perturbation-C.pdf}}%
  \end{picture}%
\endgroup%

	\end{center}}}
	\caption{A simple perturbation replaces one component of the $1$-manifold determining the decoration of the cobordism surface with two parallel copies of the same component, as illustrated in the modification from the decorated cobordism (a) to the decorated cobordism (b).}
	\label{fig:Simple-Perturbation-C}
\end{figure}  
\subsection{Moving between pointed trisection diagrams}\label{subsec:TD-are-related-by-perturbations}
 Since twists along the link components and perturbations do not change the surface or the $4$-manifold, we may regard them as operations on abstract decorated surfaces. An abstract decorated surface is a compact surfaces $\Surface$ with a decomposition of its boundary components to a  positive boundary $\partial^+\Surface$ (colored blue) and a negative boundary $\partial^-\Surface$ (colored black), which is decorated by a properly embedded  $1$-manifold $\Deco$ (without closed components) connecting the positive boundary to the negative boundary (colored red), so that  $\Surface-\Deco$ is a union of polygonal cells. As before, the closed surface obtained by collapsing each boundary component of $\Surface$ to a vertex (with the same color) is denoted by $\oS$ and the induced cell decomposition of $\oS$ is denoted by $\Cell$.  The cell decomposition $\Cell$  (associated with $\SCob$) consists of the vertices $V^+(\Cell)=V^+$ and $V^-(\Cell)=V^-$ and the edges $E(\Cell)$, each connecting a vertex in $V^+$ to a vertex in $V^-$,  which decompose $\oS$ to polygonal cells $D(\Cell)$ (each with an even number of edges). \\
 
 The cell decomposition $\Cell$ is called {\emph{simple}} if no cells in $D(\Cell)$ is a bigon, is called {\emph{complete}} if all cells in $D(\Cell)$ are quadrilaterals, and is called {\emph{deperturbed}} if $D(\Cell)$ consists of a single cell. $\SCob$ is called simple (respectively, complete or deperturbed) if the corresponding cell decomposition $\Cell$ is simple (respectively, complete or deperturbed). The number of edges for $\SCob$ is defined equal to the number of edges in $E(\Cell)$.  If $\SCob$ is an abstract decorated surface with underlying surface $\Surface$, we may apply simple deperturbations until we obtain a simple decoration. Further, deperturbations may be used to obtain a deperturbed decoration.\\
 
 An {\emph{edge switch}} is the process of removing an edge $e\in E(\Cell)$ which is common between two  polygons $P,P'\in D(\Cell)$, and adding an edge $e'$ inside the union polygon $P''=P\cup_e P'$ that connects a black vertex of $P''$ to a blue vertex of it, to create a new cell decomposition. Correspondingly, an edge switch may be regarded as the composition of a deperturbation and a perturbation on an abstract  decorated cobordism $\SCob$, which creates a new abstract decorated cobordism with the same number of edges. Note that an edge switch may be performed in case $P=P'$ as well. In this case, the new edge is required to separate the two copies if $e$ on the boundary of $P$.

\begin{thm}\label{thm:moving-between-decorations}
Every two abstract decorations $\SCob$ and $\SCob'$ of the same underlying surface  $\Surface$
 are related to one another by a sequence of perturbations and deperturbations, followed by twists along the boundary components. Moreover, if $\SCob$ and $\SCob'$ are simple decorations with the same number of edges, they may be changed to one another by a finite sequence of edge switches, followed by twists along the boundary components.
\end{thm}

This subsection is devoted to the proof of Theorem~\ref{thm:moving-between-decorations}, which plays a major role in the proof of our main result in this paper. The following two lemmas are the major tools needed in the proof.

\begin{lem}\label{lem:changing-deperturbed-cell-decomps}
Every two deperturbed cell decompositions $\Cell$ and $\Cell'$ of a closed surface $\oS$ with 
\[
V^+(\Cell)=V^+(\Cell')=V^+\quad\text{and}\quad V^-(\Cell)=V^-(\Cell')=V^-
\] 
may be changed to one another by a sequence of edge switches.	
\end{lem}
\begin{proof}
Let us denote the polygon in $D(\Cell)$ by $P$ and the polygon in $D(\Cell')$ by $P'$. Every edge in $E(\Cell)$ appears twice on the boundary of $P$ and every edge in $E(\Cell')$ appears twice on the boundary of $P'$. $\oS$ is obtained from $P$ by identifying the pairs of edges which correspond to the same element of $E(\Cell)$. Similarly, $\oS$ is also obtained from $P'$ by identifying the pairs of edges which correspond to the same element of $E(\Cell')$.\\

Let $e^1,\ldots,e^k$ denote the longest chain of consecutive edges on the boundary of $P$ which also appear in the same order and consecutively on the boundary of $P'$. The length $k$ of this chain, which is called a {\emph{common chain}} for $\Cell$ and $\Cell'$, may of course be zero. Let $e\in E(\Cell')$ denote the edge which appears right after $e^k$ on the boundary of $P'$.  Since $P$ gives a fundamental domain for $\oS$, we may draw $e$ as a union $e=e_0\amalg e_1\amalg \cdots \amalg e_l$ of $l+1$ arcs on $P$. Correspondingly, there are edges 
\[f_1,\ldots,f_l\in E(\Cell)-\{e^1,\ldots,e^k\},\] 
each corresponding to two edges on the boundary of $P$, with the following properties. The arc $e_0$ starts from the endpoint $v\in V^-$ of $e^k$ on $\partial P$, enters $P$ and leaves it from the first edge corresponding to $f_1$. For $1\leq j< l$, the arc $e_j$ enters $P$ from the second edge on $\partial P$ corresponding to $f_{j-1}$ and leaves $P$  from the first edge  on $\partial P$ corresponding to $f_j$. Finally, the arc $e_l$ enters $P$ from the second edge corresponding  to $f_l$ and reaches a vertex $v'\in V^+$ on the boundary of $P$. This setup is illustrated in Figure~\ref{fig:Edge-Switch}(a), where we assume $l=2$.\\ 	

\begin{figure}
	\def\svgwidth{0.8\textwidth}
	{\footnotesize{
			\begin{center}
				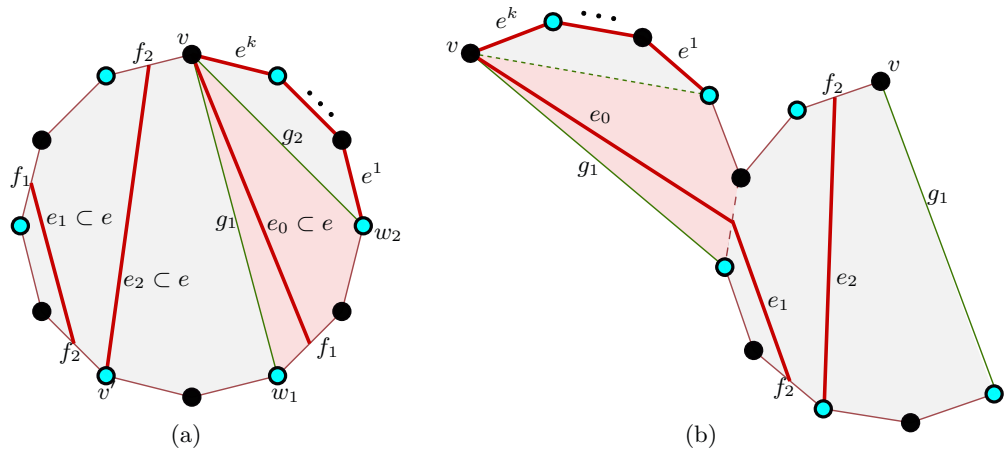
	\end{center}}}
	\caption{Two deperturbed cell decompositions $\Cell$ and $\Cell'$ with the same sets of vertices $V^\pm$ are given, and an edge $e$ in $\Cell'$, which  is a union of $l+1$ arcs $e^0\amalg e^2\amalg\cdots\amalg e^l$  on the fundamental domain $P=\oS-E(\Cell)$ is fixed. By an edge switch, which replaces $g_1$ for $f_1$, we arrive at a cell decomposition, where $e$ is represented by $l$ arcs.}
	\label{fig:Edge-Switch}
\end{figure}  

Two of the diagonals connecting $v$ to vertices $w_1$ and $w_2$ on $\partial P$ (which correspond to one or two vertices in $V^+$) are the closest of such diagonals to the arc $e_0$. The diagonal connecting $v$ to $w_i$ is denoted by $g_i$ for $i=1,2$. Then, $g_1$ and $g_2$ determine a quadrilateral subdomain of $P$, which is shaded pink in Figure~\ref{fig:Edge-Switch}(a). Note that one of the edges on $\partial P$ corresponding to $f_1$ belongs to the boundary of this quadrilateral. We may choose the labeling of $g_1$ and $g_2$ so that $g_1$ separates the two edges on $\partial P$ which correspond to $f_1$. Figure~\ref{fig:Edge-Switch}(b) illustrates the edge switch which replaces $g_1$ for $f_1$ in $\Cell_0=\Cell$ to create a new perturbed cell-decomposition $\Cell_1$. It is not hard to see that in the fundamental domain $P_1$ associated with $\Cell_1$, the edge $e$ is a union of $l$ arcs, and that $e^1,\ldots,e^k$ appear in this order and consecutively on $\partial P_1$. We may thus repeat the above process, to obtain the sequences
\[\Cell=\Cell_0,\Cell_1,\ldots,\Cell_l\quad\text{and}\quad P=P_0,P_1,\ldots,P_l\]
of perturbed cell decompositions and corresponding fundamental domains, so that $\Cell_j$ is obtained from $\Cell_{j-1}$ by an edge-switch,  the edge $e$ is a union of $l-j+1$ arcs in $P_j$, and  $e^1,\ldots,e^k$ appear in this order and consecutively on $\partial P_j$. In particular, $e$ becomes a diagonal in $P_l$ and may be switched with an edge of $P_l$ to create a perturbed cell decomposition $\Cell''$, with a corresponding fundamental domain $P''$, so that $e^1,e^2,\ldots,e^k,e$ appear in this order and consecutively on the boundary of $P''$. In other words,  we may use edge switches to change $\Cell$ to $\Cell''$, so that  the length of the longest common chain for $\Cell''$ and $\Cell'$ is bigger than  the length of the longest common chain for $\Cell$ and $\Cell'$. Repeating this process, it is now clear that $\Cell$ may be changed to $\Cell'$ by edge switches.   
\end{proof}		

The next step for proving Theorem~\ref{thm:moving-between-decorations} is to show that different decompositions of a polygon into quadrilaterals may be changed to one another by edge-switches. The proof of the following lemma (which also inspired the proof of Lemma~\ref{lem:changing-deperturbed-cell-decomps}) is suggested by Fatemeh Eftekhary.

\begin{lem}\label{lem:rectangulations}
Fix a convext polygon $P$ with a set $V^+$ of $m$ blue vertices  and a set $V^-$ of $m$ black vertices  which alternate on the boundary $\partial P$ of $P$. Fix two decompositions of $P$ into quadrilaterals, which are obtained by choosing a collection of $m-2$ diagonals with disjoint interiors, each connecting a black vertex to a blue vertex.  Then the two decompositions may be changed to one another by a sequence of edge switches, switching the diagonals in the interior of $P$. 	
\end{lem}	

\begin{figure}
	\def\svgwidth{0.65\textwidth}
	{\footnotesize{
			\begin{center}
				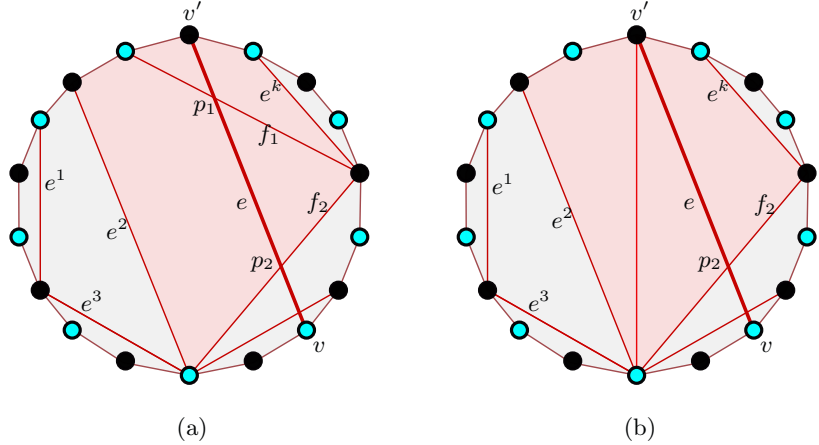
	\end{center}}}
	\caption{Two decompositions of the same polygon into quadrilaterals may be changed to one-another by switching diagonals. To accommodate an edge $e$, remove the diagonals cutting it one by one, and replace them by diagonals not cutting $e$. }
	\label{fig:Rectangulation}
\end{figure}  
\begin{proof}
Let $\Cell$ and $\Cell'$ denote the two  decompositions of $P$ into quadrilaterals which are obtained  by drawing the diagonals $E(\Cell,P)$ and $E(\Cell',P)$ of $P$, respectively. Further assume that $e^1,\ldots,e^k$ are the diagonals in $E(\Cell,P)\cap E(\Cell',P)$. Let $e$ denote a diagonal in $E(\Cell',P)\setminus E(\Cell,P)$, which connects $v\in V^+$ to $v'\in V^-$. It follows that $e$ cuts some diagonals $f_1,f_2,\ldots, f_l$ in $E(\Cell,P)-\{e^1,\ldots,e^k\}$.	Denote the intersection point of $f_j$ with $e$ by $p_j$. We may assume that $p_1,\ldots, p_l$ appear in this order on $e$, with $p_1$ the closest basepoint to $v'$ and $p_l$ the closest point to $v$ (see Figure~\ref{fig:Rectangulation}(a) for an illustration). If $f_1$ is removed from $\Cell$, a hexagon is created inside $P$, which is colored pink in Figure~\ref{fig:Rectangulation}(a). This hexagon includes $v'$ on its boundary. If we replace the diagonal $f_1$ with the diagonal (of the hexagon) having $v'$ as a vertex, we obtain a new decomposition $\Cell_1$ of $P$ into quadrilaterals. Note that 
\[e^1,\ldots, e^k\in E(\Cell_1,P)\quad\text{and}\quad \Big|e\cap \bigcup_{f\in E(\Cell_1,P)}f^\circ\Big|=l-1.\]
Repeating this process for $\Cell_1$ (instead of $\Cell_0=\Cell$), we obtain the decompositions $\Cell=\Cell_0,\Cell_1,\ldots,\Cell_l$ of $P$ into quadrilaterals, where $\Cell_j$ is obtained from $\Cell_{j-1}$ by an edge-switch, while
\[e^1,\ldots, e^k\in E(\Cell_j,P)\quad\text{and}\quad \Big|e\cap \bigcup_{f\in E(\Cell_j,P)}f^\circ\Big|=l-j,\quad\forall\ 0\leq j\leq l.\]
In particular, it follows that $e\in E(\Cell_l,P)$. Therefore, we may increase the number of common diagonals between $\Cell$ and $\Cell'$ by applying a sequence of edge switches to $\Cell$. Repeating the above process, we may eventually change $\Cell$ to $\Cell'$ by edge switches.
\end{proof}

\noindent{\sf\bfseries\color{black!50!blue} Proof of Theorem~\ref{thm:moving-between-decorations}.} 
Let $\SCob$ and $\SCob'$ be simple decorations of the same underlying surface $\Surface$,  which give the cell decompositions $\Cell$ and $\Cell'$ of  the corresponding closed surface $\oS$. If $\Cell$ and $\Cell'$ are isotopic, then $\SCob$ and $\SCob'$ may be changed to one another by isotopies and twists along the boundary components. Therefore, in the general case (where $\Cell$ and $\Cell'$ are not isotopic) it suffices to change $\Cell$ to $\Cell'$ using perturbations and deperturbations (or by edge switches, for the second cliam of the theorem).\\
 
We may deperturb $\Cell$ and $\Cell'$ to unperturbed cell decompositions $\Cell_0$ and $\Cell'_0$. Therefore, Lemma~\ref{lem:changing-deperturbed-cell-decomps} implies that $\Cell$ may be changed to $\Cell'$ by a sequence of perturbations and deperturbations. This already implies the first claim.  If we further assume that $|E(\Cell)|=|E(\Cell')|$, it also follows that the number of perturbations and deperturbations are equal. The proof of the theorem is thus reduced to showing that {\emph{in every simple decoration, a sequence of $k$ deperturbations followed by a perturbation, may be replaced by a sequence of edge-switches followed by $k-1$ deperturbations.}}\\

In order to show the above claim (and complete the proof), note that applying $k$ perturbation creates some $r$ polygons $P_1,\ldots, P_r$ with $2m_1,2m_2,\ldots,2m_r$ edges (respectively), where the initial cell decomposition includes at least one diameter from each $P_j$. The perturbation adds a diagonal $e$ to one of these polygons, say $P_1$. Let us complete the initial decomposition $\Cell|_{P_1}$ of $P_1$ (induced by $\Cell$) to a complete decomposition $\Cell_1$ of $P_1$ into quadrilaterals. We may also apply more perturbations to arrive at a decomposition $\Cell'_1$ of $P_1$ into quadrilaterals which includes $e$ as an edge. Lemma~\ref{lem:rectangulations} implies that the  decompositions $\Cell_1$ and $\Cell'_1$ of $P_1$  may be changed to one another by edge-switches. Ignoring the edges in $\Cell_1-\Cell |_{P_1}$, we obtain a sequence of edge switches which change $\Cell$ to a cell decomposition $\Cell''$, so that $\Cell$ and $\Cell''$ are identical outside $P_1$, while $\Cell''$ includes $e$ and has the same number of edges as $\Cell$.    
Therefore, the whole sequence of  $k$ deperturbations and a single perturbation may be replaced by the aforementioned edge-switches, followed by  $k-1$ deperturbations which remove all diagonals from $P_1,\ldots, P_r$, except the diagonal $e$.    	
\qed\\

The following theorem follows from Remark~\ref{rmk:equivalence} and Theorem~\ref{thm:moving-between-decorations}. It generalizes the result of  Hughes, Kim and Miller  for closed knotted surfaces, corresponding to the case where $L_j$ is an unlink in $Y_j=\#^{k_j}S^2\times S^1$ for $j=1,2$ (see \cite[Theorem 5.8]{HKM-uniqueness}, also see \cite[Theorem A]{GM-1} for the special case of doubly pointed trisection diagrams).

\begin{thm}\label{thm:uniqueness-of-trisection}
For every link cobordism $(X,\Surface):(Y_2,L_2)\ra (Y_1,L_1)$ and every decoration of it which is denoted by $(X,\SCob):(Y_2,\Link_2)\ra (Y_1,\Link_1)$,   there is a compatible nice trisection diagram. The nice diagrams $H$ and $H'$ are compatible with the same decorated cobordism if they may be changed to one another by a sequence of slide-diffeomorphisms stabilizations and destabilizations. If 
\begin{align*}
(X,\SCob):(Y_2,\Link_2)\ra (Y_1,\Link_1)\quad\text{and}\quad
(X,\SCob'):(Y_2,\Link_2')\ra (Y_1,\Link_1')
\end{align*}
are decorations of the same link cobordism $(X,\Surface)$, there are twist cobordisms
\[(X_j=Y_j\times [0,1],\SCob_j):(Y_j,\Link_j)\ra (Y_j,\Link_j)\quad j=1,2,\] 
such that the nice trisection diagrams compatible with  $(X_1,\SCob_1)\circ(X,\SCob)\circ (X_2,\SCob_2)$ and $(X,\SCob')$ may be changed to one another by a sequence of slide-diffeomorphisms, stabilizations, destabilizations, perturbations and deperturbations.
\end{thm}	

\section{The link concordance TQFT}\label{sec:concordance-tqft}
\subsection{The basepoint action and the homology action on Heegaard-Floer groups}\label{subsec:homology-action}
Assume that $Y$ is a closed $3$-manifold and $\Link$ is a pointed link consisting of a null-homologous link $L\subset Y$ with $c$ components, which form the set  $C(L)=\{L^1,\ldots,L^c\}$ of components, and the collections $\wpoint,\zpoint\subset L$  of basepoints so that $\zpoint^k=\zpoint\cap L^k$ and  $\wpoint^k=\wpoint\cap L^k$ are non-empty and alternate on $L^k$ for $k=1,\ldots,c$.
Given a $\SpinC$ structure $\spinc\in\SpinC(Y)$, let $H=(\HSurf,\alphas_1,\alphas_2,\wpoint,\zpoint)$ denote a strongly $\spinc$-admissible Heegaard diagram for the pointed link $(Y,\Link)$, and set $n=|\alpha_1|=|\alpha_2|$. If $\mathbb{T}_{\alpha_1}$ and $\mathbb{T}_{\alpha_2}$ denote the tori in $\Sym^{n}\HSurf$ associated with $\alphas_1$ and $\alphas_2$ respectively, we obtain natural maps 
\begin{align*}
\spinc_\wpoint,\spinc_\zpoint:	\mathbb{T}_{\alpha_1}\cap\mathbb{T}_{\alpha_2}\ra \SpinC(Y),
\end{align*}	
so that $\spinc_\wpoint-\spinc_\zpoint=\PD[L]$. In particular, if the pointed link $\Link$ is homologically trivial, $\spinc_\zpoint=\spinc_\wpoint$.\\

 Let $(C_\spinc,d_\spinc)=(\CFLm_{H,\spinc},d^-_{H,\spinc})$ denote the Heegaard-Floer chain complex associated with  the Heegaard diagram $H$ and the $\SpinC$ structure $\spinc\in\SpinC(Y)$, which is freely generated by the set  $\Sbf(H,\spinc) \subset\mathbb{T}_{\alpha_1}\cap\mathbb{T}_{\alpha_2}$ of generators representing the $\SpinC$ structure $\spinc$,  over the polynomial ring $\Ringm $. To define the differential $d_\spinc:C_\spinc\ra C_\spinc$, for every $\x,\y\in \Sbf(H,\spinc)$ let $\pi_2^j(\x,\y)$ denote the set of homotopy classes of Whitney disks connecting $\x$ to $\y$ which are of Maslov index $j$. Let
 \begin{align*}
 n_\wpoint(\phi)=\sum_{w\in\wpoint}n_w(\phi)\quad\text{and}\quad	
 n_\zpoint(\phi)=\sum_{z\in\zpoint}n_z(\phi)\quad(\text{for}\ \phi\in\pi_2(\x,\y)=\amalg_j\pi_2^j(\x,\y)),
 \end{align*}	
denote the coefficients of the domain $\Dcal(\phi)$ of $\phi$  at the basepoints in $\wpoint$ and $\zpoint$. For $\x\in\Sbf(H,\spinc)$ define
\[d_\spinc(\x):=\sum_{\y\in\Sbf(H,\spinc)}
\ \ \sum_{\substack{\phi\in\pi^1_2(\x,\y)}}\asf(\phi)\cdot \y,\quad\text{where}\quad
\asf(\phi):=\#\big(\Mcal(\phi)/\R\big)\cdot\wvar^{n_\wpoint(\phi)}\cdot\zvar^{n_\zpoint(\phi)}\in\Ringm,\]
where $\Mcal(\phi)$ denotes the moduli space of $J$-holomorphic representatives of $\phi$ for a generic path of almost complex structures $J$. If the reference to the Heegaard diagram $H$ is needed, we usually denote $(C_\spinc,d_\spinc)$ by $(C^H_\spinc,d^H_\spinc)$. The latter chain complex is  $\Z\oplus\Z$-filtered and its chain homotopy type is an invariant of  $(Y,\Link,\spinc)$. In fact, associated with  $(Y,\Link,\spinc)$, we obtain a transitive system of chain complexes $\CFLm(Y,\Link,\spinc)$ which is indexed by the strongly $\spinc$-admissible Heegaard diagrams $H$ representing $(Y,\Link)$. In particular, the   homology group 
\begin{align*}
\HFLm(Y,\Link,\spinc)=H_*(\CFLm(Y,\Link,\spinc))
\end{align*}
 is defined as a concrete $\Ringm$-module, and is an invariant of the triple $(Y,\Link,\spinc)$. The basepoints in $\wpoint$ and $\zpoint$ may be used to define the {\emph{relative homological gradings }}  $\grw$ and $\grz$, which are well-defined modulo the evaluation of $2c_1(\spinc)$ on second homology of $Y$. In particular, if $\spinc$ is a torsion $\SpinC$ structure, these relative homological gradings take their values in $\Z$. The homological bi-gradings of the variables $\wvar$ and $\zvar$ are defined equal to $(-2,0)$ and $(0,-2)$, respectively. The differential $d_\spinc:C_\spinc\ra C_\spinc$   drops the homological bi-grading by $(1,1)$.  
Set $\Abb^-=\Ringm$ and equip $\Abb^\circ=\Ringh$ and $\Ahat=\Ringhh$ with the structure of $\Ringm$-modules by identifying them as $\Ringm/\langle\wvar\rangle$ and $\Ringm/\langle\wvar,\zvar\rangle$, respectively. Set
\begin{align*}
&\CFLh(Y,\Link,\spinc):=\CFLm(Y,\Link,\spinc)\otimes_\Ringm\Abb^\circ,&&&&
\HFLh(Y,\Link,\spinc):=H_*(\CFLh(Y,\Link,\spinc)),\\
&\CFLhh(Y,\Link,\spinc):=\CFLm(Y,\Link,\spinc)\otimes_\Ringm\Ahat,&&&&
\HFLhh(Y,\Link,\spinc):=H_*(\CFLhh(Y,\Link,\spinc)),\\
&\CFLs(Y,\Link):=\bigoplus_{\spinc\in\SpinC(Y)}\CFLs(Y,\Link,\spinc)&&\text{and}&&
\HFLs(Y,\Link):=\bigoplus_{\spinc\in\SpinC(Y)}\HFLs(Y,\Link,\spinc).
\end{align*}	
Although different diagrams may be needed to form the complexes  in the fifth definition above, we  abuse the notation and set $(C,d)=\bigoplus_{\spinc\in\SpinC(Y)}(C_\spinc,d_\spinc)$ without referring to the corresponding Heegaard diagram. For a discussion of the naturality of these invariants, see  \cite{AE-2} or \cite{Zemke-1}. \\

The basepoints  $w\in\wpoint$ and  $z\in\zpoint$ determine the  basepoint action maps $\Phi_w:C\ra C$ and $\Psi_z:C\ra C$, which are defined on $C_\spinc$ by 
\begin{align*}
	&\Phi_w(\x)=\sum_{\substack{\y\in\Sbf(H,\spinc)}}\sum_{\substack{\phi\in\pi^1_2(\x,\y)}}
	\frac{n_w(\phi)\cdot\asf(\phi)}{\wvar} \cdot\y 	\quad\text{and}\quad
	\Psi_z(\x)=\sum_{\substack{\y\in\Sbf(H,\spinc)}}\sum_{\substack{ \phi\in\pi^1_2(\x,\y)}}
	\frac{n_z(\phi)\cdot\asf(\phi)}{\zvar} \cdot\y ,
\end{align*}	
respectively. We drop the $\SpinC$ structure from the notation when there is no confusion. These chain maps are homogenous of homological bi-degrees $(1,-1)$ and $(-1,1)$, respectively and induce the well-defined $\Ringm$-homomorphisms 
\begin{align*}
\Phi_w^\star,\Psi_z^\star:\HFLs(Y,\Link)\ra\HFLs(Y,\Link),
\quad\text{for}\ \ \star\in\{-,\circ,\wedge\}.
\end{align*}
 The latter maps are independent of the choice of the pointed Heegaard diagram for  $(Y,\Link,\spinc)$ and respect the decomposition by $\SpinC$ structures in $\SpinC(Y)$.\\

 There is also an action of $\mywedge_{Y}=\mywedge^*H_1(Y;\Fbb)$ on $\HFLm(Y,\Link,\spinc)$, which is defined as follows (see \cite[Proposition 4.17]{OS-3m1}). Let $\gamma$ denote an  immersed closed $1$-manifold in $\HSurf-\wpoint-\zpoint$ representing the homology class $[\gamma]\in H_1(Y;\Fbb)$, which is in general position with respect to $\alphas_1$ and $\alphas_2$.  Fix the points $v^j=(j,0)$ for $j=0,1$  on the boundary of $[0,1]\times\R$.  Define the chain maps $A^j_\gamma:C\ra C$  by
 \begin{displaymath}
 	\begin{split}
 		&A^j_{\gamma}(\x):=\sum_{\y\in\Sbf(H,\spinc)}\sum_{\phi\in\pi^1_2(\x,\y)}
 		\mathsf{a}^j(\gamma,\phi)\cdot\wvar^{n_\wpoint(\phi)}\cdot\zvar^{n_\zpoint(\phi)}\cdot\y,\ \ \forall\ \x\in\Sbf(H,\spinc),\ \ \text{where}\\
 		&\mathsf{a}^j(\gamma,\phi)=\#\Mcal^j(\gamma,\phi)=\#\left\{(u,x)\   \big|\  u\in\Mcal(\phi), x\in\mathrm{Domain}(u),  u(x)\in\{v^j\}\times(\gamma\cap\alphas_{j})\right\}.
 	\end{split}
 \end{displaymath}	
 A restatement of \cite[Proposition 4.7]{OS-3m1} may then be formulated as the following lemma.
 \begin{lem}\label{lem:homology-action}
 With the above notation in place,	
 	$A^j_\gamma:C\ra C$ is a chain map and the induced map
 \begin{align*}	
 	\HAction^j_\gamma:\HFLs(Y,\Link,\spinc)\ra \HFLs(Y,\Link,\spinc)
\end{align*}
 	 only depends on  the homology class $[\gamma]\in H_1(Y;\Fbb)$, for $\star\in\{-,\circ,\wedge\}$. Moreover,  $\HAction_\gamma^{1}=\HAction_\gamma^{2}$, and we may thus denote $\HAction_\gamma^{j}$ by $\HAction^\star_\gamma$, which satisfies $\HAction^\star_{\gamma}\circ\HAction^\star_{\gamma}=0$. In particular, the homology action extends to an action of the exterior algebra  $\mywedge_{Y}$  on $\HFLs(Y,\Link,\spinc)$, which  commutes with the basepoint actions of $\wedgews$ and $\wedgezs$ on $\HFLs(Y,\Link,\spinc)$. 
 \end{lem}	
 \begin{proof}
 	Only the last claim requires a new proof (the rest of the statements follow from the argument of  \cite[Proposition 4.7]{OS-3m1}).  With the above setup in place, let us fix the homology class $\gamma$ and $w\in\wpoint$ and define the $\Abb^-$-homomorphisms $H^j_{w,\gamma}:C\ra C$ by
 	\begin{align*}
 		H^j_{w,\gamma}(\x):=\sum_{\y\in\Sbf(H,\spinc)}\sum_{\phi\in\pi_2^1(\x,\y)}	n_w(\phi)\mathsf{a}^j(\gamma,\phi) \cdot \wvar^{n_\wpoint(\phi)-1}\cdot\zvar^{n_\zpoint(\phi)}\cdot\y,\quad
 		\forall\ \x\in\Sbf(H,\spinc). 
 	\end{align*}	
 	It then follows from standard degeneration arguments that
 	\begin{align*}
 		A^j_\gamma\circ \Phi_w-\Phi_w\circ A^j_\gamma=H^j_{w,\gamma}\circ d-d\circ H^j_{w,\gamma},	
 	\end{align*}	
 	completing the proof of the last claim.
 \end{proof}	
 
\subsection{The link Floer homology groups}\label{subsec:link-Floer-homology} 
 Let us fix a pointed link $(Y,\Link)$ with the underlying link $(Y,L)$ as before.  
   For $\star=-,\circ,\wedge$, set
 \[\Phi_L^\star:=\sum_{K\in C(L)}\Phi^\star_K\ \ \text{and}\ \ \Psi^\star_L:=\sum_{K\in C(L)}\Psi^\star_K\quad
 \text{where}\quad 
 \Phi^\star_K:=\sum_{w\in \wpoint\cap K}\Phi^\star_w\ \ \text{and} \ \  
\Psi^\star_K:=\sum_{z\in \zpoint\cap K}\Psi^\star_z
\]
for every component $K\in C(L)$.
The following proposition gathers most of the properties we need about the basepoint action maps:
\begin{prop}\label{prop:point-action-independence}
	Let $\wpoint$ and  $\zpoint$ denote the two collections of basepoints on the pointed link $(Y,\Link)$ with the underlying link $(Y,L)$. 
	Then, for every  $w,w'\in\wpoint$, $z,z'\in\zpoint$ and $\star\in\{-,\circ\}$ we have: 
	\begin{equation}\label{eq:commutation-properties}
		\begin{split}
			&\Phi^\star_w\circ\Phi^\star_w=0,\quad\Psi^\star_z\circ\Psi^\star_z=0,\quad [\Phi^\star_w,\Phi^\star_{w'}]=0,\quad [\Psi^\star_z,\Psi^\star_{z'}]=0\quad\text{and}\quad
			[\Phi^\star_w,\Psi^\star_z]=\delta_{w,z} \cdot Id.
		\end{split}
	\end{equation} 
	Here  $[\gmap,\gmap']=\gmap\circ\gmap'-\gmap'\circ\gmap$ 
	for $\gmap,\gmap':\HFLs(Y,\Link)\ra \HFLs(Y,\Link)$, and  
	$\delta_{w,z}=0$ unless $z$ and $w$ are neighbors on a component $K$ of $L$ with $|K\cap\wpoint|=|K\cap\zpoint|>1$, where we set $\delta_{w,z}=1$. 
	In particular, $\HFLs(Y,\Link)$  admits actions  by the exterior algebras 
	$\wedgews$ and $\wedgezs$ generated by $\{\Phi^\star_w\}_{w\in\wpoint}$ and $\{\Psi^\star_z\}_{z\in\zpoint}$ over $\Abb^\star$, respectively. Moreover, the action of $\Phi_L^-$ is trivial, i.e. zero.
\end{prop}
\begin{proof}
All claims, except for the last one, follow from \cite[Lemmas 4.6 and 4.7]{Zemke-1}. We thus focus on proving the triviality of the actions of $\Phi_L^-$. 
Let us fix the strongly $\spinc$-admissible Heegaard diagram $H$ as before for some $\spinc\in\SpinC(Y)$.	 The differential $d_\spinc$ of the chain complex $(C_\spinc,d_\spinc)$ induces a differential $d_\spinc^-$ on the $\Abb^-$-module
\begin{align*}
	C_\spinc^-=\HFLhh(Y,\Link,\spinc)\otimes_{\Fbb}\Ringm,
\end{align*}		
so that the homology of
$(C_\spinc^-,d_\spinc^-)$ gives $\HFLm(Y,\Link,\spinc)$. 
Moreover, the $\Ringm$-homomorphism $d^-_\spinc$ determines the point action map $\Phi_L^-$ on $\HFLm(Y,\Link,\spinc)$, as follows. Let us denote a set of homogeneous generators  (with respect to the homological bi-grading) of $\HFLhh(Y,\Link,\spinc)$ by $\Sbf'(H,\spinc)$ and assume that for every $\x\in\Sbf'(H,\spinc)$ we have
\begin{align*}
	d^-_\spinc(\x)=\sum_{\y\in\Sbf'(H,\spinc)}\epsilon(\x,\y)\wvar^{k(\x,\y)}\zvar^{l(\x,\y)}\y,	
\end{align*}	
for $\epsilon(\x,\y)\in\Fbb$ and some integers $k(\x,\y)$ and $l(\x,\y)$ which are non-negative if $\epsilon(\x,\y)$ is non-zero. In fact,  $k(\x,\y)+l(\x,\y)$ is strictly positive if $\epsilon(\x,\y)$ is non-zero. We then obtain a chain map, which is denoted by $\Phi_L:C_\spinc^-\ra C_\spinc^-$ and is defined by
\begin{align*}
	\Phi_L(\x)=\sum_{\y\in\Sbf'(H,\spinc)}\epsilon(\x,\y)k(\x,\y)\wvar^{k(\x,\y)-1}\zvar^{l(\x,\y)}\y,
\end{align*}
for every $\x\in\Sbf'(H,\spinc)$. It follows that $\Phi_L^-$ is the maps induced on $\HFLm(Y,\Link,\spinc)$ by $\Phi_L$. 
Choose a closed element  $\x\in\HFLm(Y,\Link,\spinc)$ which is homogeneous with respect to the homological bi-grading. For simplicity, let us assume that the homological bi-grading of $\x$ is $(0,0)$ (this may be obtained by a shift in bi-grading, which does not affect the claim). We may then choose the basis $\Sbf'(H,\spinc)$ so that 
\begin{align*}
	\x=\sum_{j=1}^m\wvar^{k_j}\zvar^{l_j}\x_j,\quad\text{where}\ \x_1,\ldots,\x_m\in\Sbf'(H,\spinc),\quad k_1>\cdots>k_m\geq 0\quad \text{and}\quad 0\leq l_1<\cdots<l_m.	
\end{align*}		 
The homological bi-grading of $\x_j$ is then given by $(2k_j,2l_j)$.  For every $\y\in\Sbf'(H,\spinc)$ in homological bi-grading $(2k(\y)-1,2l(\y)-1)$, the coefficient of $\y$ in $d^-_\spinc(\x)$ is 
\begin{align*}
	\sum_{j=1}^m \epsilon(\x_j,\y)\wvar^{k(\y)}\zvar^{l(\y)}=\wvar^{k(\y)}\zvar^{l(\y)}\Big(\sum_{j=1}^m\epsilon_j(\x_j,\y)\Big).
\end{align*}  
Since $\x$ is $d^-_\spinc$-closed, it follows that for every $\y\in\Sbf'(H,\spinc)$, $\sum_{j=1}^m\epsilon_j(\x_j,\y)=0$.
On the other hand, the coefficient of $\y$ in $\Phi_L(\x)$ is given by
\begin{align*}
	\sum_{j=1}^m (k(\y)-k_j)\epsilon(\x_j,\y)\wvar^{k(\y)-1}\zvar^{l(\y)}=\wvar^{k(\y)-1}\zvar^{l(\y)}\Big(\sum_{j: k_j\ \text{is odd}}\epsilon_j(\x_j,\y)\Big).
\end{align*}  
If we set $\z=\sum_{j: k_j\ \text{is odd}}\wvar^{k_j-1}\zvar^{l_j}\x_j$, it follows that
\begin{align*}
	d_\spinc^-(\z)=	
	\sum_{\y\in\Sbf'(H,\spinc)}\sum_{j: k_j\ \text{is odd}}\epsilon_j(\x_j,\y)\wvar^{k(\y)-1}\zvar^{l(\y)}=\Phi_L(\x).
\end{align*}  
Therefore, the image of the homology class in $\HFLm(Y,\Link,\spinc)$ represented by $\x$ under $\Phi_L^-$ is trivial. 
This completes the proof.
\end{proof}	
\begin{defn}\label{defn:weak-HF}
Given a pointed link $(Y,\Link)$ with underlying  oriented  link $(Y,L)$, for $\star\in\{-,\circ,\wedge\}$ define the {\emph{Floer homology groups}}  $\HFKs(Y,L)$  by
\begin{align}\label{eq:kernel-defn}
&\HFKs(Y,L)=\bigoplus_{\spinc\in\SpinC(Y)}\HFKs(Y,L,\spinc):=\Ker\big(\mywedge^\star_\wpoint:\HFLs(Y,\Link)\ra \HFLs(Y,\Link)\big).
\end{align}		
Moreover, for $\star\in\{-,\circ\}$ define the {\emph{weak Floer homology groups}} $\HFKsw(Y,L)$  by
 \begin{align*}
 	&\HFKsw(Y,L)=\bigoplus_{\spinc\in\SpinC(Y)}\HFKsw(Y,L,\spinc)
 	:=\HFKs(Y,L)\cap \left(\zvar\cdot\HFLs(Y,\Link)\right)\subset\HFKs(Y,L).
 \end{align*}
\end{defn}	

\begin{prop}\label{prop:weak-HF}
Given a pointed link $(Y,\Link)$ with underlying  oriented  link $(Y,L)$, and a $\SpinC$ structure $\spinc\in\SpinC(Y)$, the  Floer homology groups 	$\HFKs(Y,L,\spinc)$ and the weak Floer homology groups $\HFKsw(Y,L,\spinc)$  only depend on  $(Y,L,\spinc)$ and not the decoration $\Link$ of $L$, for $\star\in\{-,\circ,\wedge\}$. These groups are equipped with  relative homological bi-gradings $(\grw,\grz)$ (again, independent of the decoration) which take their values in $\Z$ if $\spinc$ is a torsion $\SpinC$ structure. 
\end{prop}	
\begin{proof}
Let us denote the group defined by the right-hand-side-of (\ref{eq:kernel-defn}) by $\HFKs(Y,L,\spinc;\Link)$.	
Given a pointed link $(Y,\Link)$, with underlying link $(Y,L)$ and $\Link=(L,\wpoint,\zpoint)$, assume that $z$ and $w$ are basepoints belonging to the same component of $L-\wpoint-\zpoint$. We then set $\wpoint'=\wpoint\cup\{z\}$ and $\zpoint'=\zpoint\cup\{w\}$ and further assume that $z$ and $w$ are labeled so that $\Link'=(L,\wpoint',\zpoint')$ gives a new pointed link in $Y$ with a pair of extra basepoints on it. Zemke defines the quasi-stabilization maps 
\begin{align*}
S_{w,z}^+,S_{z,w}^+:\CFLm(Y,\Link,\spinc)\ra \CFLm(Y,\Link',\spinc)\quad\text{and}\quad
S_{w,z}^-,S_{z,w}^-:\CFLm(Y,\Link',\spinc)\ra \CFLm(Y,\Link,\spinc).	
\end{align*}
in  \cite{Zemke-1} and proves that for every $w'\in\wpoint$ and $z'\in\zpoint$, $\Phi_{w'}$ and $\Psi_{z'}$ commute with $S_{w,z}^{\pm}$ and $S_{z,w}^{\pm}$ (up to chain homotopy), while the following relations are also satisfied (again, up to chain homotopy) \cite[Lemmas 4.12, 4.13 and 4.14]{Zemke-1} :
\begin{align*}
&(1)\ \Phi_w=S_{w,z}^+S_{w,z}^-,
&&&&(2)\ S_{z,w}^-S_{z,w}^+=S_{w,z}^-S_{w,z}^+=0,\\
&(3)\ S_{w,z}^-S_{z,w}^+=S_{z,w}^-S_{w,z}^+=Id,
&&\text{and}
&&(4)\ S_{z,w}^+S_{w,z}^-+S_{w,z}^+S_{z,w}^-=Id.
\end{align*}	
It follows from $(2),(3)$ and $(4)$ that the maps induced by $S_{z,w}^+$ and $S_{w,z}^+$ in homology give isomorphisms to their images and that we have a decomposition
\begin{align*}
\HFLm(Y,\Link',\spinc)=\Image\big((S_{z,w}^+)_*\big)\oplus\Image\big((S_{w,z}^+)_*\big)
=\HFLm(Y,\Link,\spinc)\oplus\HFLm(Y,\Link,\spinc).	
\end{align*}	
Moreover, it follows from $(1),(2)$ and $(3)$ that $\Phi^-_w$ is trivial on $\Image\big((S_{w,z}^+)_*\big)$ and that 
under the above identification, $\Phi^-_w$ may in fact be identified as the identity map from the first component to the second component (i.e. that it is given by $\colvec{0&0\\I&0}$). In particular, $\Ker(\Phi^-_w)$ is identified as $\Image\big((S_{w,z}^+)_*\big)$. Moreover, since the isomorphism 
\begin{align*}
	(S^+_{z,w})_*:\HFLm(Y,\Link,\spinc)\ra \Ker(\Phi^-_w)\subset\HFLm(Y,\Link',\spinc)
\end{align*}	
 commutes with all point action maps $\Phi^-_{w'}$ for all $w'\in\wpoint$, it follows that $(S^+_{z,w})_*$ gives an isomorphism from  $\HFKm(Y,L,\spinc;\Link)$ to  $\HFKm(Y,L,\spinc;\Link')$. If $(w',z')$ is a new pair of basepoints on a component of $L-\wpoint'-\zpoint'$, \cite[Proposition 4.18]{Zemke-1} implies that $S_{z,w}^+\circ S_{z',w'}^+=S_{z',w'}^+\circ S_{z,w}^+$ up to chain homotopy. Therefore, we obtain a compatible system of isomorphisms between  $\HFKm(Y,L,\spinc;\Link)$ for different decorations $\Link$ of $L$, i.e. an isomorphism 
\begin{align*}
S_{\Link\ra\Link'}:\HFKm(Y,L,\spinc;\Link)\ra \HFKm(Y,L,\spinc;\Link')
\end{align*}
whenever $\Link'$ is obtained by adding basepoints to $\Link$, so that $S_{\Link'\ra\Link''}\circ S_{\Link\ra\Link'}=S_{\Link\ra\Link''}$. In particular, $\HFKm(Y,L,\spinc)$ may be defined as a concrete group. The same argument as above implies that $\HFKs(Y,L,\spinc)$ may be defined as a concrete group  for $\star\in\{\circ,\wedge\}$. The argument for the invariance and naturality of $\HFKsw(Y,L,\spinc)$ is similar.
\end{proof}	

If $L$ is a knot and $\Link$ is a decoration of $L$ with a pair of basepoints on it, the actions of $\wedgewhh$ and $\wedgezhh$ on $\HFLhh(Y,L)$ are {\emph{usually}} non-trivial (see \cite{Sarkar-Auto}). Nevertheless, the induced actions of $\wedgews$ on $\HFLs(Y,L)$ are always trivial for $\star=-,\circ$, by the last claim in Proposition~\ref{prop:point-action-independence}.  Therefore, the following proposition follows immediately.

\begin{prop}\label{prop:HF-of-knots}
	If $\Link$ is a decoration of a knot $K$ in the $3$-manifold $Y$  by a pair of basepoints, 
\begin{align*}
			\HFKm(Y,L,\spinc)=\HFLm(Y,\Link,\spinc),\quad\forall\ \spinc\in\SpinC(Y).
\end{align*}
\end{prop}

\subsection{The trisection map and the group actions}\label{subsec:trisection-map}
Let  us now assume that $H=(\HSurf,\alphas_0,\alphas_1,\alphas_2,\wpoint,\zpoint)$  is a  trisection diagram of type $(g;n;\cbf)$  with $\cbf=(c_0,c_1,c_2)$. Let  $Y_i=\partial_i X^\circ_H$ be the $i$-th boundary component of the $4$-manifold $X^\circ_H$ associated with $H$. Let $H_i$ be obtained from $H$ by removing $\alphas_i$ and represent the pointed link $(Y_i,\Link_i)$. Set $\Sbf(H_i)=\amalg_{\spinc_i\in\SpinC(Y_i)}\Sbf(H_i,\spinc_i)$. Then, for   every triple of generators 
\[\x=\x_0\times\x_1\times\x_2\in\Sbf(H):=\Sbf(H_0)\times\Sbf(H_1)\times\Sbf(H_2)\] 
let  $\pi_2^j(\x):=\pi_2^j(\x_0,\x_1,\x_2)$ denote the space of homotopy classes of triangles with Maslov index $j$  connecting $\x_0,\x_1$ and $\x_2$ and set $\pi_2(\x)=\amalg_j\pi_2^j(\x)$.  Given   $\x$ and $\x'=\x'_0\times\x'_1\times\x'_2\in\Sbf(H)$,  the triangle classes $\Delta\in\pi_2(\x)$ and $\Delta'\in\pi_2(\x')$ are called equivalent if there exist $\phi_i\in\pi_2(\x_i,\x'_i)$ (for $i\in\Z/3$) such that $\Delta$ is obtained from $\Delta'$ by juxtaposition of each $\phi_i$ at  $\x'_i$. The set of equivalence classes of such triangle classes is denoted by $\Triangles(H)$ and  is identified   with the set $\SpinC(X^\circ_H)$ of $\SpinC$ structures on  $X^\circ_H$ (e.g. \cite[Proposition 2.6]{AE-2}) by a pair of maps \[\spinc_\wpoint,\spinc_\zpoint:\Triangles(H)\ra\SpinC(X^\circ_H).\] 
If  $H$ is nice, there is a subset $\Triangles_0(H)\subset\Triangles(H)$ which consists of classes $\Delta\in\pi_2(\x_0,\x_1,\x_2)$, where $\spinc_\wpoint(\x_0)=\spinc_\zpoint(\x_0)=\spinc_0$ is the canonical $\SpinC$ structure on $Y_0=\#^{k_0}S^2\times S^1$ (with $c_1(\spinc_0)=0$). Then
\[\spinc_\wpoint,\spinc_\zpoint:\Triangles_0(H)\ra\SpinC(X_H)\subset \SpinC(X^\circ_H).\]
If  the second homology group of $Y_i$ is trivial for $i=1,2$, $\Surface$ determines a well-defined homology class $[\Surface]\in  H_2(X_H;\Z)$ (by adding a Seifert surface for each $L_i$ to $\Surface$) and 
\begin{align*}
	\spinc_\wpoint(\Delta)-\spinc_\zpoint(\Delta)=\PD[\Surface],\quad\forall\ \Delta\in\Triangles_0(H).
\end{align*} 
We usually use the identification given by $\spinc_\wpoint$ to associate $\SpinC$ structures in $\SpinC(X_H)$ to triangle classes in $\Triangles_0(H)$. 
For $\spinct\in\SpinC(X_H)$, the corresponding subset of $\pi_2(\x)$ is denoted by $\pi_2(\x,\spinct)$. Every $\SpinC$ class $\spinct\in\SpinC(X_H)$ restricts to 
$\spinct|_{Y_i}\in\SpinC(Y_i)$ for $i=1,2$. 
Thus $\pi_2(\x,\spinct)=\emptyset$ unless $\spinc_{\wpoint}(\x_i)=\spinct|_{Y_i}$ for $i=1,2$.
Given $\spinct\in\SpinC(X_H)$, the $\spinct$-{\emph{admissibility}} of  $H$  will be implicit throughout our discussions (see \cite[Section 2.5]{AE-2} for precise definitions). We  sometimes skip mentioning this assumption, since $\spinct$-admissibility may be acheived using isotpies.\\

Fix  a decorated cobordism $(X,\TSurface,\spinct):(Y_2,\Link_2,\spinct|_{Y_2})\ra (Y_1,\Link_1,\spinct|_{Y_1})$ and a  nice ($\spinct$-admissible)
trisection diagram $H$  compatible with $(X,\TSurface)$. 
Associated with $H$ we then find the pointed links $(Y_{H_i},\Link_{H_i})$ for $i\in\Z/3$. Let $\spinc_i$ denote the $\SpinC$ structure $\spinct|_{Y_{H_i}}$ and 
$(C_i,d_i)$ denote the chain complex associated with $(H_i,\spinc_i)$.
We further assume that $Y_{H_i}=Y_i\#(\#^{k_i}S^2\times S^1)$ for $i=1,2$ and that $Y_{H_0}=\#^{k_0}S^2\times S^1$ and that
$\spinc_0$ is the $\SpinC$ class on $Y_{H_0}$ with $c_1(\spinc_0)=0$.  Let $\x_0^{top}$ denote the top generator of $\HFLm(H_0,\spinc_0)$ with respect to the homological grading $\grw$. Define 
\begin{align*}
\fmap^\star_{H,\spinct}:\HFLs(Y_{H_2},\Link_{H_2},\spinct|_{Y_{H_2}})\ra 
\HFLs(Y_{H_1},\Link_{H_1},\spinct|_{Y_{H_1}})\quad\text{for}\ \ \star\in\{-,\circ,\wedge\},
\end{align*}
as the map induced on homology by the chain map $f_{H,\spinct}:C_2\ra C_1$,  given by 	
\begin{equation}\label{eq:map-definition}
	\begin{split}	
		&f_{H,\spinct}(\x_2)=F_\spinct(\x_2\otimes\x_0^{top}):=\sum_{\x_1\in \Sbf(H_1,\spinc_1)}\sum_{\substack{\Delta\in\pi^0_2(\x_0^{top}\times\x_1\times\x_2,\spinct)}}\asf(\Delta)\x_1,\quad\forall\ \ \x_2\in\Sbf(H_2)\\ &\quad\quad\text{where}\quad \asf(\Delta):=\#\Mcal(\Delta)\cdot\wvar^{n_\wpoint(\Delta)}\cdot\zvar^{n_\zpoint(\Delta)}\in\Ringm.
	\end{split}	
\end{equation}
The  $\Abb^\star$-homomorphism $\fmap^\star_{H,\spinct}$ 
respects the homological and basepoint actions in the following sense. 
\begin{lem}\label{lem:homology-action-interactions}
	With the nice trisection diagram $H=(\HSurf,\alphas_0,\alphas_1,\alphas_2,\wpoint,\zpoint)$ and the $\SpinC$ class $\spinct$ on $X_H$ as before, let $\gamma$ be a closed loop on $\HSurf$ and $w\in\wpoint$ be a fixed basepoint.  Then for $\star\in\{-,\circ,\wedge\}$,
	\begin{equation}\label{eq:action-and-the-map}
		\begin{split}
			&\fmap^\star_{H,\spinct}\left(\HAction^\star_\gamma(\x_2)\right)=\HAction^\star_\gamma\left(\fmap^\star_{H,\spinct}(\x_2)\right)\ \ \text{and}\ \  \fmap^\star_{H,\spinct}\left(\Phi^\star_w(\x_2)\right)=\Phi^\star_w\left(\fmap^\star_{H,\spinct}(\x_2)\right)
			,\ \ \forall\ \x_2\in\HFLs(H_2,\spinc_2).
		\end{split}
	\end{equation}	
\end{lem}	
\begin{proof}  
	We prove the first equality in (\ref{eq:action-and-the-map}). Let us denote the domain of our holomorphic triangles by $\triangle$, which has vertices $v_0,v_1$ and $v_2$ in clockwise order, and edges $e_0,e_1$ and $e_2$ (so that $e_j$ is the edge opposite to $v_j$).  Define the map $g_{\gamma,\spinct}:C_2\ra C_1$ by 
	\begin{displaymath}
		\begin{split}
			&g_{\gamma,\spinct}(\x_2):=\sum_{\x_1\in\Sbf(H_1,\spinc_1)}\sum_{\Delta\in\pi^0_2(\x_0^{top}\times\x_1\times \x_2)}
			a(\gamma,\Delta)\cdot\wvar^{n_\wpoint(\Delta)}\cdot\zvar^{n_\zpoint(\Delta)}\cdot\y_i, \quad\forall\ \x_2\in\Sbf(H_2,\spinc_2),\quad\text{where} \\
			&a(\gamma,\Delta)=\# \Mcal(\gamma,\Delta)=\#\left\{(u,x)\ \big|\ u\in\Mcal(\Delta),\ x\in\mathrm{Domain}(u),\ \ u(x)\in e_0\times(\gamma\cap\alphas_{0})\right\}.
		\end{split}
	\end{displaymath}	 
	Fix $\x=\x_0\times\x_1\times\x_2\in\Sbf(H)$ and  $\Delta\in\pi^1_2(\x)$ with $\x_0=\x_0^{top}$. The  ends of the $1$-manifold $\Mcal(\gamma,\Delta)$ correspond to degenerations $\Delta=\phi_i\star\Delta'$ of $\Delta$ into a triangle  $\Delta'$ of Maslov index $0$ and a Whitney disk  $\phi_i\in\pi_2(\x_i,\x'_i)$ of index $1$ (for some $i\in\Z/3$), where $\x'_i\in\Sbf_i(H)$ is the juxtaposition point for $\Delta'$ and $\phi_i$. Such degenerations are  modeled on 
	\begin{align}\label{eq:deg-of-moduli-space-action}
		\Big(\coprod_{\substack{i\in\Z/3}}\ \coprod_{\substack{
				\Delta=\Delta'\star\phi_i}}\left(\Mcal(\gamma,\Delta')\times\Mcal(\phi_i)\right)\Big)
		\ \coprod \ \Big(\coprod_{\substack{i\in\{1,2\}}}\ \coprod_{\substack{
				\Delta=\Delta'\star\phi_i}}\left(\Mcal(\Delta')\times\Mcal^0(\gamma,\phi_i)\right)\Big)
	\end{align}	  
	The total number of points in the moduli space of (\ref{eq:deg-of-moduli-space-action}) is thus zero. Since $\x_0^{top}$ is closed, considering such endpoints for all choices of $\Delta\in\pi^1_2(\x)$  gives 
	\begin{equation}\label{eq:action-and-map-2}
		\begin{split}
			d_1(g_{\gamma,\spinct}(\x_2))+g_{\gamma,\spinct}(d_2(\x_2))&=f_{H,\spinct}\left(A^0_\gamma(\x_2)\right)-A^0_\gamma\left(f_{H,\spinct}(\x_2)\right), \quad\forall \ \x_2\in\Sbf(H_2).
		\end{split}
	\end{equation}
	This completes the proof of the first equality. The proof of the second equality is very similar. We only need to notice that for every $w\in\wpoint$, $\Phi^\star_w(\x_0^{top})=0$ in $\HFLs(H_0,\spinc_0)$. 
\end{proof}	

Let us write $\HFLs_i$ for $\HFLs(H_i,\spinc_i)$.
If $\gamma$ is a closed loop on $\HSurf$, we may compute $\HAction_\gamma(\x_2)$ using $A^1_\gamma$. The argument  of Lemma~\ref{lem:homology-action-interactions} then implies that 
$F_{H,\spinct}\circ (A^1_\gamma\otimes I-I\otimes A^1_\gamma)$ is chain homotopic to zero. Lemma~\ref{lem:homology-action-interactions} and the above observation have interesting implications. The inclusion of $Y_{H_i}$ in $X_H$ gives a corresponding map $H_1(Y_{H_i};\Fbb)\ra H_1(X_H;\Fbb)$, which induces the map 
\begin{align*}
	\imath_{i}:\mywedge_{Y_{H_i}}=\mywedge^*H_1(Y_{H_i};\Fbb)\ra \mywedge_{X_H}=\mywedge^*H_1(X_H;\Fbb). 
\end{align*}	
Let $K_i\subset \mywedge_{Y_{H_i}}$ denote the kernel of $\imath_i$. If $[\gamma]$ is in $K_2$,  Lemma~\ref{lem:homology-action-interactions} and the above observation imply that $\fmap^\star_{H,\spinct}$ is trivial on the image of $\HFLs(H_2,\spinc_2)$ under the action of $[\gamma]$. Therefore, $\fmap^\star_{H,\spinct}$ factors through  $\HFLs(H_2,\spinc_2)/K_2$. On the other hand, the same argument implies that  $\Image(\fmap^\star_{H,\spinct})$ lies in the kernel of the action of $K_1$ on $\HFLs(H_1,\spinc_1)$. The quotient map $\HFLs(H_2,\spinc_2)\ra \HFLs(H_2,\spinc_2)/K_2$ factors through $\HFLs(Y_2,\Link_2,\spinct|_{Y_2})$ and that the inclusion $\Ker(K_1)\ra \HFLs(H_1,\spinc_1)$ factors through $\HFLs(Y_1,\Link_1,\spinct|_{Y_1})$. Therefore, we obtain the induced map 
\begin{align*}
	\fmap^\star_{H,\spinct}:\HFLs(Y_2,\Link_2,\spinct|_{Y_1})\ra \HFLs(Y_1,\Link_1,\spinct|_{Y_1}),
\end{align*} 
which is homogeneous with respect to the homological  bi-grading. If  the second homology group of each $Y_i$ is trivial,  the bi-degree of $f^\star_{H,\spinct}$ with respect to  the bigradings $(\grw,\grz)$ on the two sides may be computed (see \cite[Theorem 1.4]{Zemke-2}), and is given by  $(d_{X}^\spinct,d_{X}^{\spinct-\PD[\Surface]}-\chi(\Surface))$, 
where 
\begin{align*}
	d_{X}^\spinct:=\frac{c_1(\spinct)^2-2\chi(X)-3\sigma(X)}{4}.
\end{align*}

Since the action of $\wedgews$ is respected by $\fmap^\star_{H,\spinct}$,
the map $\fmap^\star_{H,\spinct}$ induces a well-defined $\Rings$-homomorphism
\begin{align*}
	\gmap^\star_{H,\spinct}:\HFKs(Y_2,L_2,\spinct|_{Y_2})\ra \HFKs(Y_1,L_1,\spinct|_{Y_1}).	
\end{align*}	
The following theorem follows from  \cite[Theorem~1.2]{AE-2}, or alternatively \cite[Theorem~A]{Zemke-1}, and their proofs:
\begin{thm}\label{thm:invariance-decorated-case}
	Let $H$ 	be a nice $\spinct$-admissible  trisection diagram  compatible with the decorated cobordism $(X,\TSurface):(Y_2,\Link_2)\ra (Y_1,\Link_1)$. 
	Then for $\star\in\{-,\circ,\wedge\}$, the $\Abb^\star$-homomorphism
\begin{align*}
\fmap^\star_{H,\spinct}:\HFLs(H_2,\spinct|_{Y_{H_2}})\ra \HFLs(H_1,\spinct|_{Y_{H_1}})
\end{align*} 
factors through well-defined  $\Abb^\star$-homomorphisms
\begin{align*}
&\fmap_{X,\TSurface,\spinct}^\star=\fmap_{X,\TSurface,\spinct;H}^\star:\HFLs(Y_2,\Link_2,\spinct|_{Y_2})\ra \HFLs(Y_1,\Link_1,\spinct|_{Y_1})\quad\text{and}\\
&\gmap_{X,\TSurface,\spinct}^\star=\gmap_{X,\TSurface,\spinct;H}^\star:\HFKs(Y_2,L_2,\spinct|_{Y_2})\ra \HFKs(Y_1,L_1,\spinct|_{Y_1})
\end{align*} 
which remain invariant under slide diffeomorphisms and stabilizations of the trisection diagram $H$ and are thus  invariants of the decorated cobordism $(X,\TSurface)$ and $\spinct\in\SpinC(X)$.
\end{thm}	

Theorem~\ref{thm:invariance-decorated-case} implies that there is no confusion in denoting 
$\fmap_{X,\TSurface,\spinct;H}^\star$ and $\gmap_{X,\TSurface,\spinct;H}^\star$ by $\fmap_{X,\TSurface,\spinct}^\star$ and $\gmap_{X,\TSurface,\spinct}^\star$, respectively. Therefore, we may conveniently drop the trisection diagram from the notation when the reference to it is not needed.\\ 

Let $(Y,\Link)$ denote a pointed link, and $K$ denote a component of the underlying link $L$, which contains the basepoints $\wpt_1,\zpt_1,\wpt_2,\zpt_2,\ldots,\wpt_k,\zpt_k$ in this order, with $\wpt_j\in\wpoint$ and $\zpt_j\in\zpoint$ for $j=1,\ldots,k$. Set $(X,\Surface)=(Y,L)\times [0,1]$ and assume that $\SCob=\SCob_K$ is a decoration of $\Surface=L\times [0,1]$, which is trivial except on $K\times [0,1]$, where its diving set consists of parallel arcs which partition $K\times [0,1]$ to $2k$ strips such that $(\wpt_j,0)$ is in the same strip as $(\wpt_{j+1},1)$ and $(\zpt_j,0)$ is in the same strip as $(\zpt_{j+1},1)$. This cobordism is called the {\emph{twist cobordism}} associated with $\Link$ and $K$, and is illustrated in Figure~\ref{fig:Twist-Cobordism}. The cobordism map associated with $(Y\times [0,1],\SCob_K)$ is computed by Zemke \cite{Zemke-3}. The following formula, is a reformulation of Zemke's result.

\begin{figure}
	\def\svgwidth{0.38\textwidth}
	{\footnotesize{
			\begin{center}
				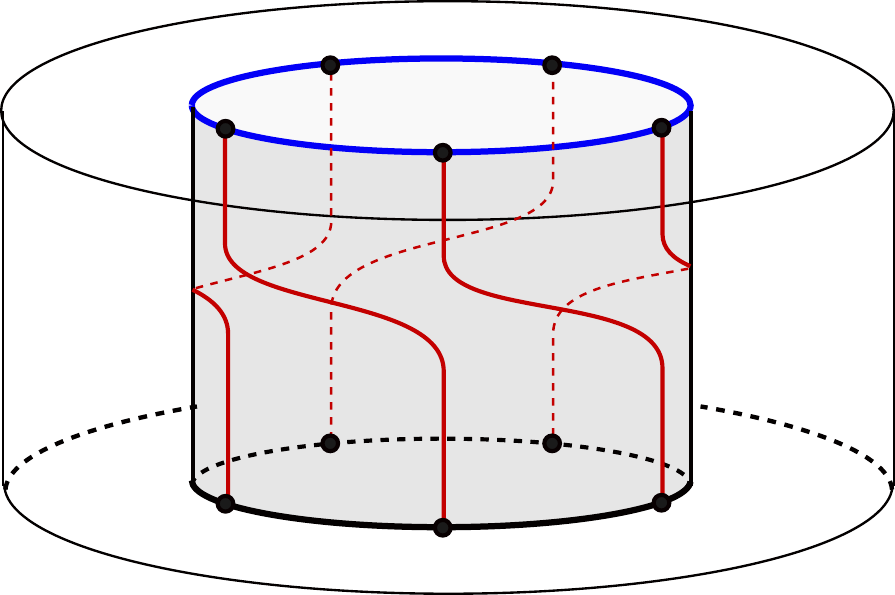
	\end{center}}}
	\caption{The underlying link cobordism of a  twist cobordism is the trivial (product) cobordism.}
	\label{fig:Twist-Cobordism}
\end{figure}

\begin{thm}[Theorem D from \cite{Zemke-3}]\label{thm:twist-formula}
Let $(Y,\Link)$, $K\subset L$, the basepoints $\wpt_1,\ldots,\wpt_k\in\wpoint$ and $\zpt_1,\ldots,\zpt_k\in \zpoint$, and the twist cobordism $(X=Y\times[0,1],\SCob_{\Link,K}):(Y,\Link)\ra (Y,\Link)$ be as above. Then for $\star\in\{-,\circ,\wedge\}$ we have 
\begin{align*}
\fmap^\star_{X,\SCob_{\Link,K},\spinc}=\sum_{I=\{j_1<\cdots<j_m\}\subset\{1,\ldots,k\}}	\Psi^\star_{\zpt_{j_1}}\circ\Phi^\star_{\wpt_{j_1}}\circ\Psi^\star_{\zpt_{j_2}}\circ\Phi\star_{\wpt_{j_2}}\circ\cdots\circ 
\Psi^\star_{\zpt_{j_m}}\circ\Phi^\star_{\wpt_{j_m}}.
\end{align*}	
\end{thm}	

\begin{thm}\label{thm:twist-invariance}
Let $(Y,\Link)$ be a  pointed link with underlying link $L$, $K$	be a component of $\Link$,  $(X,\SCob_{\Link,K})$ denote the twist cobordism associated with $\Link$ and $K$ and $\spinc$ be a fixed $\SpinC$ structure on $Y$. Then
\[\gmap^\star_{X,\SCob_{\Link,K},\spinc}=Id:\HFKs(Y,\Link,\spinc)\ra \HFKs(Y,\Link,\spinc).\]
\end{thm}	
\begin{proof}
Given $\x\in\HFKs(Y,\Link,\spinc)\subset \HFLs(Y,\Link,\spinc)$, it follows that $\x$ is in the kernel of the action of all basepoint maps $\Phi^\star_w$ for $w\in\wpoint$. Therefore, Theorem~\ref{thm:twist-formula} implies that 	$\gmap^\star_{X,\SCob_{\Link,K},\spinc}(\x)=\x$, as the only non-zero term contributing to the sum corresponds to $m=0$ and $I=\emptyset$.
\end{proof}

\subsection{The concordance TQFTs}\label{subsec:perturbation-invariance}
Let us assume that $H=(\HSurf,\alphas_0,\alphas_1,\alphas_2,\wpoint,\zpoint)$ is a nice and $\spinct$-admissible trisection diagram determining  the decorated cobordism (equipped with $\SpinC$ structures) 
\begin{align*}
(X,\TSurface,\spinct):(Y_2,\Link_2,\spinct|_{Y_2})\ra (Y_1,\Link_1,\spinct|_{Y_1}). 
\end{align*}
Choose a $2$-sphere $\HSurf'$ and a triple of oriented great circles $\alpha_0,\alpha_1$ and $\alpha_2$, so that each $\alpha_i$ and $\alpha_j$ intersect in a pair of points, and together they decompose $\HSurf'$ into $8$ triangles. Each $\alpha_i$ decomposes $\HSurf'$ into two disks $A_i$ and $B_i$, where $\alpha_i$ is oriented as the boundary of $A_i$, for $i\in\Z/3$. We further assume that $A_0\cap A_1\cap A_2$ is the (positively oriented) triangle which determines the domain of a triangle class for $(\HSurf',\alpha_0,\alpha_1,\alpha_2)$. Choose the eight basepoints $\wpt,\wpt',\zpt_i,\zpt'_i$, for $i\in\Z/3$ such that 
 \[w\in A_0\cap A_1\cap A_2,\quad w'\in B_0\cap B_1\cap B_2,\quad
 \zpt_i\in  B_i\cap A_{i-1}\cap A_{i+1}\quad\text{and}\quad \zpt'_i\in A_i\cap B_{i-1}\cap B_{i+1}. \]
 Set $\wpoint'=\{w,w'\}$ and $\zpoint'=\{\zpt_0,\zpt_0'\}$ and $H'=(\HSurf',\alpha_0,\alpha_1,\alpha_2,\wpoint',\zpoint')$. We choose a basepoint $z\in \zpoint$, remove the disk neighborhoods $D_z$ and $D_{w'}$ of $z\in \HSurf$ and $w'\in \HSurf'$ respectively, and connect the circle boundaries by a tube of length $\ell$, which is chosen sufficiently large. This gives a surface $\overline{\HSurf}=\HSurf\#_{z,w'}\HSurf'$, which has the same genus $g$ as $\HSurf$. A neighborhood of $\HSurf'-D_{w'}$ in $\overline{\HSurf}$ is illustrated in Figure~\ref{fig:Stabilization-HD}. Let us define 
\begin{align*}
	&\oH=(\overline\HSurf,\oalphas_0,\oalphas_1,\oalphas_2,\oW=\wpoint\cup\wpoint'-\{w'\}, \oZ=\zpoint\cup\zpoint'-\{z\})=H\#_{z,w'} H',\quad\text{where}\quad
\oalphas_i=\alphas_i\cup\{\alpha_i\}.
\end{align*} 
The nice trisection diagram  $\overline{H}$ is then a trisection diagram compatible with  a $1$-perturbation of $(X,\TSurface)$. The belt of the tube which connects $\HSurf$ to $\HSurf'$ is a closed curve $\gamma$, which is determined by a dashed closed curve in Figure~\ref{fig:Stabilization-HD}. 
Let us label the intersection points as in Figure~\ref{fig:Stabilization-HD}. Every   generator $\x_i\in \Sbf(H_i)$  and  every $\epsilon\in\{0,1\}$ give a generator 
$\x_i^\epsilon=\x_i\cup \{x_i^\epsilon\}\in\Sbf(\oH_i)$, so that $\grw(\x_i^\epsilon)=\grw(\x_i)-\epsilon$.
Moreover, we have  
 \begin{align*}
 	d^{\oH}_i\big(\x_i^\epsilon\big)=\big(d^H_i(\x_i)\big)^{\epsilon},\quad\quad\forall\ \  \x_i\in\Sbf(H_i),\  \epsilon\in\{0,1\}.  
 \end{align*}	
 	The following theorem follows from either of \cite[Proposition 6.2]{MO-integer-surgery}, \cite[proof of Theorem 7.3]{AE-2},  or \cite[Proposition 4.28]{Zemke-1}.
 
\begin{figure}
	\def\svgwidth{0.6\textwidth}
	{\footnotesize{
			\begin{center}
				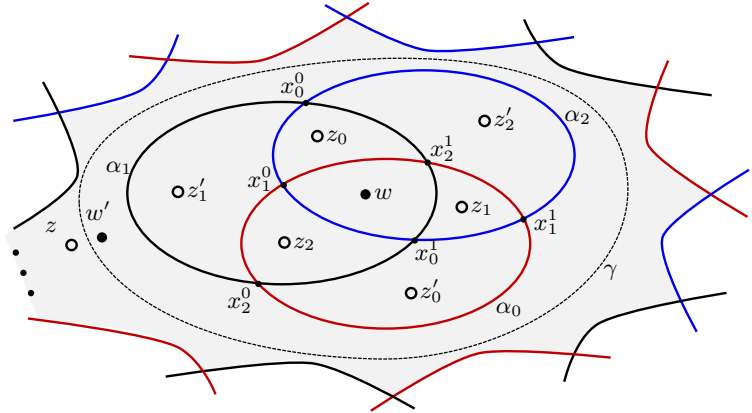
	\end{center}}}
	\caption{A simple perturbation of a trisection diagram.}
	\label{fig:Stabilization-HD}
\end{figure}  

\begin{thm}\label{thm:perturbation-formula}
Let us assume that the decorated cobordism $(X,\TSurface'):(Y_2,\Link_2')\ra (Y_1,\Link_1')$ is obtained from $(X,\TSurface):(Y_2,\Link_2)\ra (Y_1,\Link_1)$ by a simple perturbation. Fix $\spinct\in\SpinC(X)$ and  construct the $\spinct$-admissible trisection diagram $\oH$ for $(X,\TSurface')$ from the $\spinct$-admissible trisection diagram $H$ for $(X,\TSurface)$ as above. If the tube length $\ell$ is sufficiently large,  we have 
	\begin{align}\label{eq:simple-perturbation-map}
		f_{\oH,\spinct}(\x_2^\epsilon))=\left(f_{H,\spinct}(\x_2)\right)^\epsilon,
	\quad\quad\forall\ \ \x_2\in\Sbf(H_2)\ \text{and}\  \epsilon\in\{0,1\}.
	\end{align}
Therefore, 	for $\star\in\{-,\circ,\wedge\}$ the  induced maps 
\begin{align*}
\gmap^\star_{X,\SCob,\spinct},\gmap^\star_{X,\SCob',\spinct}:
\HFKs(Y_2,L_2,\spinct|_{Y_2})\ra 	\HFKs(Y_1,L_1,\spinct|_{Y_1})
\end{align*}
are identical, while 
\begin{align*}
	\fmap^\star_{X,\SCob',\spinct}:
	\HFLs(Y_2,\Link'_2,\spinct|_{Y_2})&=\HFLs(Y_2,\Link_2,\spinct|_{Y_2})\otimes_{\Abb^\star} \big(\Abb^\star\oplus\Abb^\star\llb1,-1\rrb\big)\\
	&\ra 		\HFLs(Y_1,\Link_1,\spinct|_{Y_1})\otimes_{\Abb^\star} \big(\Abb^\star\oplus\Abb^\star\llb1,-1\rrb\big)=\HFLs(Y_1,\Link'_1,\spinct|_{Y_1})
\end{align*}	
is identified as $\fmap_{X,\SCob,\spinct}^\star\otimes Id$. Here, the shift $\llb1,-1\rrb$ is made with respect to the bi-grading $(\grw,\grz)$.
\end{thm}

If $(X,\Surface):(Y,L)\ra (Y',L')$ is a link concordance (i.e. a union of cylinders connecting two links with the same number of components),  we may fix a simple decoration $\SCob$ of $\Surface$ which consists of an arc on each cylinder which connects the two boundary components. Correspondingly, for every $\SpinC$ structure $\spinct\in\SpinC(X)$ we may define the Heegaard-Floer map
\begin{align*}
	\gmap_{X,\Surface,\spinct;\SCob}^\star:=\gmap_{X,\SCob,\spinct}^\star:\HFKs(Y,K,\spinct|_Y)\ra \HFKs(Y',K',\spinct|_{Y'}),\quad\text{for}\ \star\in\{-,\circ,\wedge\}. 	 
\end{align*} 
\begin{thm}\label{thm:invariance-concordance}
Having fixed a concordance $(X,\Surface):(Y,L)\ra (Y',L')$ and a $\SpinC$ structure $\spinct\in\SpinC(X)$, 	the map $\gmap_{X,\SCob,\spinct}^\star$ is  independent of the choice of the decoration $\SCob$ of the surface $\Surface$. Therefore, we obtain the well-defined $\Abb^\star$-homomorphisms
	\begin{align*}
		\gmap_{X,\Surface,\spinct}^\star:\HFKs(Y,L,\spinct|_Y)\ra \HFKs(Y',L',\spinct|_{Y'}),\quad\text{for}\ \star\in\{-,\circ,\wedge\}.	
	\end{align*}	 
\end{thm}
\begin{proof}
	Note that the $\Abb^\star$-homomorphism
	\begin{align*}
		\gmap_{X,\SCob,\spinc}^\star:\HFKs(Y,K,\spinct|_Y)\ra \HFKs(Y',K',\spinct|_{Y'}),\quad\text{for}\ \star\in\{-,\circ,\wedge\}	 
	\end{align*} 	
	may in fact be defined for any decoration $(X,\SCob)$ of the concordance $(X,\Surface)$ (which may include more than one arc on each cylindrical component).
	If the decoration $\SCob'$ of $\Surface$ is obtained from the decoration $\SCob$ by adding an arc on one of the cylinders, it  follows that $(X,\SCob')$ is a simple perturbation of $(X,\SCob)$, and  that
	$\gmap^\star_{X,\SCob',\spinct}=\gmap^\star_{X,\SCob,\spinct}$ by Theorem~\ref{thm:perturbation-formula}.	
	The proof now follows from Theorem~\ref{thm:twist-invariance} and the observation that every two decorations of a concordance may be changed two one another by a sequence of simple perturbations, simple deperturbations, and twists.
\end{proof}	

The outcome of the above discussion is that Heegaard Floer homology gives the TQFTs
\begin{align*}
\HFKs:\LinksC\ra \HFHstar,\quad\text{for}\ \ \star\in\{\circ,\wedge\},
\end{align*}	
which assigns the Heegaard Floer homology group 
\begin{align*}
\HFKs(Y,L)=\bigoplus_{\spinc\in\SpinC(Y)}\HFKs(Y,L,\spinc),\quad\star\in\{\circ,\wedge\},	
\end{align*}	
to an oriented link $(Y,L)$, and the $\Abb^\star$-homomorphism $\HFKs(X,\Surface)=\gmap_{X,\Surface}^\star$ to the a link cobordism $(X,\Surface):(Y,L)\ra (Y',L')$ which is the sum of the maps $\gmap_{X,\Surface,\spinct}^\star$ for  $\SpinC$ structures $\spinct\in\SpinC(X)$. \\

Zemke showed that knot Floer homology provides an obstruction to ribbon concordance \cite{Zemke-Ribbon}. The following theorem follows as a direct byproduct of his result.

\begin{thm}\label{thm:ribbon-concordance}
	With $X=Y\times [0,1]$, let $(X,\Surface):(Y,K_0)\ra (Y,K_1)$ be a ribbon concordance between the knots $K_0$ and $K_1$ in $Y$  and  $(X,\overline\Surface):(Y,K_1)\ra (Y,K_0)$ denote the reverse concordance (which is a co-ribbon concordance). Then
	\begin{align*}
		\gmap_{X,\overline\Surface}^\star\circ	\gmap_{X,\Surface}^\star=Id:
		\HFKs(Y,K_0)\ra  \HFKs(Y,K_0),\quad\text{for}\  \star\in\{-,\circ,\wedge\}.
	\end{align*}		
	In particular, $\gmap_{Y\times [0,1],\Surface}^\star$ is injective for every ribbon concordance $(Y\times[0,1],\Surface)$, while $\gmap_{Y\times [0,1],\Surface}^\star$ is surjective for every co-ribbon concordance $(Y\times[0,1],\overline\Surface)$. 
\end{thm}	

\section{The link cobordism TQFT}\label{sec:cobordism-tqft}
\subsection{Merge, split and point-shift cobordisms}\label{subsec:merge-split}
We can  compute the cobordism map associated with some simple decorated cobordisms. The computations are subsequently used in our proof of the invariance of the cobordism maps over weak Floer homology groups. Let $(Y,\Link)$ denote a pointed link, and $L$ denote the underlying link, which includes the alternating basepoints $\wpoint$ and $\zpoint$.  Assume that $(Y_2,\Link_2)$ is obtained by adding an unlink component to $(Y,\Link)$, while $(Y_1,\Link_1)$ is obtained from $(Y,\Link)$ by adding a pair of adjacent basepoints  on one of the link components (the distinguished component), near some $z\in\zpoint$. In particular, we are assuming $Y_1=Y_2=Y$. We may then consider the {\emph{merge}} cobordism from $(Y_2,\Link_2)$ to $(Y_1,\Link_1)$, in which the unlink component merges  with the distinguished component to create a component with an extra pair of basepoints. This cobordism is illustrated in Figure~\ref{fig:Simple-Cobordisms}(a). The reverse of a merge cobordism is a {\emph{split}} cobordism, which is illustrated in Figure~\ref{fig:Simple-Cobordisms}(b).\\
	\begin{figure}
	\def\svgwidth{0.85\textwidth}
	{\footnotesize{
			\begin{center}
				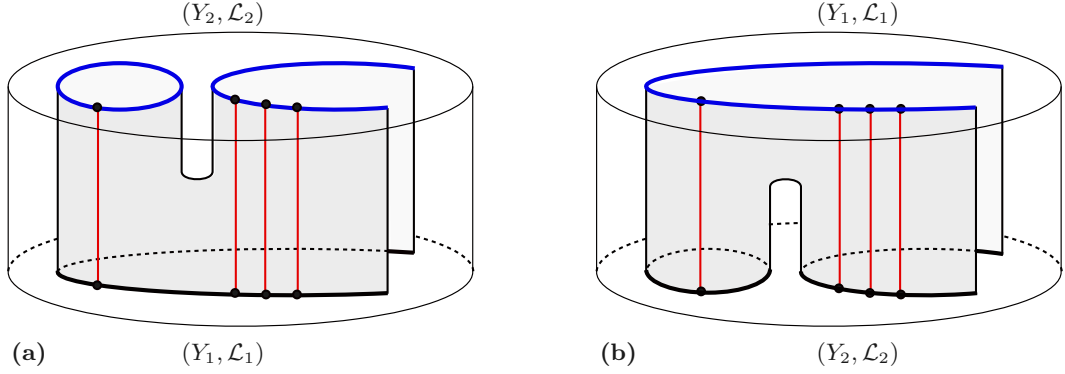
	\end{center}}}
	\caption{(a) A single-pointed unlink component of $\Link_2$ is merged into another link component to form the pointed link $\Link_1$. (b) The reverse cobordism is a split, which places one of the basepoints from a component of $\Link_2$ on a small unlink component in $\Link_1$.}
	\label{fig:Simple-Cobordisms}
\end{figure}  

Let us assume that $(X,\SCob_m):(Y_2,\Link_2)\ra (Y_1,\Link_1)$ is a merge cobordism, which is obtained from the pointed link $(Y,\Link)$ as above. In particular,   $X=Y\times I$. Let $H=(\HSurf,\alphas_0,\alphas_1,\wpoint,\zpoint)$ denote a pointed Heegaard diagram for $(Y,\Link)$, and $z\in\zpoint$ denotes a basepoint on the distinguished component of $\Link$ so that the merging happens near $z$. Using the notation set before Theorem~\ref{thm:perturbation-formula}, we may construct a trisection diagram for the merge cobordism as follows. First, we let $\alphas_2$ denote a collection of simple closed curves which is obtained from $\alphas_1$ by small Hamiltonian isotopies (away from $\wpoint$ and $\zpoint$).
Abusing the notation, we will use the same label $H$ to denote the trisection diagram 
$(\HSurf,\alphas_0,\alphas_1,\alphas_2,\wpoint,\zpoint)$, which corresponds to the trivial decorated cobordism $(Y,\Link)\times I$. Set
\begin{align*}
\oH_m=\Big(\oS=\HSurf\#_{z,w'}\, \oalphas_0,\oalphas_1,\oalphas_2,\oW=\wpoint\cup\{w\},\oZ=\zpoint\cup\{\zpt_2,\zpt'_2\}-\{z\}\Big),\quad\text{where}\ \ \oalphas_i=\alphas_i\cup\{\alpha_i\}.	
\end{align*}	
It is then straight-forward to check that $\oH_m$ is a trisection diagram for $(X,\SCob_m)$ and that if $H$ is $\spinc$-admissible for some $\spinc\in\SpinC(Y)=\SpinC(X)$, so is $\oH_m$. Similarly, the split cobordism $(X,\SCob_s)$ is described by the trisection diagram 
\begin{align*}
	\oH_s=\Big(\oS=\HSurf\#_{z,w'}\, \oalphas_0,\oalphas_2,\oalphas_1,\oW=\wpoint\cup\{w\},\oZ=\zpoint\cup\{\zpt_2,\zpt'_2\}-\{z\}\Big).	
\end{align*}	
 All holomorphic triangles which contribute to  $F_{\oH_m,\spinct}$ (or $F_{\oH_s,\spinct}$) may be described in terms of the holomorphic triangles contributing to the map $F_{H,\spinct}$ associated with the product cobordism $(Y,\Link)\times I$ using \cite[Proposition 5.2]{MO-integer-surgery}. With respect to the bi-grading $(\grw,\grz)$ 
 \begin{equation}\label{eq:merge-split-identification}
\begin{split}
 &\HFLs(Y_1,\Link_1,\spinc)=\HFLs(Y,\Link,\spinc)\otimes_{\Abb^\star}
 \big(\Abb^\star\oplus\Abb^\star\llb1,-1\rrb\big)
 \quad\text{and}\\ 
  &\HFLs(Y_2,\Link_2,\spinc)=\HFLs(Y,\Link,\spinc)\otimes_{\Abb^\star}\big(\Abb^\star\oplus\Abb^\star\llb1,1\rrb\big)
 \end{split} 
 \end{equation}	
for every $\spinc\in\SpinC(Y)$.   Note that if $L_j$ is the underlying link of $\Link_j$, there are natural identifications
\begin{align}\label{eq:merge-split-identification-2}
	\HFKs(Y_1,L_1,\spinc)=\HFKs(Y_2,L_2,\spinc)=\HFKs(Y,L,\spinc).
\end{align}	
Since the map associated with the product decorated cobordism is the identity, the following theorem is implied by \cite[Proposition 5.2]{MO-integer-surgery}.
\begin{thm}\label{thm:merge-split}
Assume that $(Y,\Link)$ is a pointed link and let the merge and split cobordisms 
\begin{align*}
(X,\SCob_m):(Y_2,\Link_2)\ra (Y_1,\Link_1)\quad\text{and}\quad 	
(X,\SCob_s):(Y_1,\Link_1)\ra (Y_2,\Link_2)
\end{align*}
be constructed from $(Y,\Link)$ as above. Then, for every $\SpinC$ structure $\spinc\in\SpinC(X)=\SpinC(Y)$ and $\star\in\{-,\circ,\wedge\}$, under the identifications of (\ref{eq:merge-split-identification}) we have
\begin{align*}
&\fmap_{X,\SCob_m,\spinc}^\star=Id\oplus \zvar\cdot Id:
\HFLs(Y,\Link,\spinc)\oplus 	\HFLs(Y,\Link,\spinc) \llb1,1\rrb\ra 
\HFLs(Y,\Link,\spinc)\oplus 	\HFLs(Y,\Link,\spinc) \llb1,-1\rrb,\\
&\fmap_{X,\SCob_s,\spinc}^\star=\zvar\cdot Id\oplus  Id:
\HFLs(Y,\Link,\spinc)\oplus 	\HFLs(Y,\Link,\spinc) \llb1,-1\rrb\ra 
\HFLs(Y,\Link,\spinc)\oplus 	\HFLs(Y,\Link,\spinc) \llb1,1\rrb.
\end{align*}	
Moreover, under the identification of (\ref{eq:merge-split-identification-2}) we have
\begin{align*}
&\gmap_{X,\SCob_m,\spinc}^\star=Id:\HFKs(Y,L,\spinc)\ra \HFKs(Y,L,\spinc)
\quad\text{and}\quad
\gmap_{X,\SCob_s,\spinc}^\star=\zvar\cdot Id: \HFKs(Y,L,\spinc)\ra \HFKs(Y,L,\spinc).	
\end{align*}	
In particular, $\widehat{\gmap}_{X,\SCob_s,\spinc}:\HFKhh(Y,L,\spinc)\ra \HFKhh(Y,L,\spinc)$ 
is trivial (i.e. zero).
\end{thm}	

	\begin{figure}
	\def\svgwidth{0.95\textwidth}
	{\footnotesize{
			\begin{center}
\begingroup%
  \makeatletter%
  \providecommand\color[2][]{%
    \errmessage{(Inkscape) Color is used for the text in Inkscape, but the package 'color.sty' is not loaded}%
    \renewcommand\color[2][]{}%
  }%
  \providecommand\transparent[1]{%
    \errmessage{(Inkscape) Transparency is used (non-zero) for the text in Inkscape, but the package 'transparent.sty' is not loaded}%
    \renewcommand\transparent[1]{}%
  }%
  \providecommand\rotatebox[2]{#2}%
  \newcommand*\fsize{\dimexpr\f@size pt\relax}%
  \newcommand*\lineheight[1]{\fontsize{\fsize}{#1\fsize}\selectfont}%
  \ifx\svgwidth\undefined%
    \setlength{\unitlength}{1438.15552173bp}%
    \ifx\svgscale\undefined%
      \relax%
    \else%
      \setlength{\unitlength}{\unitlength * \real{\svgscale}}%
    \fi%
  \else%
    \setlength{\unitlength}{\svgwidth}%
  \fi%
  \global\let\svgwidth\undefined%
  \global\let\svgscale\undefined%
  \makeatother%
  \begin{picture}(1,0.32822)%
    \lineheight{1}%
    \setlength\tabcolsep{0pt}%
    \put(0,0){\includegraphics[width=\unitlength,page=1]{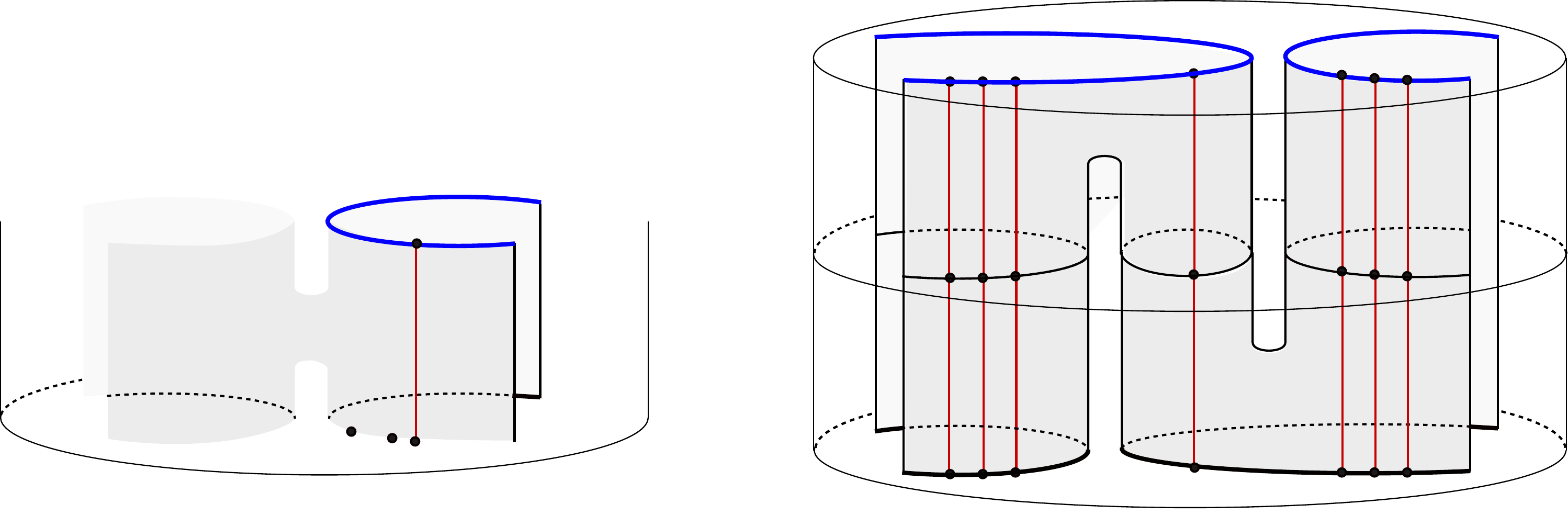}}%
    \put(0.0035414,0.00762704){\color[rgb]{0.10196078,0.10196078,0.10196078}\makebox(0,0)[lt]{\lineheight{1.25}\smash{\begin{tabular}[t]{l}\textbf{(a)}\end{tabular}}}}%
    \put(0,0){\includegraphics[width=\unitlength,page=2]{MS-Cobordism.pdf}}%
    \put(0.51148333,0.00241194){\color[rgb]{0.10196078,0.10196078,0.10196078}\makebox(0,0)[lt]{\lineheight{1.25}\smash{\begin{tabular}[t]{l}\textbf{(b)}\end{tabular}}}}%
    \put(0,0){\includegraphics[width=\unitlength,page=3]{MS-Cobordism.pdf}}%
  \end{picture}%
\endgroup%

	\end{center}}}
	\caption{(a) A point-shift cobordism is obtained from the product cobordism by shifting one basepoint from a link component to another link component. (b) a point-shift cobordism may be decomposed as the composition of a split cobordism and a merge cobordism.}
	\label{fig:MS-Cobordism}
\end{figure}  

\begin{rmk}\label{rmk:merge-split}
If the unlink component which merges to a link component of $\Link\subset Y$ includes more than one pair of basepoints, we obtain a decorated cobordism, which may be described as first modifying a product decorated cobordism to a merge, and then applying several simple perturbations. Since simple perturbations do not change the $\Abb^\star$-homomorphism $\gmap_{X,\SCob,\spinct}^\star$, the last two statements in Theorem~\ref{thm:merge-split} remain valid for these more general types of merge and split cobordisms.
\end{rmk}

Theorem~\ref{thm:merge-split} has an interesting consequence about a decorated cobordism, which we call a {\emph{point-shift}} cobordism. As before, let $(Y_2,\Link_2)$ be a pointed link, so that the underlying link $L$ has at least two components. Let $L^1,L^2$ denote two distinguished components of $L$ so that $L^1$ includes at least $r+s$ basepoints from either of $\wpoint$ and $\zpoint$ (with $r,s>0$). Adding a $1$-handle to $L\times I\subset Y\times I$ which connects the cylinders $L^1\times I$ and $L^2\times I$, we obtain a surface $\Surface$, together with a decoration shifting $r$ adjacent and alternating pairs of basepoints from $L^1\times\{0\}$ to $L^2\times\{1\}$, as illustrated in Figure~\ref{fig:MS-Cobordism}(a). This gives a decorated cobordism 
\begin{align*}
(X,\SCob_{ps}):(Y_2,\Link_2)\ra (Y_1,\Link_1),	
\end{align*}	   
where $Y_1=Y$, the underlying link in $\Link_1$ is $L$, and $r$ pairs of basepoints are moved from $L^1$ to $L^2$. The decorated cobordism $\SCob_{ps}$ is then called a point-shift cobordism. If we remove the $r$ pairs of basepoints which are shifted from one component to the other components from $\Link_1$ or $\Link_2$, we obtain a pointed link $\Link$ in $Y$. It follows that 
\begin{align}\label{eq:HF-decomposition}
\HFLs(Y_i,\Link_i,\spinc)=\HFLs(Y,\Link,\spinc)\otimes_{\Abb^*}\big(\Abb^*\oplus \Abb^*\llb1,-1\rrb\big)^{r},\quad\text{for}\ \ i=1,2,\ \star\in\{-,\circ,\wedge\}.
\end{align}	
 Let us  denote the subset $\HFLs(Y,\Link,\spinc)\otimes_{\Abb^\star}(\Abb^\star)^{r}$ of $\HFLs(Y_i,\Link_i,\spinc)$ by $\HFLs_i(Y,\Link,\spinc)$. More precisely,  we choose $\HFLs_i(Y,\Link,\spinc)$ equal to  the common kernel of the action of all basepoint maps $\Phi^\star_w$ such that $w$ belongs to the set of $r$ basepoints which are removed. The subset $\HFLs_i(Y,\Link,\spinc)$ of $\HFLs(Y_i,\Link_i,\spinc)$ is then canonically determined and identified with $\HFLs(Y,\Link,\spinc)$. It also follows that $\HFKs(Y,L,\spinc)$ is a subset of  $ \HFLs_i(Y,\Link,\spinc)$.  \\

The point-shift cobordism $(X,\SCob_{ps})$ is the composition of a generalized split cobordism and a generalized merge cobordism (defined in Remark~\ref{rmk:merge-split}), in the form
\begin{align*}
	(X,\SCob_{ps})=(X_m,\SCob_{m})\circ(X_s,\SCob_{s}):(Y_2,\Link_2)\ra (Y_1,\Link_1),	\quad
	\text{where}\ \ X=X_m=X_s=Y\times I,
\end{align*}	
as illustrated in Figure~\ref{fig:MS-Cobordism}(b). Theorem~\ref{thm:merge-split} implies the following proposition.
\begin{prop}\label{prop:point-shift}
Let $(X,\SCob_{ps}):(Y_2,\Link_2)\ra (Y_1,\Link_1)$ denote a point-shift cobordism as  above, with $Y_1=Y_2=Y$ and $X=Y\times I$, where $\Link_1$ and $\Link_2$ are decorations of the same link $L$. Let $\Link$ be a pointed link which is obtained from either of $\Link_1$ or $\Link_2$ by removing the shifted basepoints, and for $\spinc\in\SpinC(Y)$ and $\star\in\{-,\circ,\wedge\}$ construct  $\HFLs_i(Y,\Link,\spinc)\subset \HFLs(Y_i,\Link_i,\spinc)$ as above. Then $\fmap_{X,\SCob_{ps},\spinc}^\star$ maps $\HFLs_2(Y,\Link,\spinc)$ to $\HFLs_1(Y,\Link,\spinc)$. Moreover,
\begin{align*}
	&\fmap_{X,\SCob_{ps},\spinc}^\star=\zvar\cdot Id:\HFLs(Y,\Link,\spinc)=\HFLs_2(Y,\Link,\spinc)\ra \HFLs_1(Y,\Link,\spinc)=\HFLs(Y,\Link,\spinc)\quad\text{and}\\
	&\gmap_{X,\SCob_{ps},\spinc}^\star=\zvar\cdot Id:\HFKs(Y,L,\spinc)\ra \HFKs(Y,L,\spinc),\quad\quad\forall\ \ \spinc\in\SpinC(Y),\ \star\in\{-,\circ,\wedge\}.
\end{align*}	
In particular, $\widehat{\fmap}_{X,\SCob_{ps},\spinc}$ and $\widehat{\gmap}_{X,\SCob_{ps},\spinc}$ are both trivial (i.e. zero).
\end{prop}	

\subsection{Invariance under perturbations; the general case}\label{subse:perturbation-general-case}
In Section~\ref{subsec:perturbation-invariance} we studied the invariance of the the link cobordism map under simple perturbations of the decoration on the cobordism. In this section, we would like to consider the more general perturbations (the ones that are not necessarily simple).
\begin{thm}\label{thm:perturbation-invariance-general}
Suppose that  the decorated link cobordism $(X,\SCob):(Y_2,\Link_2)\ra (Y_1,\Link_1)$ 
is obtained from the decorated link cobordism $(X,\SCob')$ by a perturbation and let $L_j$ denote the underlying link in $\Link_j$ for $j=1,2$. Then the induced maps
\begin{align*}
\gmap_{X,\SCob,\spinct}^\star,\gmap_{X,\SCob',\spinct}^\star:\HFKsw(Y_2,L_2,\spinct|_{Y_2})\ra \HFKsw(Y_1,L_1,\spinct|_{Y_1})
\end{align*}
are identical for every $\spinct\in\SpinC(X)$ and $\star\in\{-,\circ\}$.
\end{thm}
\begin{proof}
Fix the $\SpinC$ structure $\spinct\in\SpinC(X)$ and drop it from notation for convenience. 	
Further assume that the cell decomposition associated with $\SCob$ is obtained from the cell decomposition associated with $\SCob'$ by adding an edge (connecting a blue vertex $v$ to a black vertex $v'$) in a $2m$-gon cell $P$. The dashed red curve  inside the red $8$-gon in Figure~\ref{fig:General-Perturbation}(a) illustrates one such perturbation. Correspondingly, we may consider a second edge  of $P$ (connecting $v$ to a neighboring black vertex $v''$), which determines a simple perturbation $(X,\SCob'')$ of $(X,\SCob)$,  as illustrated in Figure~\ref{fig:General-Perturbation}(b).\\

\begin{figure}
	\def\svgwidth{0.6\textwidth}
	{\footnotesize{
			\begin{center}
				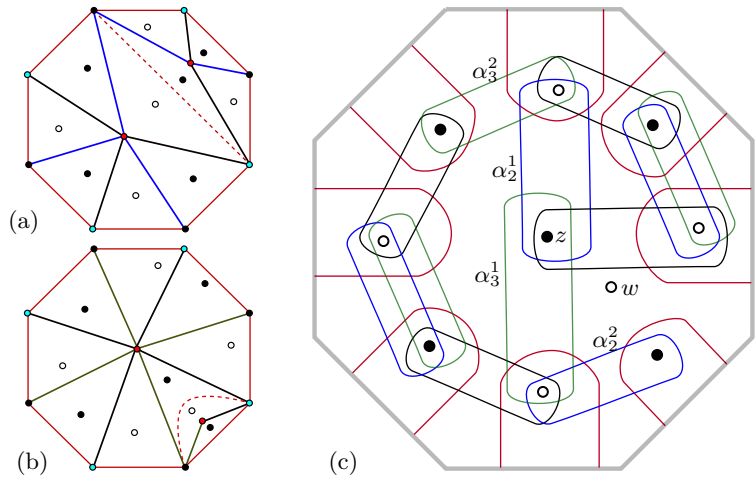
	\end{center}}}
	\caption{(a,b) Associated with a perturbation,  a  simple perturbation is constructed by changing one (red) edge in the cell-decomposition. (c) Corresponding to these two perturbations, we obtain a Heegaard quadruple 
		$H=(\HSurf,\alphas_0,\alphas_1,\alphas_2,\alphas_3,\wpoint,\zpoint)$ so that $H_2$ and $H_3$ determine the simple perturbation and the initial perturbation, respectively.}
	\label{fig:General-Perturbation}
\end{figure}  

 There is a pointed Heegaard $4$-tuple  
$H=(\HSurf,\alphas_0,\alphas_1,\alphas_2,\alphas_3,\wpoint,\zpoint)$
which may be used to describe both perturbations at the same time, as described below and illustrated in Figure~\ref{fig:General-Perturbation}(c) (the curves in $\alphas_0,\alphas_1,\alphas_2$ and $\alphas_3$ are colored red, black, blue and green, respectively). The two perturbations determine two different triangulations of $P$, while they are identical away from $P$, as illustrated in Figure~\ref{fig:General-Perturbation}(a,b). The edges of the first triangulation of $P$ are colored black and blue, while the edges of the second triangulation are colored black and green.
An open subset of the Heegaard surface $\HSurf$ is identified with $P$. 
 We may  choose $\alphas_0$ so that the only intersections of $\alphas_0$ with $P$ are $2m$ half circles which slightly enter $P$ from each one of its $2m$ red edges. Each one of the corresponding $2m$ curves in $\alphas_0$ grabs one basepoint, while $m$ of these basepoints are from $\wpoint$ and $m$ are from $\zpoint$. Moreover, a pair of extra basepoints $w\in\wpoint$ and $z\in \zpoint$ are placed closer to the center of $P$ in $\HSurf$.
The curves in $\alphas_i$ for $i=1,2,3$ are either outside $P$, or completely included inside $P$. In fact, there are precisely $m$ curves from either of these three collections which are included in $P$, and are denoted by $\alpha_i^1,\ldots,\alpha_i^m$, for $i=1,2,3$. Each  curve $\alpha_1^j$ is disjoint from the blue and green edges, cuts precisely one black edge from each triangle twice and separates a pair of basepoints (one from $\wpoint$ and the other from $\zpoint$) from the rest of basepoints. Similarly, the curve $\alpha_2^j$ is disjoint from the black and green edges, cuts a particular blue edge in a pair of points and separates a basepoint from $\wpoint$ and a basepoint from $\zpoint$ from the rest of basepoints, while the curve $\alpha_3^j$ is disjoint from the black and blue edges, cuts a particular blue edge in a pair of points and separates a basepoint from $\wpoint$ and a basepoint from $\zpoint$ from the rest of basepoints. We further assume that $z\in\zpoint$ is not grabbed by any of $\alpha_i^j$ for $i=1,2,3$ and $j=1,\ldots,m$. We may then complete this data (on $P$) to the Heegaard quadruple $H$, so that for $j\in\Z/4$, the trisection diagram $H_j$, obtained by removing $\alphas_j$ from $H$, is compatibles with the decorated cobordisms $(X,\SCob')$ and $(X,\SCob'')$ for $j=3,2$, respectively. Note that we may choose the labels so that each curve in $\alphas_3-\{\alpha_3^1,\alpha_3^2\}$ is a Hamiltonian isotope of a corresponding curve in  $\alphas_2-\{\alpha_2^1,\alpha_2^2\}$. The construction of this quadruple is illustrated in Figure~\ref{fig:General-Perturbation}(c).\\

For  every $i,j\in\Z/4$, the diagram $H_{i,j}=(\HSurf,\alphas_i,\alphas_j,\wpoint,\zpoint)$ determines the chain complex $(C_{i,j},d_{i,j})$ which gives the homology groups  $\HFLs_{i,j}$ for $\star\in\{-,\circ,\wedge\}$. The latter homology group is a  bi-graded $\Abb^\star$-module. 
Note that $\HFLs_{1,2},\HFLs_{1,3}$ and $\HFLs_{2,3}$ are Floer homology groups associated with unlinks inside connected sums of $S^1\times S^2$s, and are therefore generated over $\Abb^\star$ by a distinguished  top generators (with respect to the homological grading $\grw$) under the basepoint action of $\wedgews$. Denote the corresponding generators by 
\begin{align*}
\theta^\star_{1}\in\HFLs_{2,3},\quad \theta^\star_{2}\in\HFLs_{1,3}
\quad\text{and}\quad \theta^\star_{3}\in\HFLs_{1,2}.
\end{align*}
The trisection diagrams  $H_0,H_1,H_2$ and $H_3$ (respectively) determine the ${\Abb^\star}$-homomorphisms	   
\begin{align*}
&\fmap^\star_0:\langle\theta^\star_{3}\otimes\theta^\star_{1}\rangle_{\wedgews}
=\HFLs_{1,2}\otimes_{\Abb^\star}\HFLs_{2,3}\ra \HFLs_{1,3}=\langle\theta^\star_{2}\rangle_{\wedgews},\\
&\fmap^\star_1:\HFLs_{0,2}=\HFLs_{0,2}\otimes_{\Abb^\star}\langle\theta^\star_{1}\rangle_{\Abb^\star}
\subset\HFLs_{0,2}\otimes_{\Abb^\star}\HFLs_{2,3}\ra\HFLs_{0,3},\\
&\fmap^\star_2:\HFLs_{0,1}=\HFLs_{0,1}\otimes_{\Abb^\star}\langle \theta^\star_2\rangle_{\Abb^\star} \subset \HFLs_{0,1}\otimes_{\Abb^\star}\HFLs_{1,3}\ra\HFLs_{0,3}
\quad\quad\quad\text{and} \\
&\fmap^\star_3:\HFLs_{0,1}=\HFLs_{0,1}\otimes_{\Abb^\star}\langle\theta^\star_3\rangle_{\Abb^\star}
\subset \HFLs_{0,1}\otimes_{\Abb^\star}\HFLs_{1,2}\ra\HFLs_{0,2}.
\end{align*}	   
Moreover, considering different possible degenerations of a holomorphic quadrilateral of Maslov index $0$ and using the standard arguments of Floer theory, we obtain the relation
\begin{equation}\label{eq:support-comparison}
	\fmap^\star_2\left(\x\otimes\fmap^\star_0(\theta^\star_{3}\otimes\theta^\star_{1})\right)
	=\fmap^\star_1\left(\fmap^\star_3(\x)\right),\quad\forall\ \x\in\HFLs_{0,1}.
\end{equation}	
The crucial observation is that
$H_0$ corresponds to a point-shift cobordism, from a pointed unlink in a connected sum of $S^2\times S^1$s to itself, which moves $(k-1)$ pairs of adjacent and alternating basepoints in $\wpoint$ and $\zpoint$ on an unlink component hosting $k$ pairs of basepoints to another unlink component hosting $m-k+1$ pairs of basepoints. Therefore, 
\begin{align*}
\fmap^\star_0(\theta^\star_{3}\otimes\theta^\star_{1})=\zvar\cdot \theta^\star_2
\quad\overset{\text{by (\ref{eq:support-comparison})}}{\Longrightarrow}\quad
\zvar	\fmap^\star_2\left(\x\right)
=\fmap^\star_1\left(\fmap^\star_3(\x)\right),\quad\forall\ \x\in\HFLs_{0,1}.	
\end{align*}	
Moreover, $H_1$ corresponds to shifting the basepoints $w$ and $z$ from one component to another component, and correspondingly  changing a decoration $\Link_1$ of $L_1$ to the decoration $\Link_1'$ of $L_1$. Let us denote the decoration obtained by removing $w$ and $z$ from $\Link_1$ (or equivalently, from $\Link'_1$) by
 $\Link_1''$. It then follows from Proposition~\ref{prop:point-shift} that $\HFLs(Y_1,\Link'',\spinc_1)$ is naturally a subset of both $\HFLs(Y_1,\Link_1,\spinc_1)$ and $\HFLs(Y_1,\Link_1',\spinc_1)$ (for every $\spinc_1\in\SpinC(Y_1)$), as the kernel of the action of $\Phi_w$. Moreover, $\fmap^\star_1$ induces a map
\begin{align*}
\fmap^\star_1=\zvar\cdot Id:\HFLs(Y_1,\Link'')\ra \HFLs(Y_1,\Link'').
\end{align*}
 Moreover, it follows that $\HFKs(Y_1,L_1)\subset  \HFLs(Y_1,\Link''_1)$.
If $\x\in\HFKsw(Y_2,L_2)\subset\HFLs_{0,1}$, it follows that $\x=\zvar\y$ for some $\y\in\HFLs_{0,1}$ and that
\begin{align*}
\gmap_{X,\SCob'}(\x)=
\zvar\fmap_{X,\SCob'}(\y)
=\zvar\cdot\fmap_2(\y)
=\fmap_1(\fmap_3(\y))=\zvar\cdot \fmap_3(\y)
=\zvar\fmap_{X,\SCob''}(\y)=\gmap_{X,\SCob''}(\x).	
\end{align*}	
This completes the proof of the theorem.
\end{proof}	

\subsection{Weak Heegaard-Floer groups and the link cobordism TQFT}\label{subsec:link-tqft}
Let us assume that $(X,\Surface):(Y,L)\ra (Y',L')$ is a link cobordism and that $\SCob$ and $\SCob'$ are two decorations of $\Surface$. Note that either of these two decorations give marked points on either of $L$ and $L'$, giving them the structure of a pointed link.
Theorem~\ref{thm:perturbation-invariance-general} implies that although $\gmap^\star_{X,\SCob,\spinct}$ and $\gmap^\star_{X,\SCob',\spinct}$ may be different  as maps from $\HFKs(Y,L,\spinct|_{Y})$ to $\HFKs(Y',L',\spinct|_{Y'})$ (for $\spinct\in\SpinC(X)$), their restrictions to  $\HFKsw(Y,L,\spinct|_{Y})$ are the same. This is probably our most crucial observation.
Theorem~\ref{thm:uniqueness-of-trisection}, Theorem~\ref{thm:invariance-decorated-case}, Theorem~\ref{thm:twist-invariance} and 
Theorem~\ref{thm:perturbation-invariance-general}   imply the following conclusion.
\begin{thm}\label{thm:invariance}
Associated with every link cobordism $(X,\Surface):(Y,L)\ra (Y',L')$ and every $\SpinC$ structure $\spinct\in\SpinC(X)$ are well-defined {\emph{Heegaard-Floer cobordism maps}}  
\begin{align*}
\gmap_{X,\Surface,\spinct}^\star:\HFKsw(Y,L,\spinct|_{Y})\ra \HFKsw(Y',L',\spinct|_{Y'}),\quad\text{for}\ \ \star=-,\circ.	
\end{align*}		
If $(X,\Surface)=(X_1,\Surface_1)\circ (X_2,\Surface_2)$ is the composition of two link cobordisms $(X_j,\Surface_j)$, $j=1,2$, we have
\begin{align*}
\gmap_{X_1,\Surface_1,\spinct_1}^\star\circ\gmap_{X_2,\Surface_2,\spinct_2}^\star=
\sum_{\substack{\spinct\in\SpinC(X)\\ \spinct|_{X_j}=\spinct_j,\ j=1,2}}\gmap_{X,\Surface,\spinct}^\star,\quad\forall\ \spinct_1\in\SpinC(X_1),\ \spinct_2\in\SpinC(X_2),\ \star\in\{-,\circ\}.	
\end{align*}	
\end{thm}	

Theorem~\ref{thm:invariance} provides us with the (hat) Heegaard-Floer link  TQFT
\begin{align*}
\HFKhw:\Links\ra\HFHAhat,	
\end{align*}	
which is given by associating the $\Ringh$-module 
\begin{align*}
\HFKhw(Y,L)=\sum_{\spinc\in\SpinC(Y)}\HFKhw(Y,L,\spinc)
\end{align*}
to a link $(Y,L)$ and the $\Ringh$-homomorphism 
\begin{align*}
\HFKhw(X,\Surface):=\gmap_{X,\Surface}^\circ=\sum_{\spinct\in\SpinC(X)}\gmap_{X,\Surface,\spinct}^\circ: \HFKhw(Y,L)\ra \HFKhw(Y',L')
\end{align*}
to every link cobordism $(X,\Surface):(Y,L)\ra (Y',L')$.

\section{Basic properties and computations of Heegaard-Floer link TQFT}\label{sec:properties}
\subsection{Weak Heegaard-Floer homology groups: some examples}\label{subsec:examples} 
Before proceeding with a few results about the TQFTs  $\HFKmw$ and $\HFKhw$, let us examine the  groups $\HFKmw(S^3,K)$ and $\HFKhw(S^3,K)$ in a few examples, which are  borrowed from \cite[Section 5]{AE-unknotting}. Unlike the ordinary Heegaard-Floer groups, these groups are not necessarily non-trivial for non-trivial knots. The computations illustrate how some of the complexity of Heegaard-Floer groups is lost in exchange for making the cobordism maps independent of the decoration.  

\begin{ex}\label{ex:torus-knot}
	Let $K=T_{p,q}$ be the $(p,q)$ torus knot with $0<p<q$  and  
	\[\Delta_{K}(t)=t^{-\frac{(p-1)(q-1)}{2}}\frac{(t^{pq}-1)(t-1)}
	{(t^p-1)(t^q-1)}=\sum_{i=0}^{2n}(-1)^it^{a_i}\]
	be the symmetrized Alexander polynomial of $K$, for   	
	integers $a_0>\cdots>a_{2n}$ (with $a_i=-a_{2n-i}$). Set $b_{0}=0$ and $b_i=-2\sum_{j=1}^i(-1)^ja_j$ for $i=1,\ldots,2n$. For generators $\x_i$ in bi-grading $(-b_{2i},-b_{2(n-i)})$ (with $i=0,\ldots,n$), we then find
	\begin{align*}
		\HFKm(K)=\frac{\big\langle \x_0,\ldots,\x_n\big\rangle_\Ringm}{\left\langle \zvar^{a_{2i-1}-a_{2i}}\x_{i}-\wvar^{a_{2i-2}-a_{2i-1}}\x_{i-1}\ \big|\ i=1,\ldots,n\right\rangle_\Ringm}.
	\end{align*}
	Therefore, $\HFKmw(K)=\HFKm(K)\llb0,2\rrb$. Moreover, setting $\y_i=\zvar\cdot\x_i$, we have 
	\begin{align*}
		\HFKhw(K)=\frac{\big\langle \y_0,\ldots,\y_n\big\rangle_\Ringh}{\left\langle \zvar^{a_{2i-1}-a_{2i}-1}\y_{i}\ \big|\ i=1,\ldots,n\right\rangle_\Ringh}=\Ringh\llb0,2a_0+2\rrb\oplus
		\Big(\bigoplus_{i=1}^{n}\frac{\Ringh\llb b_{2i}, b_{2(n-i)}+2\rrb}
		{\langle\zvar^{a_{2i-1}-a_{2i}-1}\rangle_{\Ringh}}\Big).
	\end{align*}		
Note that $a_0=(p-1)(q-1)/2$ is the Seifert genus of $T_{p,q}$ which is also equal to $\tau(T_{p,q})$. We may also compute $\HFLhh(K)$ and $\HFKhh(K)$. For this purpose, suppose that $\z_i$ is a generator in bi-grading $(-b_{2i-1},-b_{2(n-i)+1})$, where
$b_{2i-1}:=b_{2i}-1$  for $i=1,\ldots,n$.
\begin{align*}
\HFLhh(K)&=\big\langle \x_0,\ldots,\x_n,\z_1,\ldots,\z_n\big\rangle_\Fbb 
=\Big(\bigoplus_{i=0}^n \Fbb\llb b_{2i},b_{2(n-i)}\rrb\Big)\oplus
\Big(\bigoplus_{j=1}^n\Fbb\llb b_{2j-1},b_{2(n-j)+1}\rrb\Big)
\quad\text{and}\\
\HFKhh(K)
&=\Big(\bigoplus_{i=0}^n \Fbb\llb b_{2i},b_{2(n-i)}\rrb\Big)\oplus
\Big(\bigoplus_{j: a_{j-1}-a_j>1}\Fbb\llb b_{2j-1},b_{2(n-j)+1}\rrb\Big).
\end{align*}	
\end{ex}

\begin{ex}
	If $K$ is the figure $8$ knot, as discussed in \cite[Example 5.6]{AE-unknotting},
	$\HFKs(K)$ is generated by two generators $\x$ and $\y$ in bi-degree $(0,0)$, while $\wvar\y$ and $\zvar\y$ are both zero. In particular, $\HFKh(K)$ is different from the knot Floer homology group $\HFLh(K)$ (which has $3$ generators $\x,\y,\z$, with $\zvar\y=0$, $\zvar\z=0$ and $\Phi(\z)=\y$). 
	Moreover, $\HFKsw(K)$  is generated by $\zvar\x$ and we thus have
	\[\HFKsw(K)=\HFKsw(U_1)=\Abb^\star\llb0,2\rrb,\quad\star=-,\circ.\] 
\end{ex}

\begin{ex}\label{ex:alternating}
	Let  $K$ be a pointed alternating knot with $\tau(K)>0$.	Since alternating knots  have simple knot Floer chain complexes, as discussed in \cite[Example 5.7]{AE-unknotting}, it follows   that 
	\begin{align*}
		&\HFKs(K)=\HFsTorsion(K)\oplus\HFsTF(K),\quad\HFhTF(K)=\Ringh\llb0,2\tau(K)\rrb\\
		&\text{and}\quad\quad
		\HFmTF(K)=\left\langle\wvar^{i}\zvar^{\tau(K)-i}\ \big|\ i=0,1,\ldots,\tau(K)\right\rangle_\Ringm.	
	\end{align*}	
Here $\HFsTorsion(K)$ and $\HFsTF(K)=\HFKs(K)/\HFsTorsion(K)$ denote the torsion and the torsion-free part of $\HFKs(K)$, respectively. 	
	Moreover, 
	\begin{align*}
		\HFKhw(K)=\Ringh\llb0,2\tau(K)+2\rrb\quad\text{and}\quad
		\HFKmw(K)=\left\langle\wvar^{i}\zvar^{\tau(K)-i+1}\ \big|\ i=0,1,\ldots,\tau(K)\right\rangle_\Ringm\llb0,2\rrb.
	\end{align*}		
\end{ex}

\begin{figure}
	\begin{center}	
		\includegraphics[scale=0.63]{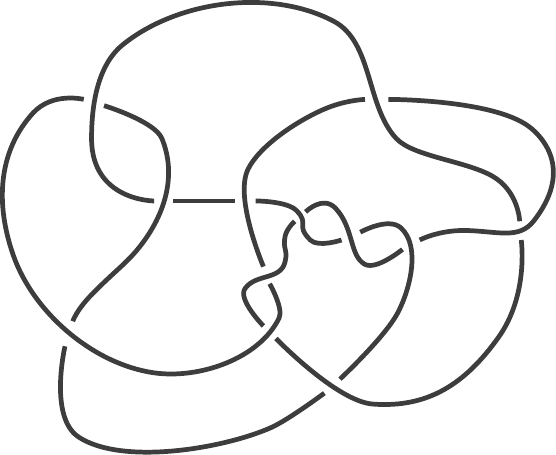}\quad\quad
		\includegraphics[scale=0.35]{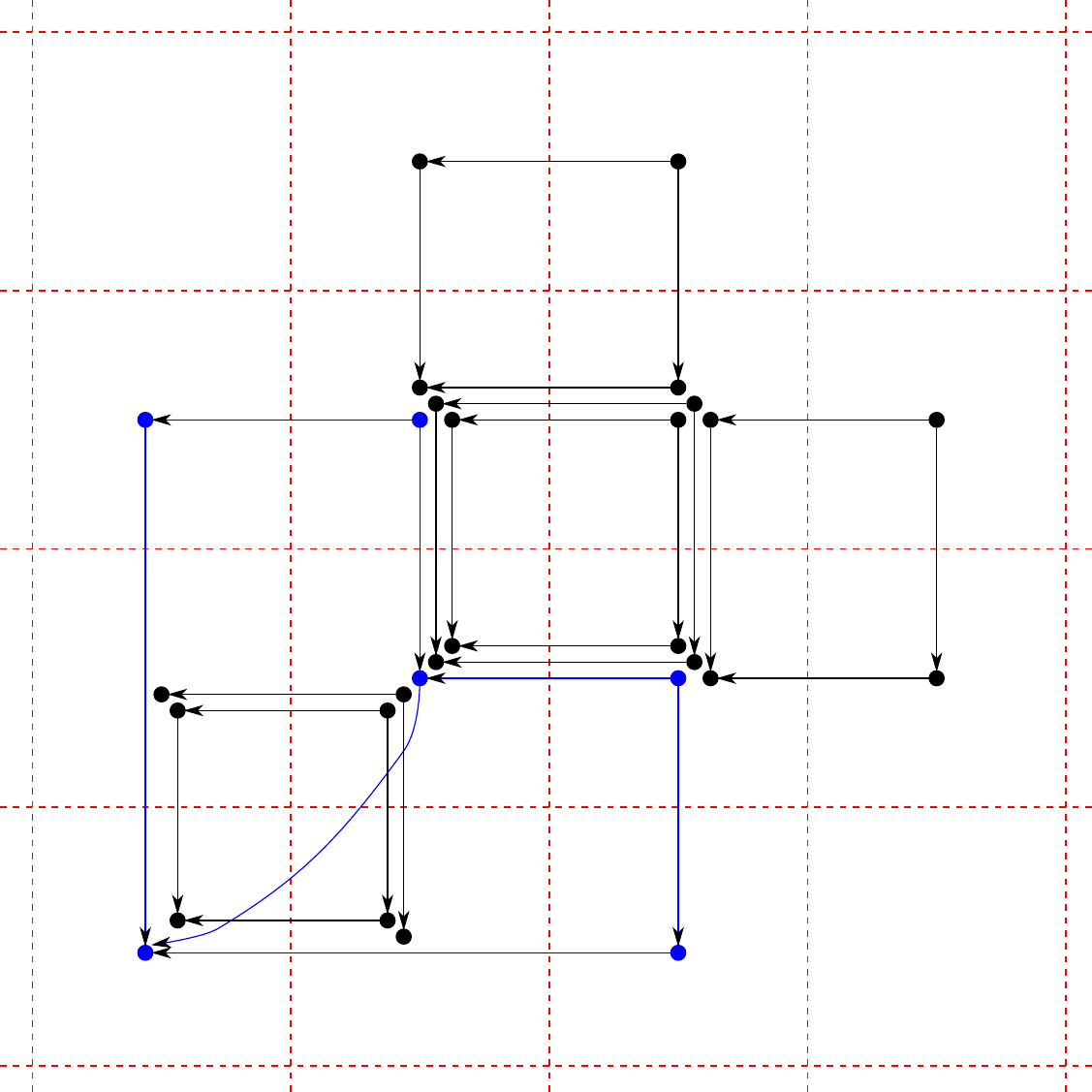}
		\caption{
			The knot $12n_{404}$ and its corresponding chain complex.
		}\label{fig:12n404}
	\end{center}
\end{figure}  
\begin{ex}
	The $(1,1)$ knot $K=12n_{404}$ is illustrated in 
	Figure~\ref{fig:12n404} (left).  The corresponding chain complex $\CFLm(K)$ is discussed in \cite[Example 5.10]{AE-unknotting}, and is  illustrated in Figure~\ref{fig:12n404} (right). Each dot represents a generator of $\CFLm (K)$. An arrow which 
	connects a dot corresponding to a generator $\x$ to a dot 
	representing a generator $\y$ and cuts $i$ vertical lines and 
	$j$ horizontal lines corresponds to the contribution of 
	$\wvar^i\zvar^j\y$ to $d(\x)$. We then have 
	\begin{align*}
	\HFKm(12n_{404})=\langle\zvar,\wvar\rangle_\Ringm\oplus\left(\frac{\Ringm}{\langle\wvar,\zvar\rangle}\right)^5\oplus \frac{\Ringm}{\langle\wvar^2,\wvar\zvar,\zvar^2\rangle}\ \ \text{and}\ \ 
		\HFKmw(12n_{404})=\langle\zvar,\wvar\rangle_\Ringm\oplus\frac{\Ringm}{\langle\wvar,\zvar\rangle},
	\end{align*}		
where the computations are  valid up to shifts in the homological bi-gradings. A similar computation implies that $\HFKhw(12n_{404})=\Ringh\oplus\frac{\Ringh}{\langle\zvar\rangle}$.
\end{ex}

\begin{ex}
	Let $K=T_{2,3;2,-1}$ and $K'=-T_{2,3;2,-3}$ denote the $(2,-1)$ cable of the torus knot $T_{2,3}$, and the mirror image of the $(2,-3)$ cable of it, respectively. The chain complexes associated 
	with $K$ and $K'$ are  illustrated in Figure~\ref{fig:TCable-1} on the left and right, respectively.   It follows from the computations of \cite[Section 5]{AE-unknotting} that 
	\begin{align*}
		& \HFKm(T_{2,3;2,-1})=\HFmTF(T_{2,3;2,-1})\oplus \HFmTorsion(T_{2,3;2,-1})
		=\langle \wvar,\zvar\rangle_\Ringm\oplus 
		\frac{\langle \wvar,\zvar\rangle_\Ringm}
		{\langle \wvar^2,\zvar^2 \rangle_\Ringm}\\
		& \HFKm(-T_{2,3;2,-3})=\HFmTF(-T_{2,3;2,-3})\oplus \HFmTorsion(-T_{2,3;2,-3})
		=\langle \wvar,\zvar\rangle_\Ringm\oplus
		\frac{\Ringm}{\langle \wvar,\zvar^2\rangle_\Ringm}\oplus
		\frac{\Ringm}{\langle\wvar^2,\zvar\rangle_\Ringm},\\
		&\Rightarrow\quad 
		\HFKmw(T_{2,3;2,-1}),\HFKmw(-T_{2,3;2,-3})=\langle \wvar,\zvar\rangle_\Ringm\oplus 
		\frac{\Ringm}{\langle \wvar,\zvar\rangle_\Ringm}.
	\end{align*}
Again, the identifications are up to shifts in homological bi-grading.
It follows from  a similar computation that $\HFKhw(T_{2,3;2,-1}),\HFKhw(-T_{2,3;2,-3})=\Ringh\oplus\frac{\Ringh}{\zvar}$.
	In particular, both of these cables of $T_{2,3}$  share the same weak Heegaard Floer homology group (upto a shift in homological bi-grading), although $\HFKm(K)$ and $\HFKm(K')$ are completely different.
\end{ex}
\begin{figure}
	\begin{center}	
		\def\svgwidth{4.6cm}
		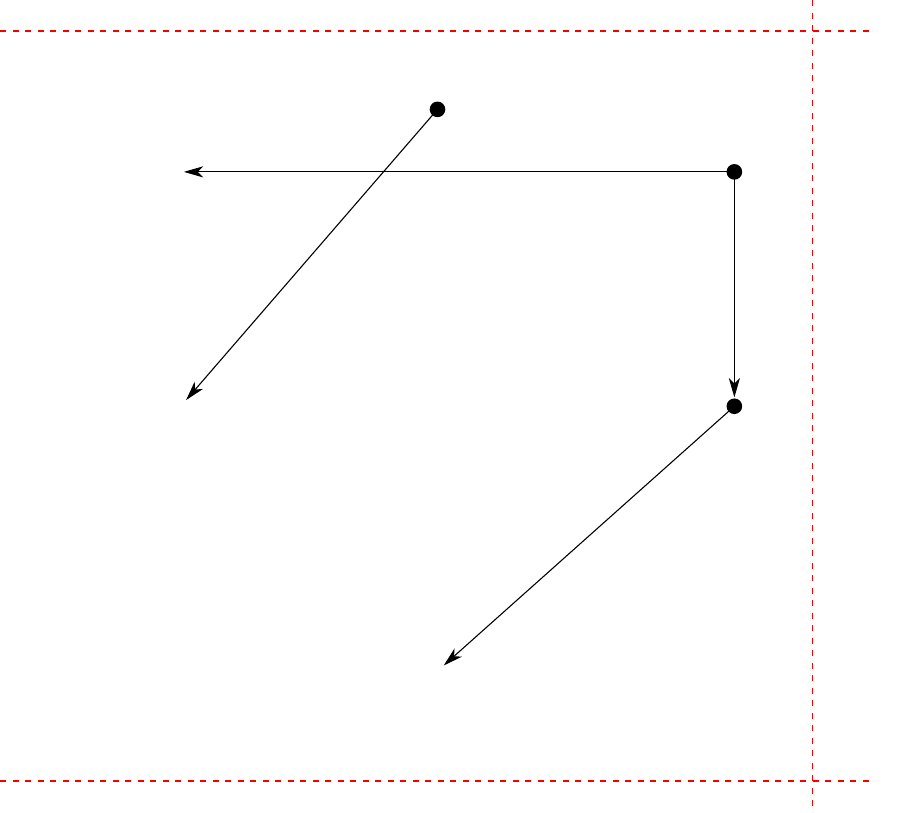
		\quad\quad
		\def\svgwidth{5.8cm}
		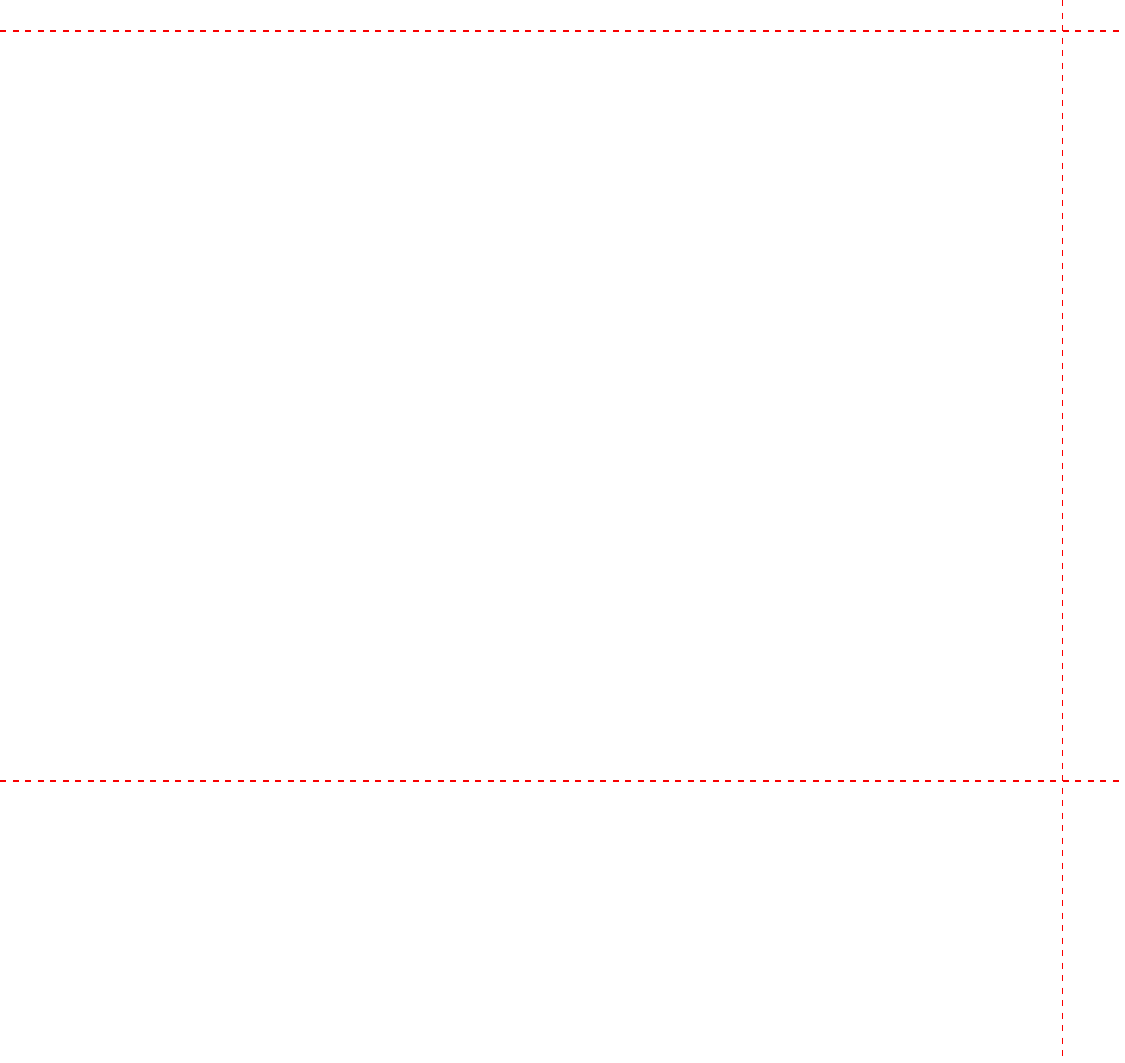
		\caption{
			The chain complex associated with the knots $T_{2,3;2,-1}$ (left) and $-T_{2,3;2,-3}$ (right).
		}\label{fig:TCable-1}
	\end{center}
\end{figure}

\begin{ex} 
	For $K=T_{4,5}\#-T_{2,3;2,5}\#T_{2,3}$, we have
	$\CFLm(K)=C\oplus C'$, where  the chain complex $C$ is illustrated in 
	Figure~\ref{fig:TCable-3} and the homology of $C'$ is freely generated 
	by elements $\{\tb_i\}_{i=1}^m$ such that $\wvar \tb_i=\zvar \tb_i=0$. 
	The homology of $C$ is generated by the classes
	$\z_i=[Z_i]$ for $i=1,2,3,4$, satisfying
	\[\wvar \z_1=\zvar \z_2,\quad \wvar \z_3=\zvar \z_4,\quad
	\wvar\zvar \z_2=\zvar^2 \z_3\quad 
	\text{and}\quad \wvar^2 \z_2=\wvar\zvar \z_3.\]
	In particular, $\tb=\wvar \z_2-\zvar \z_3$ is  torsion, and 
	$\wvar \tb=\zvar \tb=0$.
	We then have a short exact sequence 
	\begin{center}
		\begin{tikzcd}
			0\arrow[r]&\HFmTorsion(K)=\big(\frac{\Ringm}{\langle \wvar,\zvar\rangle}\big)^{m+1}\arrow[r]&
			H_*(C)\arrow[r] &\HFmTF(K)=
			\langle \wvar^3,\wvar^2\zvar,\wvar\zvar^2,\zvar^3\rangle\arrow[r]
			&0,
		\end{tikzcd}
	\end{center}
	which does not split. The generators in $\HFmTorsion(K)$ are killed after multiplication by $\zvar$. Therefore, 
	\begin{align*}
		\HFKmw(K)=\langle \wvar^3,\wvar^2\zvar,\wvar\zvar^2,\zvar^3\rangle_\Ringm\quad\text{and}\quad 
		\HFKhw(K)=\Ringh\oplus\Ringh.	
	\end{align*}	
	In particular, the non-split module $\HFKm(K)$ is reduced to a torsion-free module $\HFKmw(K)$ when we go from ordinary Heegaard-Floer groups to weak Heegaard-Floer groups.
\end{ex}

\subsection{The link cobordism TQFT: some examples}\label{subsec:tqft-examples}
We would now like to compute the cobordism maps in a number of simple examples. We work out $\HFKh$, which is easier to compute.  

\begin{figure}
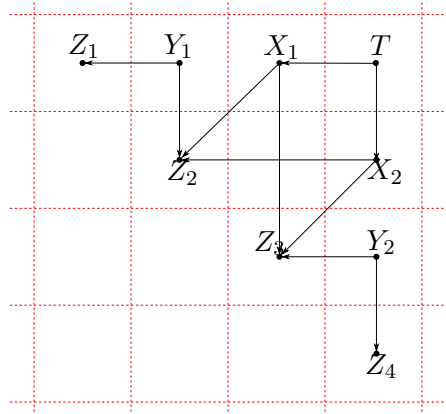

	\begin{center}
		\def\svgwidth{5.8cm}
		
		\caption{
			The chain complex $C$ associated with the knot $T_{4,5}\#-T_{2,3;2,5}\#T_{2,3}$.
		}\label{fig:TCable-3}
	\end{center}	
\end{figure}

\begin{ex}\label{ex:unknot-to-hopf}
Figure~\ref{fig:unknot-to-hopf}(a) illustrates a trisection diagram on the sphere $S^2$ which is given as
\begin{align*}
H=(S^2,\alpha_0,\alpha_1,\alpha_2,\wpoint=\{w_1,w_2\},\zpoint=\{z_1,z_2\})
\end{align*}
 and corresponds to the genus-$0$ cobordism from a $2$-pointed unknot to a Hopf link (with a single marked point on each one of its components). The decoration of the surface is by a pair of arcs connecting the unknot to the components of the Hopf link $L$. Let us denote the aforementioned decorated link cobordism by 
 \begin{align*}
(X,\Surface_0,\Path_0)= (S^3\times[0,1],\Surface_0,\Path_0):(S^3,U,\ppoint_U)\ra (S^3,L,\ppoint_L).	
 \end{align*}	
Some of the intersection points in the  diagram are labeled. In particular, the top intersection point between $\alpha_1$ and $\alpha_2$ is labeled $\theta$.
 It is not hard to see that  	
\begin{align*}
\HFLh(S^3,U,\ppoint_U)=\langle x,y \rangle_\Ringh\quad\text{and}\quad 
\HFLh(S^3,L,\ppoint_L)=\frac{\langle x_{0,0},x_{1,1},x_{0,1}+x_{1,0}\rangle_\Ringh}
{\langle \zvar x_{1,1}\rangle_\Ringh}. 
\end{align*} 	
Moreover, the action of $\Phi=\Phi^\circ_{w_1}=\Phi^\circ_{w_2}$ on $\HFLh(S^3,U,\ppoint_U)$ and 
$\HFLh(S^3,L,\ppoint_L)$ is given by 
\begin{align*}
\Phi(y)=x,\quad\Phi(x)=0,\quad\Phi(x_{1,1})=0,\quad\Phi(x_{0,0})=0\quad\text{and}\quad
\Phi(x_{0,1}+x_{1,0})=x_{1,1}.
\end{align*}
In particular, from the above computations we find 
\begin{align*}
	\HFKh(S^3,U)=\langle x \rangle_\Ringh=\Ringh\quad\text{and}\quad 
	\HFKh(S^3,L)=\frac{\langle x_{0,0},x_{1,1}\rangle_\Ringh}
	{\langle \zvar x_{1,1}\rangle_\Ringh}=\Ringh\llb 0,2\rrb\oplus\Fbb\llb 2,0\rrb. 
\end{align*}
The only triangle class which contributes to  $\gmap_{X,\Surface_0;\Path_0}^\circ(x)$ is the triangle $\Delta\in\pi_2(x,\theta,x_{0,0})$ which is shaded gray in Figure~\ref{fig:unknot-to-hopf}(a). In particular, $\gmap_{X,\Surface_0;\Path_0}^\circ(x)=\x_{0,0}$.	Note that $\gmap_{X,\Surface;\Path_0}^\circ$ shifts the homological bi-grading by $(0,-2)$.\\

In Figure~\ref{fig:unknot-to-hopf}(b), the mirror of the trisection diagram $H$ is presented. It provides a trisection diagram for the link cobordism
\begin{align*}
(-X,-\Surface_0,\Path_0):(S^3,U,\ppoint_U)\ra (S^3,-L,\ppoint_L).	
\end{align*}	  
As in  the above computation, and using the labeling of Figure~\ref{fig:unknot-to-hopf}(b) we then find
\begin{align*}
	\HFKh(S^3,U)=\langle \overline{y} \rangle_\Ringh=\Ringh\quad\text{and}\quad 
\HFKh(S^3,-L)=\frac{\langle \overline{x}_{0,1},\overline{x}_{1,0}\rangle_\Ringh}
{\langle \zvar \overline{x}_{0,1}=\zvar \overline{x}_{1,0}\rangle_\Ringh}=\Ringh\oplus\Fbb. 	
\end{align*}	
The only triangle class which contributes to  $\gmap_{-X,-\Surface_0;\Path_0}^\circ(x)$ is the triangle $\Delta\in\pi_2(\overline{y},\overline{\theta},\overline{x}_{0,1})$ which is shaded gray in Figure~\ref{fig:unknot-to-hopf}(b). In particular, $\gmap_{-X,-\Surface_0;\Path_0}^\circ(\overline{y})=\zvar \overline{x}_{0,1}=\zvar \overline{x}_{1,0}$.	
\end{ex}

\begin{figure}
	\def\svgwidth{0.7\textwidth}
	{\footnotesize{
			\begin{center}
				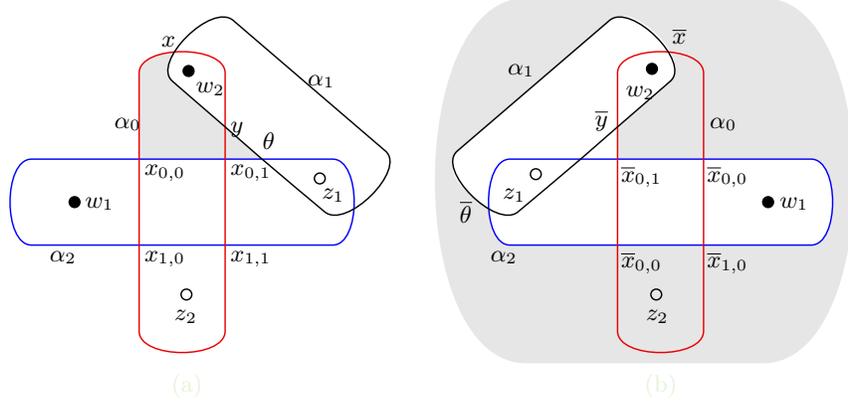
	\end{center}}}
	\caption{A  trisection diagram for a link cobordism of genus $0$ from the unknot to the Hopf link.}
	\label{fig:unknot-to-hopf}
\end{figure}  

\begin{ex}\label{ex:hopf-to-trefoil}
	A trisection diagram for a genus-$0$ cobordism from a $2$-pointed  Hopf link $L$ (with a single marked point on each one of its components) to a $2$-pointed left handed trefoil $K=T_{2,3}$ is illustrated in Figure~\ref{fig:hopf-to-trefoil} (a). The decorated cobordism and the diagram are  denoted by
	\begin{align*}
	(X,\Surface_1,\Path_1):(S^3,L,\ppoint_L)\ra (S^3,K,\ppoint_K)\quad\text{and}\quad	H=(S^2,\alpha_0,\alpha_1,\alpha_2,\wpoint=\{w_1,w_2\},\zpoint=\{z_1,z_2\})
	\end{align*}
respectively. As in the previous example (and compatible with the notation set there), some of the intersection points in the  diagram are labeled. 
We then have
	\begin{align*}
		&\HFLh(S^3,L,\ppoint_L)=\frac{\langle x_{0,0},x_{1,1},x_{0,1}+x_{1,0}\rangle_\Ringh}
		{\langle \zvar x_{1,1}\rangle_\Ringh}\quad\text{and}\\
		&\HFLh(S^3,K,\ppoint_K)=\frac{\langle y_0+y_{1,1},y_1,y_{0,0},y_{0,1}+y_{1,0} \rangle_\Ringh}{\langle \zvar y_{1},\zvar (y_{0}+y_{1,1})\rangle_{\Ringh}}. 
	\end{align*} 	
The action of $\Phi=\Phi^\circ_{w_1}=\Phi^\circ_{w_2}$ on $\HFLh(S^3,L,\ppoint_L)$ is already described in Example~\ref{ex:unknot-to-hopf}, while the aforementioned  action on  $\HFLh(S^3,K,\ppoint_K)$ is given by 
	\begin{align*}
		\Phi(y_1)=y_0+y_{1,1},\quad\Phi(y_{0,1}+y_{1,0})=y_{0,0},\quad\Phi(y_{0,0})=0\quad\text{and}\quad
		\Phi(y_{0}+y_{1,1})=0.
	\end{align*}
In particular, it follows that 
\begin{align*}
	&\HFKh(S^3,L)=\frac{\langle x_{0,0},x_{1,1}\rangle_\Ringh}
	{\langle \zvar x_{1,1}\rangle_\Ringh}=\Ringh\llb0,2\rrb\oplus\Fbb\llb2,0\rrb\quad\text{and}\\
	&\HFKh(S^3,K)=\frac{\langle y_{0,0},y_{0}+y_{1,1} \rangle_\Ringh}{\langle \zvar (y_0+y_{1,1})\rangle_\Ringh}=\Ringh\llb0,2\rrb\oplus\Fbb\llb2,0\rrb. 
\end{align*}
\begin{figure}
	\def\svgwidth{\textwidth}
	{\footnotesize{
			\begin{center}
				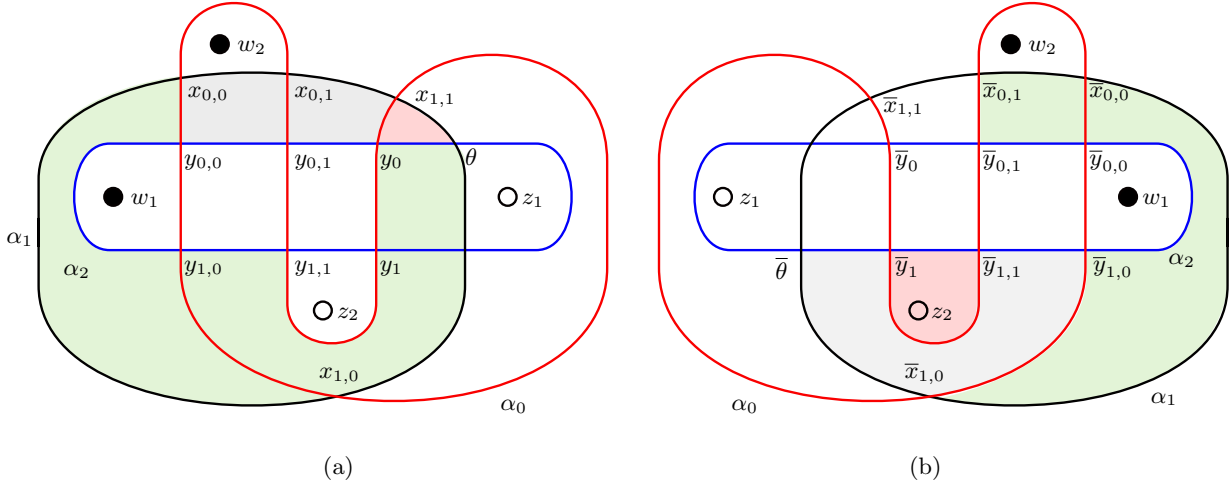
	\end{center}}}
	\caption{A  trisection diagram for a link cobordism of genus $0$ from the Hopf link to the trefoil.}
	\label{fig:hopf-to-trefoil}
\end{figure} 
The only triangle classes which contributes to  $\gmap_{X,\Surface_1;\Path_1}^\circ$ are the triangle classes $\Delta\in\pi_2(x_{1,1},\theta,y_{0})$, $\Delta'\in\pi_2(x_{0,0},\theta,y_{0,0})$ and $\Delta''\in\pi_2(x_{1,1},\theta,y_{1,1})$. The domain $\Dcal$ of $\Delta$ is shaded pink, while the domain $\Dcal'$ is the union of the pink and the gray region in  Figure~\ref{fig:hopf-to-trefoil}. The domain $\Dcal''$ of $\Delta''$ is given by $2$ times the pink region plus the sum of the gray and green domains.
In particular, $\gmap_{X,\Surface_1;\Path_1}^\circ(x_{0,0})=y_{0,0}$ and $\gmap_{X,\Surface_1;\Path_1}^\circ(x_{1,1})=y_0+y_{1,1}$. In other words, $\gmap_{X,\Surface_1;\Path_1}^\circ$ is an isomorphism.\\

As in the previous example, the mirror image of the same diagram (which is illustrated in Figure~\ref{fig:hopf-to-trefoil}(b)) describes a trisection diagram for 	
	\begin{align*}
	(-X,-\Surface_1,-\Path_1):(S^3,-L,\ppoint_L)\ra (S^3,-K,\ppoint_K).
\end{align*}
The Floer homology group $\HFKh(S^3,-L)$ may be computed from the diagram. With the labeling of the intersection points as in Figure~\ref{fig:hopf-to-trefoil}(b), we find
\begin{align*}
	&\HFKh(S^3,-L)=\frac{\langle \overline{x}_{0,1},\overline{x}_{1,0}\rangle_\Ringh}
	{\langle \zvar \overline{x}_{0,1}=\zvar \overline{x}_{1,0} \rangle_\Ringh}=\Ringh\oplus\Fbb\quad\text{and}\\
	&\HFKh(S^3,-K)=\frac{\langle \zvar\overline{y}_{1,0}=\zvar\overline{y}_{0,1},\overline{y}_{0}+\overline{y}_{1,1} \rangle_\Ringh}{\langle \zvar (\overline{y}_0+\overline{y}_{1,1})\rangle_\Ringh}=\Ringh\llb0,2\rrb\oplus\Fbb\llb-1,1\rrb. 
\end{align*}
There are two contributing triangle classes $\Delta_1\in\pi_2(\overline{x}_{1,0},\overline{\theta},\overline{y}_{1,0})$ and $\Delta_2\in\pi_2(\overline{x}_{0,1},\overline{\theta},\overline{y}_{0,1})$ with domains $\Dcal_1$ and $\Dcal_2$ which are obtained as the union of the pink domain containing $z_2$ in Figure~\ref{fig:hopf-to-trefoil}(b) with the gray and the green domain, respectively. 
Therefore, we have
\begin{align*}
\gmap^\circ_{-X,-\Surface_1;\Path_1}(\overline{x}_{1,0})=\gmap^\circ_{-X,-\Surface_1;\Path_1}(\overline{x}_{0,1})
=\zvar	\overline{y}_{1,0}=\zvar \overline{y}_{0,1}.
\end{align*}	 
\end{ex}

\begin{ex}\label{ex:u-to-trefoil}
The composition of the two decorated link cobordisms from Example~\ref{ex:unknot-to-hopf} and Example~\ref{ex:hopf-to-trefoil} is a decorated link cobordism 
\begin{align*}
	(X,\Surface,\Path)=(X,\Surface_1,\Path_1)\circ(X,\Surface_0,\Path_0):(S^3,U,\ppoint_U)\ra (S^3,T_{2,3},\ppoint_K).
\end{align*}	
The corresponding cobordism map is given by $\gmap_{X,\Surface;\Path}^\circ(x)=y_{0,0}$, and  is a homogeneous map of bi-degree $(0,-2)$ which  takes the generator of the group 
$\Ringh=\HFKh(S^3,U)$ to the non-torsion element in $\HFKh(S^3,K)=\Ringh\llb0,2\rrb\oplus\Fbb\llb2,0\rrb$. 	Therefore, $\gmap_{X,\Surface}^\circ:\Ringh\llb0,2\rrb\ra\Ringh\llb0,4\rrb$ is an isomorphism which drops the bi-degree by $(2,0)$. On the other hand, we have
\begin{align*}
	\gmap^\circ_{-X,-\Surface;\Path}(\overline{y})=\zvar^2\overline{y}_{1,0}=\zvar^2\overline{y}_{0,1}.
\end{align*}
In other words, the image of the generator of $\HFKh(S^3,U)=\Ringh$ under $\gmap^\circ_{-X,-\Surface;\Path}$ is $\zvar$ times the non-torsion generator of $\HFKh(S^3,-T_{2,3})=\Ringh\llb0,2\rrb\oplus\Ringhh\llb-1,1\rrb$. Therefore, $\gmap_{-X,-\Surface}^\circ:\Ringh\llb0,2\rrb\ra\Ringh\llb0,4\rrb$ is identified as $\zvar\cdot Id$. 	
\end{ex}

The next example considers the construction of Juh\'asz and Zemke from \cite{JZ-slice-disks} of different slice disks for a knot of the form $K\#-K$. Although we only consider the case where $K$ is the figure-eight knot, most of the argument remains valid for roll-spinning along arbitrary knots. Example~\ref{ex:slice-disks-for-F8} also implies that $\HFKh$ is an invariant of smooth isotopy, rather than topological isotopy.

\begin{ex}\label{ex:slice-disks-for-F8}
	Let $K$ denote the figure-eight knot and $f:(S^3,K)\ra (S^3,K)$  denote the diffeomorphism which fixes a point $p$ on $K$ and gives rise to roll-spinning. Let  $D,D'=D_{f,p}\subset B^4$ denote the properly embedded disks with boundary $-K\#K$ which are determined by the identity and the map $f$, respectively, as discussed in \cite{JZ-slice-disks}. Note that the diffeomorphism types of the pairs $(B^4,D)$ and $(B^4,D')$ are the same by \cite[Proposition 3.2]{JZ-slice-disks}.
	We may  assume that  the knot chain complex $C=C_K$ for $K$ is $C=\langle x,x_0,x_1,y_0,y_1\rangle_{\Ringm}$ and that the differential $d$ of $C$ is given by 
	\begin{align*}
		d(x)=0,\quad d(x_0)=0,\quad d(x_1)=\zvar x_0,\quad d(y_0)=\wvar x_0\quad\text{and}\quad
		d(y_1)=\zvar y_0+\wvar x_1.	
	\end{align*}	   
	The dual complex $(C^*,d^*)$ is then generated by $x^*,x_0^*,x_1^*,y_0^*$ and $y_1^*$, while we have
	\begin{align*}
		d^*(x^*)=0,\quad d^*(y_1^*)=0,\quad d^*(x^*_1)=\wvar y_1^*,\quad d^*(y^*_0)=\zvar y_1^*\quad\text{and}\quad
		d^*(x_0^*)=\wvar y_0^*+\zvar x_1^*.	
	\end{align*}
	The chain map $\fmap$ associated with $f$ is given by the identity on $x,x_0,x_1,y_0$, while $\fmap(y_1)=y_1+x_0$. The classes  in $C\otimes C^*$ associated with $D$ and $D'$ are then determined by  
	\begin{align*}
		\tmap=x\otimes x^*+x_0\otimes x_0^*+x_1\otimes x_1^*+y_0\otimes y_0^*+y_1\otimes y_1^*\quad\text{and}\quad\tmap'=\tmap+x_0\otimes y_1^*. 	
	\end{align*}	
	Note that the class represented by $\tmap'-\tmap$ in $\HFKs(S^3,-K\#K)$ is non-trivial for $\star=-,\circ,\wedge$. Therefore, 
	\begin{align*}\tmap^\star_{D,-K\#K}\neq  \tmap^\star_{D',-K\#K},\quad\text{for}\ \ \star=-,\circ,\wedge.
	\end{align*}	
In particular, the slice disks $D$ and $D'$ are different up isotopy. This result should be compared with \cite[Theorem~5.4]{JZ-slice-disks}.
\end{ex}	

\subsection{Elementary cobordisms}\label{subsec:elementary-cobordisms}
A union  $\Surface\subset Y\times [0,1]$ of smooth annuli  from a link $(Y,L)$ to another link $(Y,L')$   with 
\begin{align*}
	\partial \Surface\cap Y\times\{0\}=-L\quad\text{and}\quad  \partial \Surface\cap Y\times\{1\}=L'
\end{align*}
gives a {{concordance}} $(Y\times[0,1],\Surface):(Y,L)\ra (Y,L')$. A concordance as above is called a {\emph{ribbon concordance}} if the projection
$h:Y\times [0,1]\ra [0,1]$ restricts to a Morse function on $\Surface$ which does not have any critical points of index $2$. 
Let us write $(Y,L)\leq (Y,L')$ if a ribbon concordance from $(Y,L)$ to $(Y,L')$ exists. The reverse of a ribbon concordance is called a {\emph{co-ribbon}} concordance. Ribbon concordances were studied by Gordon \cite{Gordon-Ribbon}. His conjecture, that $\leq$ is a partial order among the knots in $S^3$, is proved by Agol \cite{Agol-Ribbon}.
Zemke showed that knot Floer homology provides an obstruction to ribbon concordance \cite{Zemke-Ribbon}.  The following theorem follows as a byproduct of Zemke's result and our construction.

\begin{thm}\label{thm:ribbon-concordance-intro}
	With $X=Y\times [0,1]$, let $(X,\Surface):(Y,L)\ra (Y,L')$ be a ribbon concordance and  $(X,\overline\Surface):(Y,L')\ra (Y,L)$ be the reverse concordance (which is co-ribbon). Then for $\star=-,\circ,\wedge$
	\begin{align*}
		\gmap_{X,\overline\Surface}^\star\circ	\gmap_{X,\Surface}^\star=Id:
		\HFKs(Y,L)\ra  \HFKs(Y,L).
	\end{align*}		
	In particular, $\gmap_{Y\times [0,1],\Surface}^\star$ is injective (respectively, surjective) for every ribbon (respectively, co-ribbon) concordance $(Y\times[0,1],\Surface)$.
\end{thm}	

More generally, let $K$ and $K'$ be a pair of knots in a $3$-manifold $Y$, $X=Y\times [0,1]$, and $(X,\Surface):(Y,K)\ra (Y,K')$
denote a link cobordism.  The projection map from $X$ to $[0,1]$ may be perturbed to a function $h:X\ra [0,1]$ without critical points on $X$, such that $h|_{\Surface}:\Surface\ra [0,1]$ is Morse.  One may  modify $h:X\ra \R$ without creating critical points in $X$ so that $h|_{\Surface}$ has $m=2a+2b+k$ critical values 
\[0<s_1<s_2<\cdots<s_m<1,\] 
and associated with $s_j$  there is a unique critical point $\wpt_j\in \Surface$  for $h|_{\Surface}$, of index $0$ for $j=1,\ldots,a$, index $1$ for $j=a+1,\ldots,m-b$ and index $2$ for $j=m-b+1,\ldots,m$. Fix $\{r_j\}_{j=0}^m$ so that
\[0=r_0<s_1<r_1<s_2<r_2<\cdots<r_{m_1}<s_m<r_m=1.\] 
We may further change the order of the critical points 
$\wpt_{a+1},\ldots,\wpt_{m-b}$ arbitrarily, and after such a change of order, one may assume that that the pre-image of $[0,r_{2a}]$ gives a ribbon concordance
\[(X=Y\times[0,1],\Surface_{ribb}):(Y,K)\ra (Y,L),\] 
while the pre-image of $[r_{m-2b},1]$ gives the  co-ribbon concordance
\[(X=Y\times[0,1],\Surface_{co-ribb}):(Y,L')\ra (Y,K'),\]
for knots $L=L_0$ and $L'=L_k$ in $Y$. 
Correspondingly, $(X,\Surface)$ breaks down as the composition of a ribbon concordance, 
the band surgeries determined by the critical values $s_{a+1},\ldots,s_{m-a}$, and a co-ribbon concordance. More precisely, 
\[(X,\Surface)=(X,\Surface_{co-ribb})\circ(X,\Surface_k)\circ \cdots \circ  (X,\Surface_1)\circ (X,\Surface_{ribb}),\]
where $(X,\Surface_j):(Y,L_{j-1})\ra (Y,L_{j})$ is determined as the pri-image of $[r_{2a+j-1},r_{2a+j}]$ for every $j=1,\ldots,k$.  
\begin{defn}\label{defn:elementary-cobordism}
	A link cobordism $(X,\Surface):(Y_2,L_2)\ra (Y_1,L_1)$ is called an {\emph{elementary link cobordism}}  if there is a 	function $h:X\ra [0,1]$ without critical points on $X$, such that $h|_{\Surface}:\Surface\ra [0,1]$ is a Morse function with a single critical point $z$ of index $1$. 
\end{defn}	

In the setup of Definition~\ref{defn:elementary-cobordism}, it  follows that $Y_1$ and $Y_2$ are diffeomorphic to the same closed $3$-manifold $Y$ and that $X$ is diffeomorphic to $Y\times [0,1]$. The discussion before Definition~\ref{defn:elementary-cobordism}  implies that each $(X,\Surface_j):(Y,L_{j-1})\ra (Y,L_{j})$ is an elementary cobordism. If 
$\Surface$ is a surface of genus $g$, we find that $k=2g$ and  
\begin{align*}
	\gmap_{X,\Surface}^\star=\gmap_{X,\Surface_{co-ribb}}^\star\circ
	\gmap_{X,\Surface_k}^\star\circ\cdots\circ \gmap_{X,\Surface_1}^\star
	\circ \gmap_{X,\Surface_{ribb}}^\star,\quad\forall\ \ \star\in\{-,\circ\}.
\end{align*}	  

Theorem~\ref{thm:ribbon-concordance} gives some useful information about   $\gmap_{X,\Surface_{ribb}}^\spinc$ and $\gmap_{X,\Surface_{co-ribb}}^\spinc$. On the other hand, \cite[Lemma 2.1]{AE-unknotting} may be restated as the following proposition.

\begin{prop}\label{prop:reversing-elementary-cobordism}
	Let $(X,\Surface):(Y,L)\ra (Y,L')$ be an elementary cobordism. Denote the reverse cobordism (which is also elementary) by $(X,\oS):(Y,L')\ra (Y,L)$. Then the maps
	\begin{align*}
		&\gmap_{X,\oS}^\star\circ 	\gmap_{X,\Surface}^\star:\HFKsw(Y,L)\ra \HFKsw(Y,L)\quad\text{and}\quad
		\gmap_{X,\Surface}^\star\circ 	\gmap_{X,\oS}^\star:\HFKsw(Y,L')\ra \HFKsw(Y,L')
	\end{align*}	
	are both identified as $\zvar\cdot Id$ for $\star=-,\circ$. In particular, both maps are zero for $\star=\circ$.
\end{prop}

Proposition~\ref{prop:reversing-elementary-cobordism} has an interesting consequence. We 
refer the reader to Definition~\ref{defn:compression-disk} for the definition of a compression disk for a link cobordism.

\begin{thm}\label{them:compression-along-disk}
	Suppose that the link cobordism $(X,\Surface_D)$ is obtained from the link cobordism $(X,\Surface)$ by compressing along a compression disk $D$. Then $	\gmap_{X,\Surface}^\star=\zvar\gmap_{X,\Surface_D}^\star$ for $\star=-,\circ$.	
\end{thm}
\begin{proof}
	Identify $D$ with the standard disk $D^2$ and let $h$ denote the height function on $D^2$. Then $h$ extends to a Morse function on $W=D\times [-1,1]^2=D^2\times [-1,1]^2$ by setting 
	\begin{align*}
		h(x,t):=h(x)+\epsilon\langle x,t\rangle =x_2+\epsilon t_1x_1+\epsilon t_2x_2,\quad\forall\ (x,t)=(x_1,x_2,t_1,t_2)\in D^2\times [-1,1]^2,
	\end{align*}
	for some $\epsilon\in\R^{>0}$. Note that $\nabla h=((0,1)+\epsilon t,\epsilon x)$ is always non-zero. On $\partial D\times [-1,1]\times\{0\}$, the gradient of the restriction of $h$ is zero if $x=(0,\pm 1)$ and $t=0$. In other words, the restriction of $h$ to $W\cap\Surface$ has precisely two critical points of index $1$.  The function $h:W\ra \R$, which does not have any critical points, may be extended to a Morse function $h:X\ra \R$ so that all critical values of $h$ are in $\R-[-1,1]$ and all critical values of the restriction $h:\Surface-W\ra \R$ (which may also be assumed to be Morse) are also outside $\R-[-1,1]$.
	By considering the pre-images of $(-\infty,-1-\epsilon]$, $[-1-\epsilon,1+\epsilon]$ and 
	$[1+\epsilon,\infty)$ under $h$, we then obtain a decomposition of $(X,\Surface)$ and $(X,\Surface_D)$ as  
	\begin{align*}
		&(X,\Surface)=(X_2,\Surface_2)\circ (Y''\times [0,1],\Surface'')\circ (X_1,\Surface_1),\quad\text{and}\quad
		(X,\Surface)=(X_2,\Surface_2)\circ  (X_1,\Surface_1).
	\end{align*}	
	The surface $\Surface''$ is a (possibly disconnected) surface of genus $1$.
	The crutical observation is that 
	\begin{align*}
		(Y''\times [0,1],\Surface'')=(Y''\times [0,1],\overline{\Surface'})\circ (Y''\times [0,1],{\Surface'})
	\end{align*}
	where the elementary cobordism $(X''=Y''\times [0,1],{\Surface'})$ is  the pre-image of $[-1-\epsilon,0]$ under $h$. In other words,  compression along the dik $D$ may be realized as the composition of an elementary cobordism and its reverse, see Figure~\ref{fig:reversing-elementary-cobordisms} for an illustration. Therefore, for $\star=-,\circ$
	\begin{figure}
		\def\svgwidth{0.85\textwidth}
		{\footnotesize{
				\begin{center}
					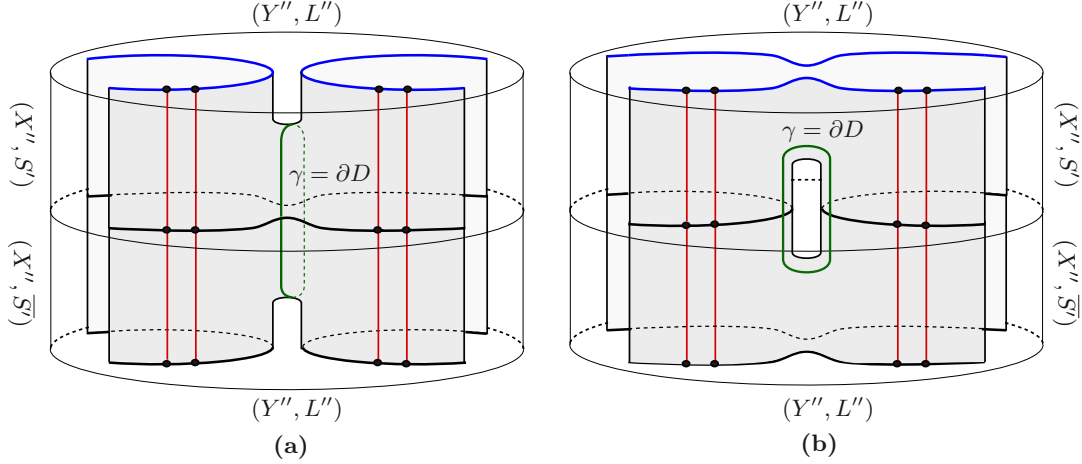
		\end{center}}}
		\caption{Compression along a disk $D$ may be described as the composition of an elementary cobordism with the reverse of that cobordism, as illustrated in parts (a) and (b).}
		\label{fig:reversing-elementary-cobordisms}
	\end{figure} 
	\begin{align*}
		\gmap_{X,\Surface,\spinct}^\star&=
		\gmap_{X_1,\Surface_1,\spinct|_{X_1}}^\star\circ \gmap_{Y''\times [0,1],\Surface'',\spinct|_{Y''}}^\star\circ 	\gmap_{X_2,\Surface_2,\spinct|_{X_1}}^\star\\
		&=
		\gmap_{X_1,\Surface_1,\spinct|_{X_1}}^\star\circ \gmap_{Y''\times [0,1],\Surface',\spinct|_{Y''}}^\star\circ \gmap_{Y''\times [0,1],\overline{\Surface'},\spinct|_{Y''}}^\star\circ 	\gmap_{X_2,\Surface_2,\spinct|_{X_1}}^\star\\
		&=\zvar\gmap_{X_1,\Surface_1,\spinct|_{X_1}}^\star\circ 	\gmap_{X_2,\Surface_2,\spinct|_{X_1}}^\star=\zvar\gmap_{X,\Surface_D,\spinct}^\star,\quad\quad\quad\quad\quad\quad\forall\ \ \spinct\in\SpinC(X).
	\end{align*}	
	The third equality above follows from Proposition~\ref{prop:reversing-elementary-cobordism} for the elementary cobordism $(Y''\times [0,1],{\Surface'})$. The above computation completes the proof.
\end{proof}

\subsection{General merge and split cobordisms}\label{subsec:general-merge-split}
In Section~\ref{subsec:merge-split}, we considered the merge cobordism, where a trivial (unknotted) link component merges with a component of a link $(Y,L)$. The result is an elementary cobordism in $Y\times[0,1]$ which consists of several trivial cylinders and a distinguished component, which is a sphere with three disks removed. The split cobordism is the reverse of the aforementioned elementary cobordism. \\

More generally, let $(Y^j,L^j)$ be a link  and $p^j\in L^j$ be a point on a distinguished component of $L^j$ for $j=1,2$. Remove a ball neighborhood $B^j$  of $p^j$ and glue the $S^2$-boundaries of the resulting $3$-manifolds, to obtain the manifold $Y=Y^1\#Y^2$. If the gluing is done appropriately, so that the two boundary components in $L^1-B^1$ are identified with the two boundary components of $L^2-B^2$, we obtain a link $L=L^1\#_{p^1,p^2} L^2$ in $Y$. The orientations of $L^1$ and $L^2$ induce an orientation on $L$ (provided that the gluing is done appropriately). If $p^j$ is replaced by two nearby points in $Y^j-L^j$, the construction gives the split union $L'=L^1\amalg L^2$ in $Y$. Note that  $\HFKs(Y,L')\simeq \HFKs(Y,L)$. Two simple cobordisms connecting the links $(Y,L')$ and $(Y,L)$ are hosted by  $X=Y\times [0,1]$, which are generalizations of the merge and split cobordisms, and are denoted by $(X,\Surface_m):(Y,L)\ra (Y,L')$ and $(X,\Surface_s):(Y,L)\ra (Y,L')$, respectively.  \\

\begin{figure}
	\def\svgwidth{0.9\textwidth}
	{\footnotesize{
			\begin{center}
				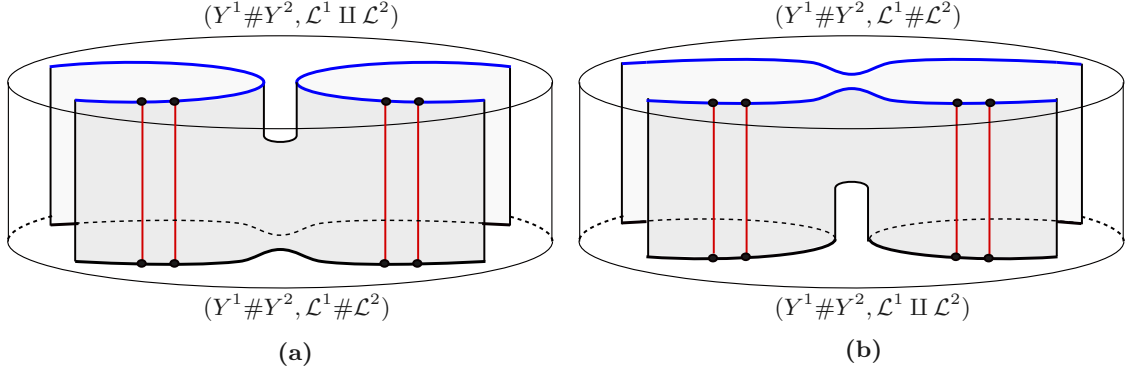
	\end{center}}}
	\caption{(a)  The general merge cobordism connects the link $(Y^1\#Y^2,\Link^1\amalg\Link^2)$ to  $(Y^1\#Y^2,\Link^1\#\Link^2)$. (b) The reverse of a general merge cobordism is a general split cobordism.}
	\label{fig:Merge-Split-Cobordisms}
\end{figure}

Let $\Link^j=(L^j,\wpoint^j,\zpoint^j)$ be a pointed version of the link $L^j$, and assume that $p^j$ is not in $\wpoint^j\amalg\zpoint^j$. Denote the closest points of $\wpoint^j$ and $\zpoint^j$ to $p^j$  by $z^j$ and $w^j$, respectively. The subsets $\{\wpoint^1\amalg\wpoint^2\}\times[0,1]$ of $\Surface_m$ and $\Surface_s$ are decorations of the latter  link cobordisms, giving the decorated cobordisms $(X,\SCob_m)$ and $(X,\SCob_s)$, respectively, as illustrated in Figure~\ref{fig:Merge-Split-Cobordisms}(a,b). 

\begin{thm}\label{thm:merge-split-general}
	Let $(X,\Surface_m):(Y,L)\ra 	(Y,L')$ be a general merge cobordism, with $Y=Y^1\#Y^2$, $L=L^1\#_{p^1,p^2}L^2$ and $L'=L^1\amalg L^2$ as above, and  $(X,\Surface_s):(Y,L')\ra 	(Y,L)$ be the corresponding general split cobordism. Then for $\star=-,\circ$ we have
	\begin{equation}\label{eq:Merge-Split-Formula}
		\begin{split}
			&\gmap_{X,\Surface_m}^{\star}= Id:
			\HFKsw(Y,L')\ra \HFKsw(Y,L)\quad  \text{and}\quad
			\gmap_{X,\Surface_s}^{\star}=\zvar\cdot Id:
			\HFKsw(Y,L)
			\ra \HFKsw(Y,L').
		\end{split}
	\end{equation}	
\end{thm}	
\begin{proof}
	For $j=1,2$, let  $H^j=(\HSurf^j,\alphas_0^j,\alphas_1^j;\wpoint^j,\zpoint^j)$ denote a  pointed Heegaard diagram for $(Y^j,\Link^j)$. Assume that $w^j\in\wpoint^j$ and $z^j\in\zpoint^j$ are in the same component of $\HSurf^j-\alphas_1^j$, and that the connected sum takes place near  the basepoint $z^j$. Further assume that  there is a unique curve $\alpha_0^j\in\alphas_0^j$ which separates $w^j$ from $z^j$. The neighborhood of $w^j$ and $z^j$ in $H^j$ is illustrated in Figure~\ref{fig:Merge-Split-Cobordisms-B}(a). \\

	A trisection diagram for $(X,\SCob_m)$ is then constructed as follows. We remove a disk neighborhood of $z^j$ in $\HSurf^j$, and connect the circle boundaries of the resulting surfaces by a $1$-handle to form a closed surface $\HSurf=\HSurf^1\#\HSurf^2$. Let $\alpha_0,\alpha_1$ and $\alpha_2$ denote three simple closed curves which are Hamiltonian isotopes of the belt of the aforementioned $1$-handle. Therefore, each pair of these three curves intersect each-other in a pair of transverse intersection points. We then set $\alphas_i=\alphas_i^1\amalg\alphas_i^2\amalg\{\alpha_i\}$ for $i\in\Z/3$, where $\alphas_2^j$ consists of Hamiltonian isotopes of the curves in $\alphas_1^j$ (for $j=1,2$). This is illustrated in Figure~\ref{fig:Merge-Split-Cobordisms}(b). A pair of basepoints in two of the $6$ small traingles bounded between $\alpha_0,\alpha_1$ and $\alpha_2$ may be labeled $z^1$ and $z^2$ (to replace the missing basepoints with the same label from $\zpoint^1$ and $\zpoint^2$), as illustrated in Figure~\ref{fig:Merge-Split-Cobordisms}(b). The result is a trisection diagram
	\begin{align*}
		H=(\HSurf,\alphas_0,\alphas_1,\alphas_2,\wpoint=\wpoint^1\amalg\wpoint^2,\zpoint=\zpoint^1\amalg\zpoint^2),	
	\end{align*}
	which represents $(X,\SCob_m)$.	We fix the $\SpinC$ structures $\spinc^j$ on $Y^j$ and $\spinc=\spinc^1\#\spinc^2$ on $Y$, and drop them from the notation for simplicity. We may further assume that $H$ is $\spinc$-admissible. There are $6$ intersection points on the $1$-handle connecting $\HSurf^1$ and $\HSurf^2$ between $\alpha_0,\alpha_1$ and $\alpha_2$. The intersection points between $\alpha_{i+1}$ and $\alpha_{i-1}$ are denoted by $\theta_i^+$ and $\theta_i^-$, while we assume that the two bigons connecting them are the domains of  Whitney disks from $\theta^+_i$ to $\theta^-_i$ (i.e. that $\theta^+_i$ is the top intersection point, see Figure~\ref{fig:Merge-Split-Cobordisms}(d)). \\

	\begin{figure}
		\def\svgwidth{0.95\textwidth}
		{\footnotesize{
				\begin{center}
					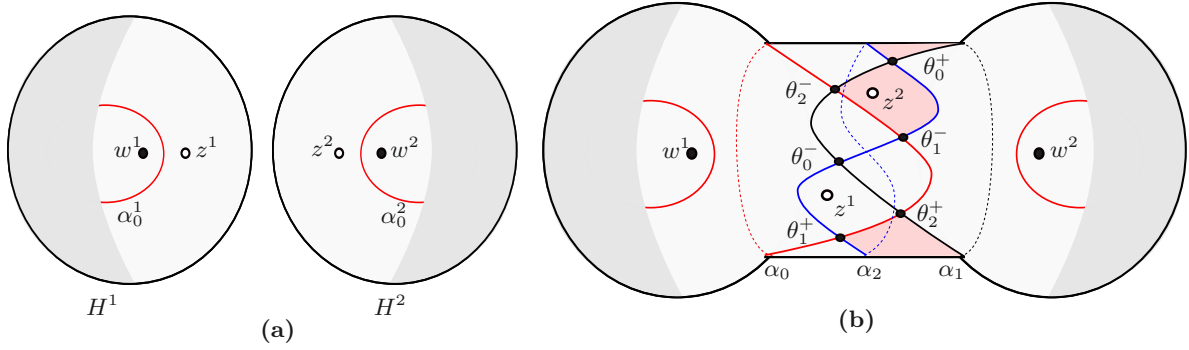
		\end{center}}}
		\caption{A pair of diagrams $H^1$ and $H^2$ for the links $(Y^1,\Link^1)$ and $(Y^2,\Link^2)$, as in (a), give a trisection diagram for the general merge cobordism, illustrated in (b).}
		\label{fig:Merge-Split-Cobordisms-B}
	\end{figure} 
	
	Abusing the notation, we use $H^j$ also to denote $(\HSurf^j,\alphas_0^j,\alphas_1^j,\alphas_2^j,\wpoint^j,\zpoint^j)$, corresponding to the  cobordism $(Y^j,\Link^j)\times [0,1]$. Therefore, the chain map $f_{H^j}$ induces the identity map of $\HFLs(Y^j,\Link^j)$. The set $\Sbf(H_i)$ of generators of the chain complex $C(H_i)$  is decomposed as 
	\[\Sbf(H_i)=\Sbf(H^1_i)\times\Sbf(H^2_i)\times\{\theta_i^+,\theta_i^-\},\quad\forall\ i\in\Z/3.\]
	 The diagram $H$ is of the type studied in \cite[Theorem 7.3]{AE-2} (or \cite[Proposition 5.3]{Zemke-1}), except for the positioning of the basepoints $z^1$ and $z^2$. We may thus describe the chain map $f_{H,\spinc}$ (associated with $H$ and $\spinc$), provided that the path of almost complex structures is sufficiently stretched along the two attaching circles of the $1$-handle containing $\alpha_0,\alpha_1$ and $\alpha_2$, by
	\begin{equation}\label{eq:merge-map}
		\begin{split}
			&f_{H,\spinc}(\x^1\times \x^2\times \theta_2^+)= f_{H^1,\spinc^1}(\x^1)\otimes 	f_{H^2,\spinc^2}(\x^2)\otimes \theta_1^+\quad\text{and}\\
			&f_{H,\spinc}(\x^1\times \x^2\times \theta_2^-)=\zvar\cdot f_{H^1,\spinc^1}(\x^1)\otimes 	f_{H^2,\spinc^2}(\x^2)\otimes\theta_1^-+F_\spinc(\x_1,\x_2)\otimes\theta_2^+,
		\end{split}
	\end{equation}	 
	for some map $F_\spinc:C(H^1_2)\otimes C(H^2_2)\ra C(H^1_1)\otimes C(H^2_1)$. 
	The domain of every contributing triangle for $f_{H,\spinc}$ includes one of the two pink triangles in Figure~\ref{fig:Merge-Split-Cobordisms}(b).
	Under the assumption on the path of almost complex structures (that it is sufficiently stretched), the disks contributing to the differential of $C(H_i)$ (for $i=1,2$) may be described using \cite[Proposition 5.1]{AE-2}. In particular, 
	\begin{align*}
		&C(H_1)=C(H^1_1)\otimes_\Ringm C(H^2_1)\otimes_\Ringm\big(\Ringm\oplus \Ringm\llb1,1\rrb\big) \quad\text{and}\\
		&C(H_2)=C(H^1_2)\otimes_\Ringm C(H^2_2)\otimes_\Ringm\big(\Ringm\oplus \Ringm\llb1,0\rrb\big),
	\end{align*}
	where the tensor product with $\Ringm\oplus\Ringm\llb1,1\rrb$ (or with $\Ringm\oplus\Ringm\llb1,0\rrb$) on the right-hand-side corresponds to the choice of $\theta_i^+$ or $\theta_i^-$. In particular, the generators of $\HFKs(H_i)\subset \HFLs(H_i)$, which are in the kernel of the action of $\wedgews$, use $\theta^+_i$ (and not $\theta^-_i$). Therefore, the first formula in (\ref{eq:merge-map}) implies that the induced cobordism map 
	\begin{align*}
		\gmap_{X,\Surface_m,\spinc}^\star=\gmap_{H,\spinc}^\star:\HFKsw(Y,L,\spinc)
		\ra \HFKsw(Y,L',\spinc)=\HFKsw(Y,L,\spinc)	
	\end{align*}		 
	is the natural identity map. This completes the proof of the first claim  in (\ref{eq:Merge-Split-Formula}). The proof of the second claim is completely similar.
\end{proof}

\begin{ex}\label{ex:RTR}
	Let $K$ and $K'$ denote a pair of knots in $S^3$ which correspond to knot diagrams which are identical outside a small disk $D$, where the two diagrams differ by a {\emph{proper rational replacement}} (i.e. a rational replacement which does not change the pair of points on the boundary of $D$ which are connected to one another). The rational replacement determines a $(1,1)$ link $R$ with two components and the link cobordism 
	\begin{align*}
		&(S^3\times [0,1],\Surface_1): (S^3,K\amalg R\amalg -R)\ra (S^3,K'\amalg -R)\quad\text{and}\\
		&(S^3\times [0,1], \Surface_2): (S^3,K'\amalg -R)\ra (S^3,K),	
	\end{align*}	
	of genus zero.  The composition of the above two link cobordisms is a link cobordism $(S^3\times [0,1],\Surface)$ which decomposes as the composition of a merge cobordism  and $(S^3\times [0,1],\Surface_m)$ with the split union of a cobordism 
		\begin{align*}
		&(S^3\times [0,1], K\times [0,1]\amalg \Surface_3): (S^3, K\amalg R\amalg -R)\ra (S^3,K\amalg U). 	
	\end{align*}	
Correspondingly, Theorem~\ref{thm:merge-split-general} and the composition law for TQFTs imply that  associated with every $\y^+\in\HFKs(R),\y^-\in\HFKs(-R)$ we obtain the maps
\begin{align*}
\gmap_{\y^+}:\HFKs(K)\ra \HFKs(K')\quad\text{and}\quad
\gmap_{\y^-}:\HFKs(K')\ra \HFKs(K), 	
\end{align*}	
so that $\gmap_{\y^-}\circ\gmap_{\y^+}=\afrak(\y^-,\y^+)\cdot Id$ for some $\afrak(\y^-,\y^+)\in\Rings$.  The main observation in \cite{Ef-RTR} is that for $\star=\circ$, we may choose $\y^-$ and $\y^+$ so that $\afrak(\y^-,\y^+)=\zvar$.
\end{ex}	

\bibliographystyle{hamsalpha}

\end{document}